\documentclass[12pt]{iopart}
\usepackage{iopams,graphicx,subfigure,latexsym,amsthm}
\usepackage[a4paper,twoside,bindingoffset=1cm,text={16.5cm,24cm},includehead,includefoot]{geometry}
\usepackage{siunitx}
\usepackage{setspace}
\usepackage{url}
 \usepackage{pstricks}
 \usepackage{pst-grad} 
 \usepackage{pst-plot} 
 \usepackage[space]{grffile} 
 \usepackage{etoolbox} 
 \makeatletter 
 \patchcmd\Gread@eps{\@inputcheck#1 }{\@inputcheck"#1"\relax}{}{}
 \makeatother
\newcommand{\f}{\frac}
\newcommand{\pa}{\partial}

\newcommand{\lam}{\lambda}

\newcommand{\rb}{\mathbb{R}}

\newcommand{\alp}{\alpha}
\newcommand{\lap}{\Delta}

\newcommand{\bke}[1]{\left( #1 \right)}
\newcommand{\bkc}[1]{\langle #1 \rangle}

\newcommand{\norm}[1]{\left\Vert #1 \right\Vert}
\newcommand{\abs}[1]{\left| #1 \right|}

\newcommand{\fB}{\mathbf{B}}

\newcommand{\fd}{\mathbf{d}}

\newcommand{\fW}{\mathbf{W}}
\newcommand{\fU}{\mathbf{U}}
\newcommand{\fr}{\mathbf{r}}

\newcommand{\fz}{\mathbf{z}}
\newcommand{\fx}{\mathbf{x}}

\newcommand{\eps}{\varepsilon}
\renewcommand{\vec}[1]{\mathbf{#1}}

\newcommand{\cmmnt}[1]{\ignorespaces}

\theoremstyle{plain}
\newtheorem{theorem}{Theorem}[section]
\newtheorem{lemma}[theorem]{Lemma}

\theoremstyle{remark}
\newtheorem{remark}{Remark}[section]
\newtheorem{example}{Example}[section]

\begin{document}
\title[DSM for imaging small dielectric inhomogeneities: analysis and improvement]{\begin{center}Direct sampling method for imaging small dielectric inhomogeneities: analysis and improvement\end{center}}

\author{\begin{center}Sangwoo Kang$^{1}$, Marc Lambert$^{1}$, and Won-Kwang Park$^{2}$\end{center}}
\address{\begin{center}$^{1}$Group of Electrical Engineering, Paris (GeePs),UMR CNRS 8507, CentraleSup\'elec, Univ. Paris Sud, Universit\'e. Paris Saclay, UPMC Univ. Paris 06, 3 \& 11 rue Joliot-Curie 91192, Gif-sur-Yvette, France\end{center}}

\address{\begin{center}$^{2}$Department of Information Security, Cryptology, and Mathematics, Kookmin University, Seoul, 02707, Korea\end{center}}

\ead{\begin{center}sangwoo.kang@geeps.centralesupelec.fr, marc.lambert@geeps.centralesupelec.fr, and parkwk@kookmin.ac.kr\end{center}}


\begin{abstract}
The direct sampling method (DSM) has been introduced  for non-iterative imaging of small inhomogeneities and is known to be fast, robust, and effective for inverse scattering problems. However, to the best of our knowledge, a  full analysis of the behavior of the DSM has not been  provided yet. Such an analysis is  proposed here within the framework of the asymptotic hypothesis in the 2D case leading to the expression of the DSM indicator function in terms of the Bessel function of order zero and the sizes, shapes and permittivities of the inhomogeneities. Thanks to this analytical expression the limitations of the DSM method when one of the inhomogeneities is smaller and/or has lower permittivity than the others is exhibited and illustrated. An improved DSM is proposed to overcome this intrinsic limitation in the case of multiple incident waves. Then we show that both the traditional and improved DSM are closely related to a normalized version of the Kirchhoff migration. The theoretical elements of our proposal are supported by various results from numerical simulations with synthetic and experimental data.
\end{abstract}

\section{Introduction}
The non-invasive and non-destructive reconstruction of location and shape of unknown targets is a popular research subject since it can be applied to various problems like, and without exhaustivity, identifying defects in bridges and concrete walls \cite{FFK,PGBM,VS}, through-wall imaging \cite{B6,C3,KKN}, non-destructive testing \cite{AD,HLD,NDT_review}, and biomedical imaging for detecting breast cancer \cite{HSM1,IMD,S2}. Unfortunately, because of the intrinsic difficulties related to its ill-posedness and nonlinearity, such inverse scattering problems are difficult to solve. Various reconstruction algorithms have been investigated to overcome those difficulties. The main approach of such algorithms is based on the least-squares method and Newton-type iterative schemes in order to obtain the shape of the unknown target (minimizer), which minimizes the norm between the measured scattered in the presence of true and man-made targets; see \cite{CZBN,DL,2d_kress_open_arc,RMMP}.

Generally, for an iterative scheme to be successfully applied, it has to be initialized with a initial guess that is close enough to the unknown target. In other words, we require \textit{a priori} information about it. In order to have a better initialization, various non-iterative techniques have been investigated, such as MUltiple SIgnal Classification (MUSIC) \cite{Ammari_music, kirsch_music, PL3}, linear sampling method \cite{lsm2d_crack,Kisch_lsm2d,why_lsm}, topological derivatives \cite{ammari_topological, Park2, topological1}, and Kirchhoff and subspace migrations \cite{Ammari_mf,kirchhoff_phases,Kirchhoff_bybrid}. These techniques yield good results with a large number of incident waves and corresponding scattered fields whereas the efficiency decreases when the number of the incident waves is not large enough,  refer to  \cite{limited-view-it,PL1,PL3}.

The recently developed direct sampling method (DSM) is a non-iterative technique for imaging the shapes and the localizations of small and extended targets using either one or a few incident fields. According to \cite{dsm2d_ito1,dsm2d_farfield,dsm3d_ito1}, DSM is fast because it does not require any additional operation such as singular-value decomposition (in subspace migration), generating a projection operator onto the noise space (in MUSIC), solving ill-posed integral equations (in linear sampling) or adjoint problems (in topological derivatives), and is robust with respect to some random noise. However,DSM might fail to identify an inhomogeneity that is much smaller than the others or whose permittivity is much lower. This behavior can be explained in the framework of the asymptotic theory of the scattering of small scatterers which, into our best knowledge, has not been done yet.


With the help of the expression of the scattered field obtained using the already mentioned asymptotic theory the indicator function of DSM is expressed as a function of the number, the sizes and the permittivities of the inhomogeneities and the Bessel function of order zero. Thanks to this analysis the reasons of the limitations of the original DSM are exhibited and an improved version is proposed.  Then, we show that the original DSM and its alternative version are strongly connected with a normalized version of the Kirchhoff migration.

The paper is structured as follows. In Section~\ref{sec:2},  the two-dimensional direct scattering problem with small dielectric anomalies in a homogeneous medium is introduced. In Section~\ref{sec:3}, a short description of original DSM  for single and multiple impinging directions is proposed. In Section~\ref{sec:4}, a theoretical analysis of the performance of the DSM is established in the framework of the asymptotic hypothesis and an alternative DSM which performs better for a multiple-transmitter configuration is proposed. Section~\ref{sec:5} introduces the Kirchhoff migration and shows its connection with the traditional DSM and its alternative version.  Section~\ref{sec:6} is dedicated to the numerical experiments illustrating our proposal in various cases. Conclusions and perspectives are in Section~\ref{sec:7}.

\section{Two-dimensional direct scattering problem}\label{sec:2}
In this section, the two-dimensional direct scattering problem in the case of a set of small dielectric inclusions is briefly introduced (see Figure~\ref{Location1-1} for a sketch). Let us assume that an homogeneous space is affected by a collection of $M$ inhomogeneities $\tau_{m}, m = 1, \ldots, M$ and let $\tau$ be the collection of $\tau_{m}$, i.e., $\tau=\sum_{m=1}^{M}\tau_{m}$. 

Herein, we assume that all involved materials are non-magnetic and are characterized by their dielectric permittivity at the operating angular frequency $\omega=2\pi f$, $f$ being the frequency in Hz. Let $\mu(\fx)\equiv\mu_0$ be the magnetic permeability and $\eps_{0}$ and $\eps_{m}$ be he dielectric permittivity of $\rb^{2}$ and $\tau_{m}$, respectively. A piecewise-constant permittivity $0<\eps(\fx)<+\infty$ and wavenumber $0<k(\fx)<+\infty$ can then be defined as
\[\eps(\fx)=\left\{
\begin{array}{rl}
\medskip\eps_{m},&\fx\in\tau_{m},\\
\eps_{0},&\fx\in\rb^{2}\backslash\overline{\tau}
\end{array}
\right.~~\mbox{and}\quad
k(\fx)=\left\{
\begin{array}{rl}
\medskip k_m=\omega\sqrt{\eps_m\mu_0},&\fx\in\tau_{m},\\
k_0=\omega\sqrt{\eps_0\mu_0},&\fx\in\rb^{2}\backslash\overline{\tau}
\end{array}
\right.,\]
respectively. Herein, the wavenumber $k_0$ is of the form $k_0=2\pi/\lambda$, where $\lambda$ denotes the wavelength. 

In order to be within the framework of the asymptotic formula $\tau_m$ is defined as a dielectric inhomogeneity of small size $\alpha_m\ll\lambda/2$  as
\begin{equation}
\tau_{m}=\fr_{m}+\alp_{m}\fB_{m},
\end{equation}
where $\fB_{m}$ is a simply connected domain with a smooth boundary containing the origin and $\fr_{m}$ denotes the location of $\tau_m$, assumed to satisfy 
\begin{equation}
0< d_{0} < |\fr_{m}-\fr_{m'}|,\quad\forall m\neq m'.
\end{equation} 
 For simplicity, we assume that all $\tau_m$ are in a ball with radius $\alpha_m$, i.e., we let $\fB_{m}=\fB$ (Figure~\ref{Location1-2}) .

\begin{figure}
\centering
		\subfigure[Scattering problem]{\label{Location1-1}\centering 
		\begin{minipage}[c][6cm][c]{.45\linewidth}\centering
		\psscalebox{1.0 1.0} 
{\centering\small
\begin{pspicture}(0,-2.4345348)(6.833256,2.4345348)
\definecolor{colour0}{rgb}{1.0,0.8,0.0}
\definecolor{colour1}{rgb}{1.0,0.8,0.2}
\pscircle[linecolor=black, linewidth=0.04, fillstyle=gradient, gradlines=2000, gradbegin=colour1, gradend=colour0, dimen=outer](3.1651163,-0.11011628){0.3}
\psframe[linecolor=black, linewidth=0.04, dimen=outer](3.8069768,1.2922093)(0.94186044,-1.572907)
\pscircle[linecolor=black, linewidth=0.04, fillstyle=gradient, gradlines=2000, gradbegin=colour1, gradend=colour0, dimen=outer](1.7023256,0.7410465){0.3}
\rput(1.711628,0.7410465){$\tau_1$}
\rput(3.155814,-0.09151163){$\tau_2$}
\pscircle[linecolor=black, linewidth=0.04, fillstyle=gradient, gradlines=2000, gradbegin=colour1, gradend=colour0, dimen=outer](1.8418604,-0.8124419){0.3}
\rput(1.8325582,-0.8124419){$\tau_3$}
\pscircle[linecolor=black, linewidth=0.04, linestyle=dashed, dash=0.17638889cm 0.10583334cm, dimen=outer](2.3790698,-0.055465117){2.3790698}
\rput[bl](1.2255814,-1.3915117){$\Omega_\Gamma$}
\pscircle[linecolor=black, linewidth=0.04, fillstyle=gradient, gradlines=2000, gradbegin=black, gradend=black, dimen=outer](2.6744187,2.2782557){0.06744186}
\pscircle[linecolor=black, linewidth=0.04, fillstyle=gradient, gradlines=2000, gradbegin=black, gradend=black, dimen=outer](3.7209303,1.882907){0.06744186}
\pscircle[linecolor=black, linewidth=0.04, fillstyle=gradient, gradlines=2000, gradbegin=black, gradend=black, dimen=outer](4.3674417,1.2363954){0.06744186}
\pscircle[linecolor=black, linewidth=0.04, fillstyle=gradient, gradlines=2000, gradbegin=black, gradend=black, dimen=outer](4.716279,0.32011628){0.06744186}
\rput[bl](3.3116279,1.6177907){$\vec{x}_{n}$}
\psline[linecolor=black, linewidth=0.04, arrowsize=0.05291667cm 2.5,arrowlength=1.4,arrowinset=0.0]{<-}(5.293023,0.14802326)(6.0744185,0.14802326)(6.0744185,0.14802326)(6.0744185,0.14802326)
\rput[bl](6.2232556,0.045697674){$\vec{d}$}
\psline[linecolor=black, linewidth=0.026](5.9255815,0.49220932)(5.9255815,-0.21476744)
\psline[linecolor=black, linewidth=0.026](5.772093,0.4782558)(5.772093,-0.19151163)
\rput[bl](5.023256,-0.55895346){Transmitter}
\rput[bl](4.9581394,2.1945348){Receivers}
\psline[linecolor=black, linewidth=0.04, arrowsize=0.05291667cm 2.5,arrowlength=1.4,arrowinset=0.0]{<-}(2.660465,2.2968605)(4.837209,2.4084883)
\psline[linecolor=black, linewidth=0.04, arrowsize=0.05291667cm 2.5,arrowlength=1.4,arrowinset=0.0]{<-}(3.7425542,1.9259433)(4.8248878,2.2956848)
\psline[linecolor=black, linewidth=0.04, arrowsize=0.05291667cm 2.5,arrowlength=1.4,arrowinset=0.0]{<-}(4.3912096,1.2584435)(4.827395,2.1445796)
\rput[bl](0.34418604,1.6177907){$\Gamma$}
\psline[linecolor=black, linewidth=0.04, arrowsize=0.05291667cm 2.5,arrowlength=1.4,arrowinset=0.0]{<-}(4.7865586,0.40728074)(4.980883,2.0189984)
\end{pspicture}
}
\end{minipage}
		}
		\subfigure[Inhomogeneities $\tau_{m}$]{\label{Location1-2}\centering
				\begin{minipage}[c][6cm][c]{.35\linewidth}\centering
\psscalebox{1.0 1.0} 
{\centering\small
\begin{pspicture}(0,-0.97)(2.52,0.97)
\definecolor{colour0}{rgb}{1.0,0.8,0.0}
\pscircle[linecolor=black, linewidth=0.04, fillstyle=solid,fillcolor=colour0, dimen=outer](0.97,0.0){0.97}
\psline[linecolor=black, linewidth=0.04, linestyle=dashed, dash=0.17638889cm 0.10583334cm, arrowsize=0.01cm 4.0,arrowlength=1.4,arrowinset=0.0,dotsize=0.07055555cm 2.0]{<-*}(1.58,0.59)(0.98,0.01)(0.98,0.01)
\rput[bl](0.78,-0.47){$\vec{r}_{m}$}
\rput[bl](1.22,-0.05){$\alpha_m$}
\rput[bl](2.02,-0.63){$\vec{B}_{m}$}
\end{pspicture}
}		
\end{minipage}
}
		\caption{\label{Location}Configuration of the scattering problem for $M=3$ (left) and sketch of the inhomogeneities $\tau_{m}$ (right).}
\end{figure}
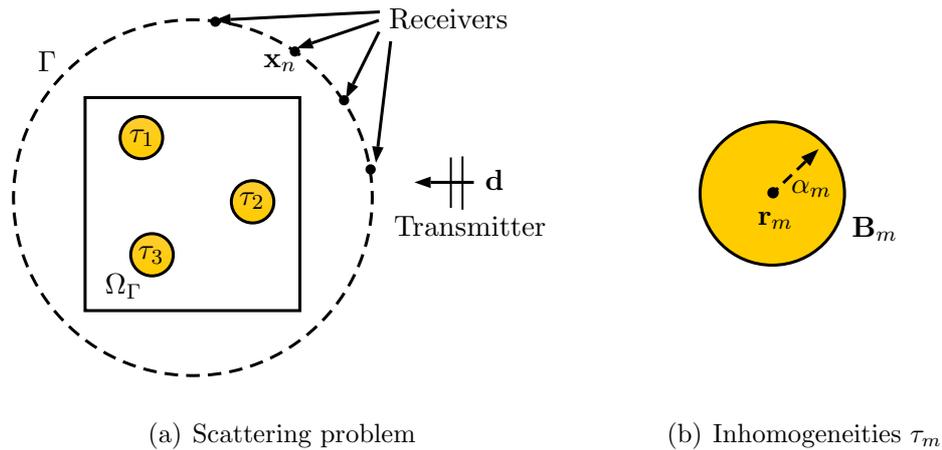

We consider the following plane-wave illumination: let $u^{i}(\fx,\fd)=\rme^{\rmi k_{0}\fd\cdot\fx}$ be the incident field with propagation direction $\fd\in\mathbb{S}^{1}$, where $\mathbb{S}^{1}$ denotes the two-dimensional unit circle. We let $u(\fx,\fd)$ be the time-harmonic total field that satisfies the Helmholtz equation
\begin{equation}\label{helmholtz}
\lap u(\fx,\fd)+k^{2}(\fx)u(\fx,\fd)=0
\end{equation}
with the appropriate transmission conditions on the boundary of $\tau_{m}$. The total field $u(\fx,\fd)$ can be decomposed as $u^{i}(\fx,\fd)+u^{s}(\fx,\fd)$, where $u^{s}(\fx,\fd)$ is the scattered field that is required to satisfy the Sommerfeld radiation condition
\begin{equation}
\lim_{|\fx|\to\infty}\bke{\f{\pa u^{s}(\fx,\fd)}{\pa\fx}-\rmi k_{0}u^{s}(\fx,\fd)}=0
 \end{equation}
uniformly in all directions $\hat{\fx}=\fx/|\fx|$. Under such assumptions and according to \cite{AmmariKang}, we have, what follows
\begin{lemma}[Asymptotic formula for the scattered field]\label{asymptotic_scatter}
	Assume that $\tau_m$ are well separated from each other. Then, $u^{s}(\fx,\fd)$ can be represented by the following asymptotic expansion:
	\begin{equation}\label{Asymptotic}
	u^{s}(\fx,\fd)=\f{k_0^{2}(1+\rmi)}{4\sqrt{k_0\pi}}\sum_{m=1}^{M}\alp_{m}^{2}
\left(\frac{\eps_{m}-\eps_{0}}{\sqrt{\eps_0\mu_0}}\right)\abs{\fB}u^{i}(\fr_m,\fd)\Phi(\fr_{m},\fx) + \mathcal{O}(\alp_{m}^{2}),
	\end{equation}
	where $\abs{\fB}$ denotes the area of $\fB$ and $\Phi$ is the two-dimensional fundamental solution of the Helmholtz equation (or Green's function):
\begin{equation}
\Phi(\fz,\fx)=-\f{\rmi}{4}\mathrm{H}^{1}_{0}(k_{0}|\fz-\fx|)=-\f{\rmi}{4}\bigg(\mathrm{J}_0(k_{0}|\fz-\fx|)+\rmi \mathrm{Y}_0(k_{0}|\fz-\fx|)\bigg),
\end{equation}
where $\mathrm{J}_0$ and $\mathrm{Y}_0$ are the zeroth-order Bessel and Neumann function, respectively.
\end{lemma}
This formula will play a key role in our investigation.

\section{Introduction to the direct sampling method}\label{sec:3}
In this section, the DSM is briefly introduced (a more detailed discussion can be found in \cite{dsm2d_ito1,dsm2d_farfield}). The scattered-field data are measured at $N$ points $\fx_n$, $n=1,2,\cdots,N$, over the measurement curve $\Gamma$ and $\Omega_\Gamma$ is a domain that is enclosed by $\Gamma$ as described in Figure~\ref{Location1-1}. As in \cite{dsm2d_ito1}, we assume that the total number of measurement points $N$ is sufficiently large and that $\Gamma$ is a simply connected smooth curve.

\paragraph{Single impinging direction} The indicator function of DSM is defined by
\begin{equation}\label{DSM}
\mathcal{I}_{\mathrm{DSM}}(\fz):=\f{\displaystyle\abs{\bkc{u^{s}(\fx_n,\fd),\Phi(\fz,\fx_n)}_{L^{2}(\Gamma)}}}{\displaystyle\norm{u^{s}(\fx_n,\fd)}_{L^{2}(\Gamma)}\norm{\Phi(\fz,\fx_n)}_{L^{2}(\Gamma)}},
\end{equation}
for any search point $\fz\in\Omega_\Gamma$ where
\begin{eqnarray}
\bkc{u^{s}(\fx_n,\fd),\Phi(\fz,\fx_n)}_{L^{2}(\Gamma)}&=\sum_{n=1}^{N}u^{s}(\fx_n,\fd)\overline{\Phi(\fz,\fx_n)}\rmd\fx \label{DiscretInnerProduct}\\
& \approx\int_{\Gamma}u^{s}(\fx,\fd)\overline{\Phi(\fz,\fx)}\rmd\fx \text{ when } N \rightarrow \infty \label{ContinousInnerProduct}
\end{eqnarray}
and
\begin{equation}
\norm{u^{s}(\fx_n,\fd)}_{L^{2}(\Gamma)}^2=\bkc{u^{s}(\fx_n,\fd),u^{s}(\fx_n,\fd)}_{L^{2}(\Gamma)}.
\end{equation}

Rewriting (9) of \cite{dsm2d_ito1} with our notation gives the unknown scattered field $u^{s}(\fx_n,\fd)$  as
\begin{equation}\label{EscatApprox}
u^{s}(\fx_n, \fd)\approx\sum_{m=1}^{M}W_{m}(\fd,\fr_m)\Phi(\fx_n,\fr_m),
\end{equation}
where $W_{m}(\fd,\fr_m)$ denotes the weight function corresponding to $\tau_{m}$. 
Combining \eref{DiscretInnerProduct} and \eref{EscatApprox} after proper developments (see \cite{dsm2d_ito1} for details)  the following expression is obtained 
\begin{eqnarray}
\bkc{u^{s}(\fx_n,\fd),\Phi(\fz,\fx_n)}_{L^{2}(\Gamma)} &\approx \f{1}{k_0}\sum_{m=1}^{M}W_{m}(\fd,\fr_m)\Im\left(\Phi(\fz,\fx_n) \right)\\
&\approx -\f{\rmi}{4k_0}\sum_{m=1}^{M}W_{m}(\fd,\fr_m)\mathrm{J}_{0}(k_{0}|\fr_{m}-\fz|) \label{lemma_near}
\end{eqnarray}
and thanks to the H{\"o}lder's inequality, we have 
\begin{equation}\label{TraditionalStructure}
\mathcal{I}_{\mathrm{DSM}}(\fz, \fd)=\f{\displaystyle\abs{\bkc{u^{s}(\fx_n,\fd),\Phi(\fz,\fx_n)}_{L^{2}(\Gamma)}}}{\displaystyle\norm{u^{s}(\fx_n,\fd)}_{L^{2}(\Gamma)}\norm{\Phi(\fz,\fx_n)}_{L^{2}(\Gamma)}}\propto\sum_{m=1}^{M} \left | \mathrm{J}_{0}(k_{0}|\fr_{m}-\fz|)\right |.
\end{equation}
Hence, if a point $\fz$ is in the support of one of the inclusions (i.e., $\fz\approx\fr_{m}\in\tau$) then $\mathcal{I}_{\mathrm{DSM}}(\fz)\approx 1$; otherwise, if $\fz\not\in\tau$ then $\mathcal{I}_{\mathrm{DSM}}(\fz)\not\approx 1$ which allow the localization $\fr_m$ of $\tau_m$ via the map of $\mathcal{I}_{\mathrm{DSM}}(\fx)$.

\paragraph{Multiple impinging directions} The indicator function  $\mathcal{I}_{\mathrm{DSM}}(\fz)$  is then given by
\begin{equation}\label{ImagingFew}
\mathcal{I}_{\mathrm{DSM}}(\fz)=\max\left\{\mathcal{I}_{\mathrm{DSM}}(\fz,\fd_{l}), l=1,\ldots,L:\fz\in\Omega_\Gamma\right\},
\end{equation}
where $\mathcal{I}_{\mathrm{DSM}}(\fz,\fd_{l})$ is the indicator function for an incident field of propagation direction $\fd_l$. Note that \eref{ImagingFew} also works if $L=1$ and provides the same result than \eref{TraditionalStructure} so from now on only \eref{ImagingFew} is used whatever the number of incidences $L$ is.

However, this is restricted to the following situation: either there is only one inhomogeneity or, if several, permittivities and radii of all $\tau_m$ are the same. In practice, \eref{TraditionalStructure} does not describe other situations accurately. Hence, further analysis of the structure of the indicator function is required . 

\section{Structure analysis of the direct sampling method and  alternative direct sampling method}\label{sec:4}
In the following the mathematical structure of the DSM indicator function is analyzed thanks to the asymptotic hypothesis on the scatterers and the corresponding scattering field formula \eref{Asymptotic}. An alternative direct sampling method is suggested to improve efficiency of the classical method.
  

\subsection{Analysis of direct sampling method in the asymptotic hypothesis}
\paragraph{Single impinging direction} in this case $\mathcal{I}_{\mathrm{DSM}}(\fz)$ can be written as follows
\begin{theorem}\label{DSM_case1}
Assume that the total number $N$ of measurement points is sufficiently large. Then $\mathcal{I}_{\mathrm{DSM}}(\fx)$ can be represented as
	\begin{equation}\label{Structure_case1}\fl
	\mathcal{I}_{\mathrm{DSM}}(\fz)=\frac{|\Psi_1(\fz)|}{\displaystyle\max_{\fz\in\Omega_\Gamma}|\Psi_1(\fz)|}, \text{ where } \Psi_1(\fz)=\sum_{m=1}^{M}\alp_{m}^{2}(\eps_{m}-\eps_{0})\e^{\rmi k_{0}\fd\cdot\fr_{m}}\mathrm{J}_{0}(k_0|\fz-\fr_{m}|).
	\end{equation}
\end{theorem}
\begin{proof}
	Combining \eref{Asymptotic}  and \eref{ContinousInnerProduct} leads to 
	\begin{eqnarray}\label{Calculation}
	\fl
	\bkc{u^{s}(\fx_n,\fd),\Phi(\fz,\fx_n)}_{L^{2}(\Gamma)}&=\int_{\Gamma}u^{s}(\fx,\fd)\overline{\Phi(\fz,\fx)}\rmd\fx\\
	&=\int_{\Gamma}\f{k_0^{2}(1+\rmi)}{4\sqrt{k_0\pi}}\sum_{m=1}^{M}\alp_{m}^{2}
\left(\frac{\eps_{m}-\eps_{0}}{\sqrt{\eps_0\mu_0}}\right)\abs{\fB}\rme^{\rmi k_{0}\fd\cdot\fr_m}\Phi(\fr_{m},\fx)\overline{\Phi(\fz,\fx)}\rmd\fx\nonumber\\
&=\f{k_0^{2}(1+\rmi)}{4\sqrt{k_0\pi}}\sum_{m=1}^{M}\alp_{m}^{2}
\left(\frac{\eps_{m}-\eps_{0}}{\sqrt{\eps_0\mu_0}}\right)\abs{\fB}\rme^{\rmi k_{0}\fd\cdot\fr_m}\int_{\Gamma}\Phi(\fr_{m},\fx)\overline{\Phi(\fz,\fx)}\rmd\fx\nonumber\\
	&\approx\f{k_0(1-\rmi)}{16\sqrt{k_0\pi}}\sum_{m=1}^{M}\alp_{m}^{2}
\left(\frac{\eps_{m}-\eps_{0}}{\sqrt{\eps_0\mu_0}}\right)\abs{\fB}\rme^{\rmi k_{0}\fd\cdot\fr_m}\mathrm{J}_{0}(k_{0}|\fz-\fr_{m}|).
	\end{eqnarray}
Applying H{\"o}lder's inequality gives
\begin{equation}
\abs{\bkc{u^{s}(\fx_n,\fd),\Phi(\fz,\fx_n)}_{L^{2}(\Gamma)}}\leq\norm{u^{s}(\fx_n,\fd)}_{L^{2}(\Gamma)}\norm{\Phi(\fz,\fx_n)}_{L^{2}(\Gamma)},
\end{equation}
which leads to \eref{Structure_case1} and completes the proof.
\end{proof}

\begin{remark}\label{Remark1}
Theorem~\eref{DSM_case1} shows that the imaging performance of the indicator function $\mathcal{I}_{\mathrm{DSM}}(\fx)$ is highly dependent on the permittivity, size, and number of the inhomogeneities. If one of those has a  permittivity and/or a size which is significantly larger than of the others, it might be the only one to be identified, the remaining others being not or only partially seen.
\end{remark}

\begin{remark}\label{Remark2} If the radii and permittivities of all the inhomogeneities are the same (i.e., $\alpha_m\equiv\alpha$ and $\eps_m\equiv\eps$ for $m=1,2,\cdots,M$), 
and knowing that $|\e^{\rmi k_{0}\fd\cdot\fr_{m}}|=1$ and $\alp_{m}^{2}(\eps_{m}-\eps_{0})\equiv\alp^{2}(\eps-\eps_{0})$ then $\mathcal{I}_{\mathrm{DSM}}(\fx)$ becomes
\begin{equation}
\mathcal{I}_{\mathrm{DSM}}(\fz)\propto\sum_{m=1}^{M} |\mathrm{J}_{0}(k_{0}|\fr_{m}-\fz|)|
\end{equation}
which is the same as \eref{TraditionalStructure} derived in \cite{dsm2d_ito1}.
\end{remark}

\paragraph{Multiple impinging directions} 
By combining \eref{ImagingFew} and \eref{Structure_case1}, it is easy to see that
\begin{equation}\label{ddsm}\fl
\mathcal{I}_{\mathrm{DSM}}(\fz)\propto \max\left\{\left|\sum_{m=1}^{M}\alp_{m}^{2}(\eps_{m}-\eps_{0})\abs{\fB}\e^{\rmi k_{0}\fd_{l}\cdot\fr_{m}}\mathrm{J}_{0}(k_0|\fz-\fr_{m}|)\right|, l=1, \ldots, L : \fz\in\Omega_\Gamma \right \}
\end{equation}
for which Remark~\ref{Remark1} and  Remark~\ref{Remark2} are also verified.

\subsection{Introduction and analysis of an alternative direct sampling method}
Thanks to our analysis of $\mathcal{I}_{\mathrm{DSM}}(\fz)$ and, in particular, $\Psi_1(\fz)$  \eref{Structure_case1}, it can be seen that the latter \eref{ddsm} contains a factor of the form of $\rme^{\rmi k_{0}\fd_l\cdot\fr_m}$ which generates artifacts due to the oscillating nature of the exponential function. To reduce such a behavior an alternative indicator function of DSM $\mathcal{I}_{\mathrm{DSMA}}(\fz)$ is proposed
\begin{equation}\label{DSMA}
\mathcal{I}_{\mathrm{DSMA}}(\fz):=\f{\displaystyle\left|\sum_{l=1}^{L}\e^{-\rmi k_{0}\fd_l\cdot\fz}\bkc{u^{s}(\fx_n,\fd_l),\Phi(\fz,\fx_n)}_{L^{2}(\Gamma)}\right|}{\displaystyle\max_{\fz\in\Omega_\Gamma}\left|\sum_{l=1}^{L}\e^{-\rmi k_{0}\fd_l\cdot\fz}\bkc{u^{s}(\fx_n,\fd_l),\Phi(\fz,\fx_n)}_{L^{2}(\Gamma)}\right|}.
\end{equation}

%
%
\begin{theorem}\label{DSM_case2}
Assume that the number $N$ of measurement points and the number $L$ of incident fields are sufficiently large. Then, $\mathcal{I}_{\mathrm{DSMA}}(\fz)$ can be represented as 
	\begin{equation}\label{Structure_case2}
	\mathcal{I}_{\mathrm{DSMA}}(\fz)=\frac{|\Psi_2(\fz)|}{\displaystyle\max_{\fz\in\Omega_\Gamma}|\Psi_2(\fz)|}, \text{ where } \Psi_2(\fz)=\sum_{m=1}^{M}\alp_{m}^{2}(\eps_{m}-\eps_{0})\mathrm{J}_{0}(k_0|\fz-\fr_{m}|)^2.
	\end{equation}
\end{theorem}
\begin{proof} Let us note that if $L$ is sufficiently large, the following relationship holds (see \cite{Park}):
\begin{equation}\label{RepresentationBessel}
\sum_{l=1}^{L}\rme^{\rmi k_{0}\fd_l\cdot\fr_m}\overline{\rme^{\rmi k_{0}\fd_l\cdot\fz}}\approx2\pi \mathrm{J}_0(k_0|\fr_m-\fz|).
\end{equation}
Hence, by combining \eref{Calculation} and \eref{RepresentationBessel}, we obtain
\begin{equation}
\fl\eqalign{\Psi_2(\fz)&=\sum_{l=1}^{L}\overline{\e^{\rmi k_{0}\fd_l\cdot\fz}}\bkc{u^{s}(\fx_n,\fd_l),\Phi(\fz,\fx_n)}_{L^{2}(\Gamma)}\\
&=\sum_{l=1}^{L}\overline{\e^{\rmi k_{0}\fd_l\cdot\fz}}\left(\f{k_0\mu_0(1-\rmi)}{16\sqrt{k_0\pi}}\sum_{m=1}^{M}\alp_{m}^{2}
\left(\frac{\eps_{m}-\eps_{0}}{\sqrt{\eps_0\mu_0}}\right)\abs{\fB}\rme^{\rmi k_{0}\fd_l\cdot\fr_m}\mathrm{J}_{0}(k_{0}|\fz-\fr_{m}|)\right)\\
&=\f{k_0\mu_0(1-\rmi)}{16\sqrt{k_0\pi}}\sum_{m=1}^{M}\alp_{m}^{2}\left(\frac{\eps_{m}-\eps_{0}}{\sqrt{\eps_0\mu_0}}\right)\abs{\fB}\mathrm{J}_{0}(k_{0}|\fz-\fr_{m}|)\sum_{l=1}^{L}\rme^{\rmi k_{0}\fd_l\cdot\fr_m}\overline{\e^{\rmi k_{0}\fd_l\cdot\fz}}\\
&=\f{k_0\mu_0(1-\rmi)\pi}{8\sqrt{k_0\pi}}\sum_{m=1}^{M}\alp_{m}^{2}\left(\frac{\eps_{m}-\eps_{0}}{\sqrt{\eps_0\mu_0}}\right)\abs{\fB}\mathrm{J}_{0}(k_{0}|\fz-\fr_{m}|)^2.}
\end{equation}
Finally, applying H{\"o}lder's inequality, \eref{Structure_case2} is derived which completes the proof.
\end{proof}

\begin{remark}\label{Remark3}
	Based on the result in Theorem~\ref{DSM_case2}, we see that
	\begin{equation}\label{CompareImagingFunctions}
	\mathcal{I}_{\mathrm{DSM}}(\fz)\propto \abs{\mathrm{J}_{0}(k_{0}\abs{\fz-\fr_{m}})}\quad\mbox{and}\quad\mathcal{I}_{\mathrm{DSMA}}(\fz)\propto \mathrm{J}_{0}(k_{0}\abs{\fz-\fr_{m}})^2.
	\end{equation}
	One-dimensional plots of \eref{CompareImagingFunctions} are shown in Figure \ref{compare_bessel} and illustrate that $\mathcal{I}_{\mathrm{DSMA}}(\fz)$ would yield better images because its oscillations are smaller than those of $\mathcal{I}_{\mathrm{DSM}}(\fz)$. Hence, any unexpected artifact in the plot of $\mathcal{I}_{\mathrm{DSMA}}(\fz)$ is mitigated by having a sufficiently large number $L$ of incident fields. This result explains theoretically why $\mathcal{I}_{\mathrm{DSMA}}(\fz)$ with large $L$ offers better results than $\mathcal{I}_{\mathrm{DSM}}(\fz)$.
\end{remark}
\begin{figure}[h]
	\centering
	\includegraphics[width=\textwidth]{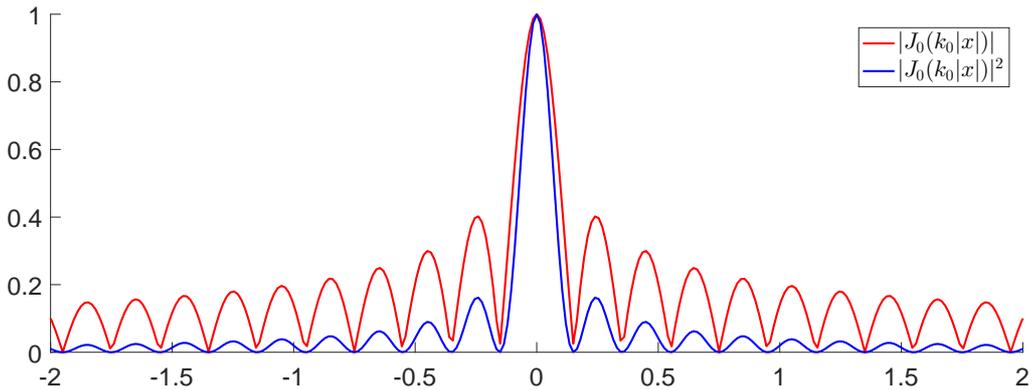}
	\caption{One-dimensional plots of $\abs{\mathrm{J}_{0}(k_{0}|x|)}$ and $|\mathrm{J}_{0}(k_{0}|x|)|^{2}$ for $k_{0}=2\pi/0.4$.}
	\label{compare_bessel}
\end{figure}

\section{Comparison between Kirchhoff migration and direct sampling method}\label{sec:5}
In the following the structures of the Kirchhoff migration, DSM and DSMA are compared. Let us assume that the total numbers of measurement $N$  and  of incident fields $L$ are sufficiently large and let us define the Multi-Static Response (MSR) matrix $\mathbb{K}\in\mathbb{C}^{N\times L}$ as
\begin{equation}
\mathbb{K}=\left[
\begin{array}{cccc}
\medskip u^{s}(\fx_1,\fd_1) & u^{s}(\fx_1,\fd_2) & \cdots & u^{s}(\fx_1,\fd_{L}) \\
\medskip u^{s}(\fx_2,\fd_1) & u^{s}(\fx_2,\fd_2) & \cdots & u^{s}(\fx_2,\fd_{L}) \\
\medskip \vdots & \vdots & \ddots & \vdots \\
u^{s}(\fx_{N},\fd_1) & u^{s}(\fx_{N},\fd_2) & \cdots & u^{s}(\fx_{N},\fd_{L})
\end{array}
\right].
\end{equation}
For $\fz\in\Omega_\Gamma$, the imaging function of Kirchhoff migration is defined as (e.g., see \cite{Ammari_mf})
\begin{equation}\label{image_kirchhoff}
\mathcal{I}_{\mathrm{KM}}(\fz):=\abs{\overline{\fW}_{1}(\fz)^{T}\mathbb{K}\overline{\fW}_{2}(\fz)},
\end{equation}
where
\begin{equation}\label{testvectors}
\eqalign{\fW_{1}(\fz)&=\bigg[\Phi(\fx_{1},\fz),\Phi(\fx_{2},\fz)\cdots,\Phi(\fx_{N},\fz)\bigg]^{T},\\
\fW_{2}(\fz)&=\bigg[\rme^{\rmi k_{0}\fd_{1}\cdot\fz},\rme^{\rmi k_{0}\fd_{2}\cdot\fz},\cdots,\rme^{\rmi k_{0}\fd_{L}\cdot\fz}\bigg]^{T}.}
\end{equation}
A normalized version of \eref{image_kirchhoff} is defined as
\begin{equation}\label{image_nkirchhoff}
\mathcal{I}_{\mathrm{NKM}}(\fz):=\f{\abs{\overline{\fW}_{1}(\fz)^{T}\mathbb{K}\overline{\fW}_{2}(\fz)}}{\displaystyle\max_{\fz\in\Omega_\Gamma}\abs{\overline{\fW}_{1}(\fz)^{T}\mathbb{K}\overline{\fW}_{2}(\fz)}}
\end{equation}
and will be used for our purpose. Then, the following statement is proposed:
\begin{theorem}\label{KM_case2}
	Suppose that the total numbers $L$ of incident fields and $N$ of measurement points are sufficiently large. Then, $\mathcal{I}_{\mathrm{NKM}}(\fz)$ can be represented as 
	\begin{equation}\label{StructureKM_case2}\fl
	\mathcal{I}_{\mathrm{NKM}}(\fz)=\frac{|\Psi_3(\fz)|}{\displaystyle\max_{\fz\in\Omega_\Gamma}|\Psi_3(\fz)|}, \text{ where } \Psi_3(\fz)=\sum_{m=1}^{M}\alp_{m}^{2}(\eps_{m}-\eps_{0})\mathrm{J}_{0}(k_{0}\abs{\fz-\fr_{m}})^{2}.
	\end{equation}
\end{theorem}
\begin{proof}
	From \eref{image_kirchhoff} it can be shown that
	\begin{equation}
	\fl\eqalign{
	\overline{\fW}_{1}(\fz)^{T}\mathbb{K}&=
\left[\begin{array}{c}\medskip\overline{\Phi(\fz,\fx_{1})}\\ \medskip\overline{\Phi(\fz,\fx_{2})}\\
\medskip\vdots\\
\overline{\Phi(\fz,\fx_{N})}\end{array}\right]^T\left[
\begin{array}{cccc}
\medskip u^{s}(\fx_1,\fd_1) & u^{s}(\fx_1,\fd_2) & \cdots & u^{s}(\fx_1,\fd_{L}) \\
\medskip u^{s}(\fx_2,\fd_1) & u^{s}(\fx_2,\fd_2) & \cdots & u^{s}(\fx_2,\fd_{L}) \\
\medskip \vdots & \vdots & \ddots & \vdots \\
u^{s}(\fx_{N},\fd_1) & u^{s}(\fx_{N},\fd_2) & \cdots & u^{s}(\fx_{N},\fd_{L})
\end{array}
\right]\\
	&=\bigg[U_{1}\left(\fz,\fd_1\right),U_{2}\left(\fz,\fd_2\right),\cdots,U_{L}\left(\fz,\fd_L\right)\bigg]:=\fU\left(\fz\right),}
	\end{equation}
	where
	\begin{equation}
	U_{l}\left(\fz,\fd_l\right):=\sum_{n=1}^{N}\overline{\Phi(\fz,\fx_{n})}u^{s}(\fx_{n},\fd_{l}),\quad l=1,\cdots,L.
	\end{equation}
	Combining the latter with \eref{Calculation} leads to
	\begin{equation}\label{KM_U}
	U_{l}\left(\fz,\fd_l\right)=\f{k_0\mu_0(1-\rmi)}{16\sqrt{k_0\pi}}\sum_{m=1}^{M}\alp_{m}^{2}
\left(\frac{\eps_{m}-\eps_{0}}{\sqrt{\eps_0\mu_0}}\right)\abs{\fB}\rme^{\rmi k_{0}\fd_l\cdot\fr_m}\mathrm{J}_{0}(k_{0}|\fz-\fr_{m}|).
	\end{equation}
	Rewritting \eref{image_kirchhoff} with the use of  \eref{KM_U} and \eref{RepresentationBessel} gives
	\begin{equation}
	\fl\eqalign{\mathcal{I}_{\mathrm{KM}}(\fz)&=\overline{\fW}_{1}(\fz)^{T}\mathbb{K}_{1}\overline{\fW}_{2}(\fz)=\fU\left(\fz\right)\overline{\fW}_{2}(\fz)\\
	&=\bigg[U_{1}\left(\fz,\fd_1\right),U_{2}\left(\fz,\fd_2\right),\cdots,U_{L}\left(\fz,\fd_L\right)\bigg]\bigg[\rme^{-\rmi k_{0}\fd_{1}\cdot\fz},\rme^{-\rmi k_{0}\fd_{2}\cdot\fz},\cdots,\rme^{-\rmi k_{0}\fd_{L}\cdot\fz}\bigg]^{T}\\
	&=\f{k_0\mu_0(1-\rmi)}{16\sqrt{k_0\pi}}\sum_{m=1}^{M}\alp_{m}^{2}
\left(\frac{\eps_{m}-\eps_{0}}{\sqrt{\eps_0\mu_0}}\right)\abs{\fB}\mathrm{J}_{0}(k_{0}|\fz-\fr_{m}|)\bke{\sum_{l=1}^{L}\rme^{\rmi k_{0}\fd_{l}\cdot(\fr_{m}-\fz)}}\\
	&=\f{k_0\mu_0(1-\rmi)\pi}{8\sqrt{k_0\pi}}\sum_{m=1}^M\alp_{m}^{2}\bke{\f{\eps_{m}-\eps_{0}}{\sqrt{\eps_{0}\mu_0}}}\abs{\fB_{m}}\mathrm{J}_{0}(k_{0}\abs{\fz-\fr_{m}})^{2}.},
	\end{equation}
which completes the proof.
\end{proof}

\begin{remark}\label{Remark4}
The comparison of \eref{Structure_case2} and \eref{StructureKM_case2} shows that the alternative DSM and normalized Kirchhoff migration are identical when the number of incident fields becomes sufficiently large. Furthermore, for a single impinging direction, DSM can be regarded as normalized Kirchhoff migration since $|\rme^{\rmi k_{0}\fd\cdot\fr_m}|=|\rme^{\rmi k_{0}\fd\cdot(\fr_{m}-\fz)}|\equiv1$, $\mathcal{I}_{\mathrm{DSM}}(\fz)$ \eref{Structure_case1} can then be rewritten as
\begin{equation}
\mathcal{I}_{\mathrm{DSM}}(\fz)=\f{\abs{\overline{\fW}_{1}(\fz)^{T}\mathbb{K}}}{\displaystyle\max_{\fz\in\Omega_\Gamma}\abs{\overline{\fW}_{1}(\fz)^{T}\mathbb{K}}}=\f{\abs{\overline{\fW}_{1}(\fz)^{T}\mathbb{K}\overline{\fW}_{2}(\fz)}}{\displaystyle\max_{\fz\in\Omega_\Gamma}\abs{\overline{\fW}_{1}(\fz)^{T}\mathbb{K}\overline{\fW}_{2}(\fz)}}=\mathcal{I}_{\mathrm{NKM}}(\fz).
\end{equation}
where $\fW_{1}(\fz)$ and $\fW_{2}(\fz)$ are defined  \eref{testvectors}. 

In summary the relationship between $\mathcal{I}_{\mathrm{NKM}}(\fz)$, $\mathcal{I}_{\mathrm{DSM}}(\fz)$ and $\mathcal{I}_{\mathrm{DSMA}}(\fz)$ is given by
\begin{equation}\label{Comparison}
\mathcal{I}_{\mathrm{NKM}}(\fz)=\left\{\begin{array}{lll}
\medskip\mathcal{I}_{\mathrm{DSM}}(\fz) \left(=\mathcal{I}_{\mathrm{DSMA}}(\fz)\right)&\mbox{when}&L=1\\
\mathcal{I}_{\mathrm{DSMA}}(\fz)&\mbox{when}&L\geq 2.\end{array}\right.
\end{equation}
\end{remark}

\section{Numerical experiments}\label{sec:6}
In this section, some numerical experiments are provided in order to support our theoretical proposal. Throughout this section, the applied wavenumber $k_0$ is of the form $k_0=2\pi/\lambda$ with $\lambda=\SI{0.4}{\meter}$, the measurement curve $\Gamma$ is chosen as the circle with radius $7.5\lambda=\SI{3}{\meter}$ centered at the origin, and the total number of measurement points is set to $N=36$. The search domain $\Omega_\Gamma$ is a square of side length $3\lambda\cmmnt{(=\SI{1.2}{\meter})}$ divided into squares ot equal side $h=0.612\lam=\SI{0.0245}{\meter}$.

The scattered fields  $u^{s}(\fx_n,\fd_{l})$ due to planar incident waves are generated by \textit{FEKO} (EM simulation software) and a $20$-dB white Gaussian random noise is added using the MATLAB function \texttt{awgn}.

In order to compare the accuracy of the results as objectively as possible  the use of  the Jaccard index \cite{jaccard} which measures the similarity between two finite samples sets A and B is proposed. It is defined as
\begin{equation}\label{jaccard}
 J(A,B) (\%):=\f{|A \cap B|}{|A\cup B|}\times100.
\end{equation}
In our case the Jaccard index is calculated by comparing $\mathcal{I}^{\kappa}_{\text{exact}}\left(\fz\right)$  with various index maps $\mathcal{I}^{\kappa}\left(\fz\right)$ defined as  
\begin{equation}
\label{cases}
\mathcal{I}^{\kappa}\left(\fz\right)=\cases{\mathcal{I}\left(\fz\right)&$\forall \fz$  such that  $\mathcal{I}\left(\fz\right) \ge \kappa$\\
0 & $\forall \fz$  such that  $\mathcal{I}\left(\fz\right) < \kappa$\\}
\end{equation}
where $\kappa$ varies from $0$ to $1$ and where $\mathcal{I}\left(\fz\right)$ can be $\mathcal{I}_{\mathrm{DSM}}(\fz)$, $\mathcal{I}_{\mathrm{DSMA}}(\fz)$ or $\mathcal{I}_{\mathrm{NKM}}(\fz)$ and where  $\mathcal{I}_{\text{exact}}\left(\fz\right)$ is defined as
\begin{equation}
\mathcal{I}_{\text{exact}}\left(\fz\right) =\frac{\abs{k(\fz) - k_{0}}}{\max{\abs{k(\fz) - k_{0}}}}
\end{equation}

\begin{example}[Small disks with the same radius and permittivity]\label{Example1}
First, we consider small dielectric disks $\tau_m$, $m=1,2,3$. The locations $\fr_m$ of $\tau_m$ are
selected as $\fr_1=(0.75\lambda,-0.75\lambda)=(\SI{0.3}{\meter},\SI{-0.3}{\meter})$, $\fr_2=(-\lambda,-0.5\lambda)=(\SI{-0.4}{\meter},\SI{-0.2}{\meter})$, and $\fr_3=(-0.75\lambda,\lambda)=(\SI{-0.3}{\meter},\SI{0.4}{\meter})$.
In this example, we consider the identification of $\tau_m$ with constant radius and permittivity $\alpha_m\equiv0.075\lambda=\SI{0.03}{\meter}$ and $\eps_m\equiv5\eps_0$, respectively. 
\end{example}
Figure~\ref{Result1} shows the map of $\mathcal{I}_{\mathrm{DSM}}(\fx)$ for a single incident wave with $\fd=(-1,0)$. As shown by the previous results \cite{dsm2d_ito1} and the discussion in Remark~\ref{Remark2}, the locations of each inhomogeneity $\tau_m$ are identified even though the Jaccard index has not high value. It can be explain by the fact that a lot of artefacts are present in the image and a high $\kappa$ threshold is needed to better identified the localization of the defects.   

Then the imaging performance of
$\mathcal{I}_{\mathrm{DSM}}(\fz)$  and
$\mathcal{I}_{\mathrm{DSMA}}(\fz)$  is compared as a function of  the number of incident fields $L$ (Figure~\ref{Result1} with $L=1$, $2$, $12$, and $36$). As stated in Remark~\ref{Remark3} and confirmed by the comparison of the Jaccard index, $\mathcal{I}_{\mathrm{DSMA}}(\fz)$ is an improved version of $\mathcal{I}_{\mathrm{DSM}}(\fz)$.
\begin{figure}[h]
	\centering
	\subfigure[Map of $\mathcal{I}_{\mathrm{DSM}}(\fz)$ for $L=1$]{\label{Result1-1}\centering\includegraphics[width=0.325\textwidth]{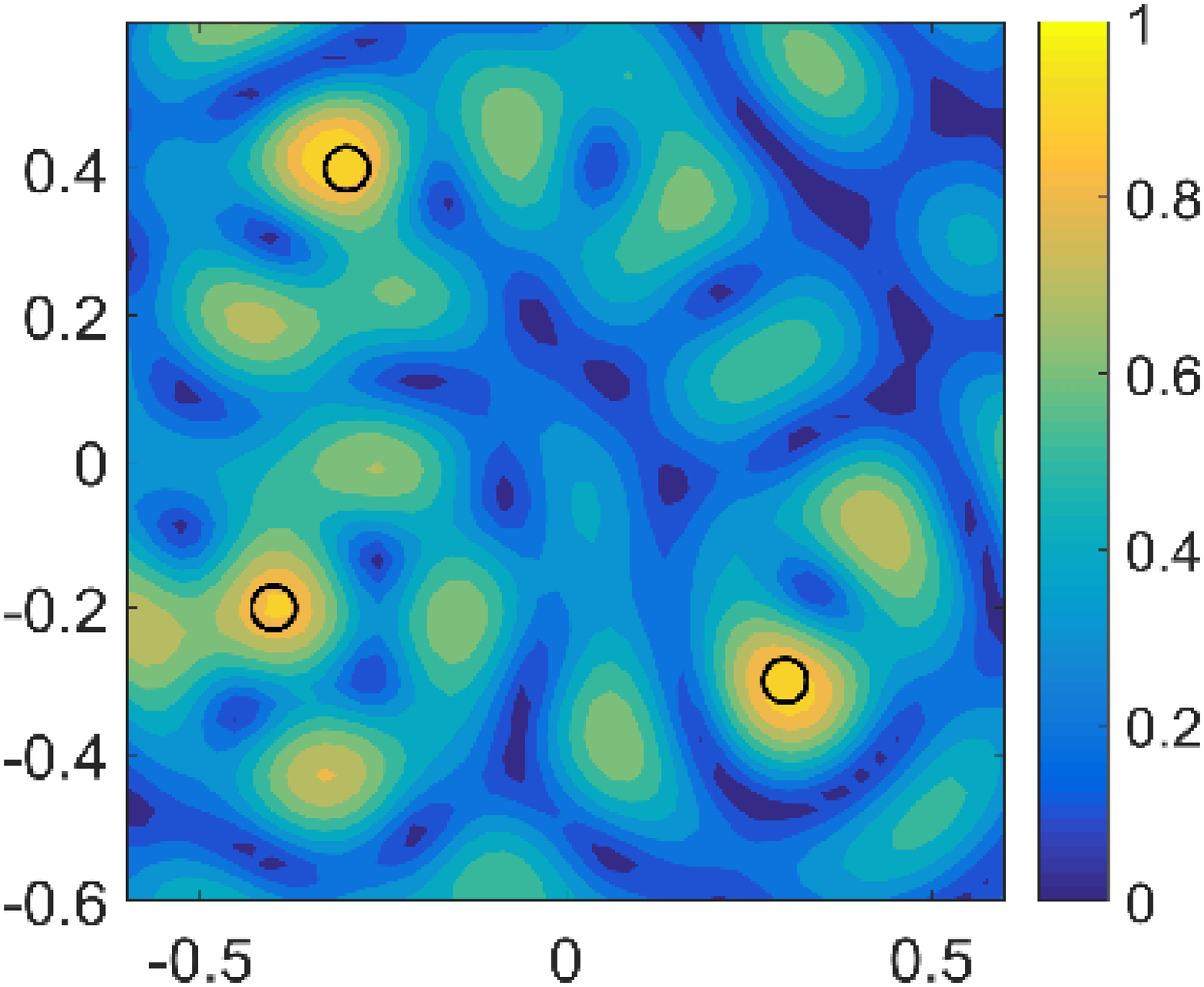}}
	\subfigure[Map of $\mathcal{I}_{\mathrm{DSMA}}(\fz)$ for $L=1$]{\label{Result1-2}\centering\includegraphics[width=0.325\textwidth]{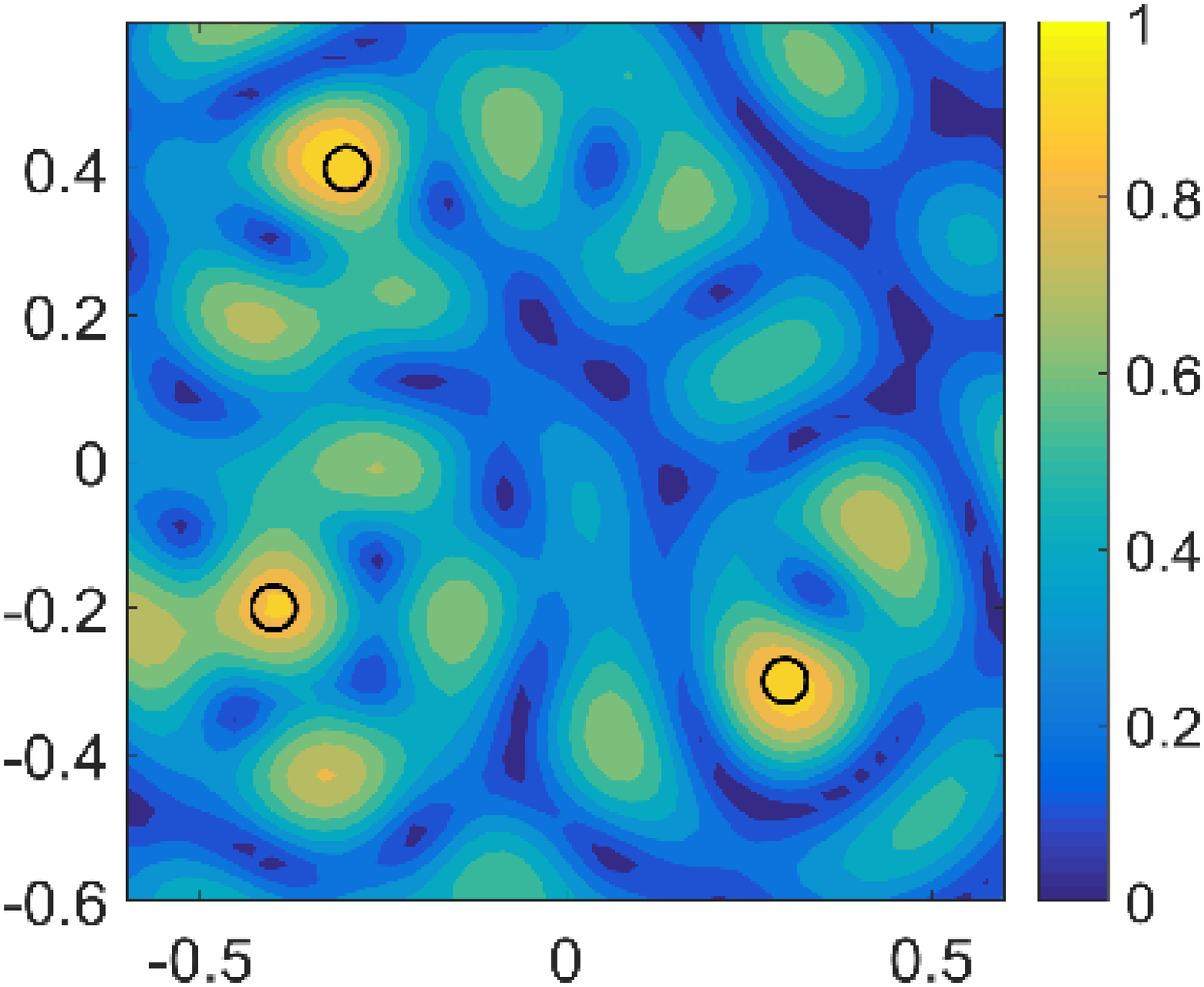}}
	\subfigure[Jaccard index for $L=1$]{\label{Result1-3}\centering\includegraphics[width=0.325\textwidth]{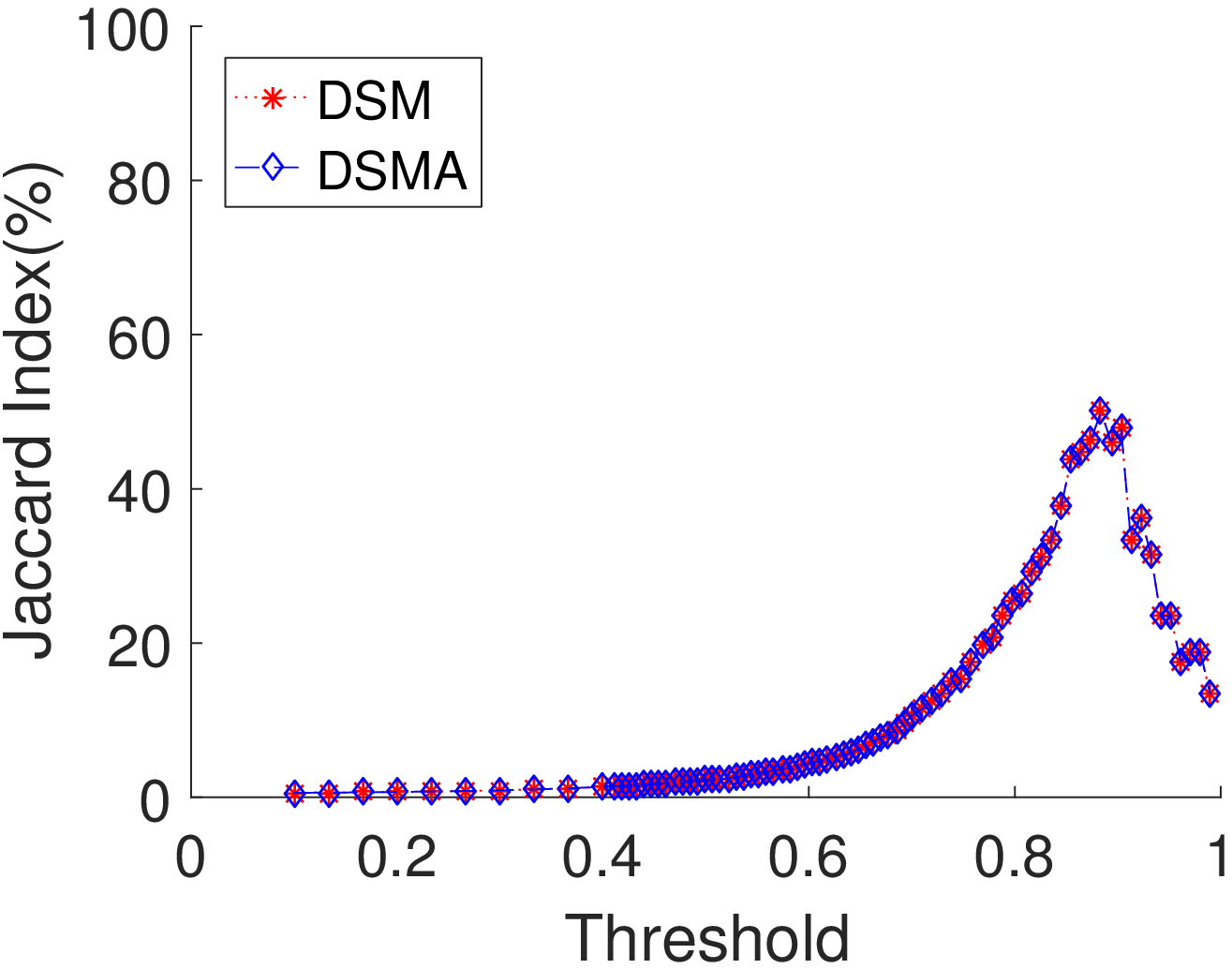}}
	
	\subfigure[Map of $\mathcal{I}_{\mathrm{DSM}}(\fz)$ for $L=2$]{\label{Result1-4}\centering\includegraphics[width=0.325\textwidth]{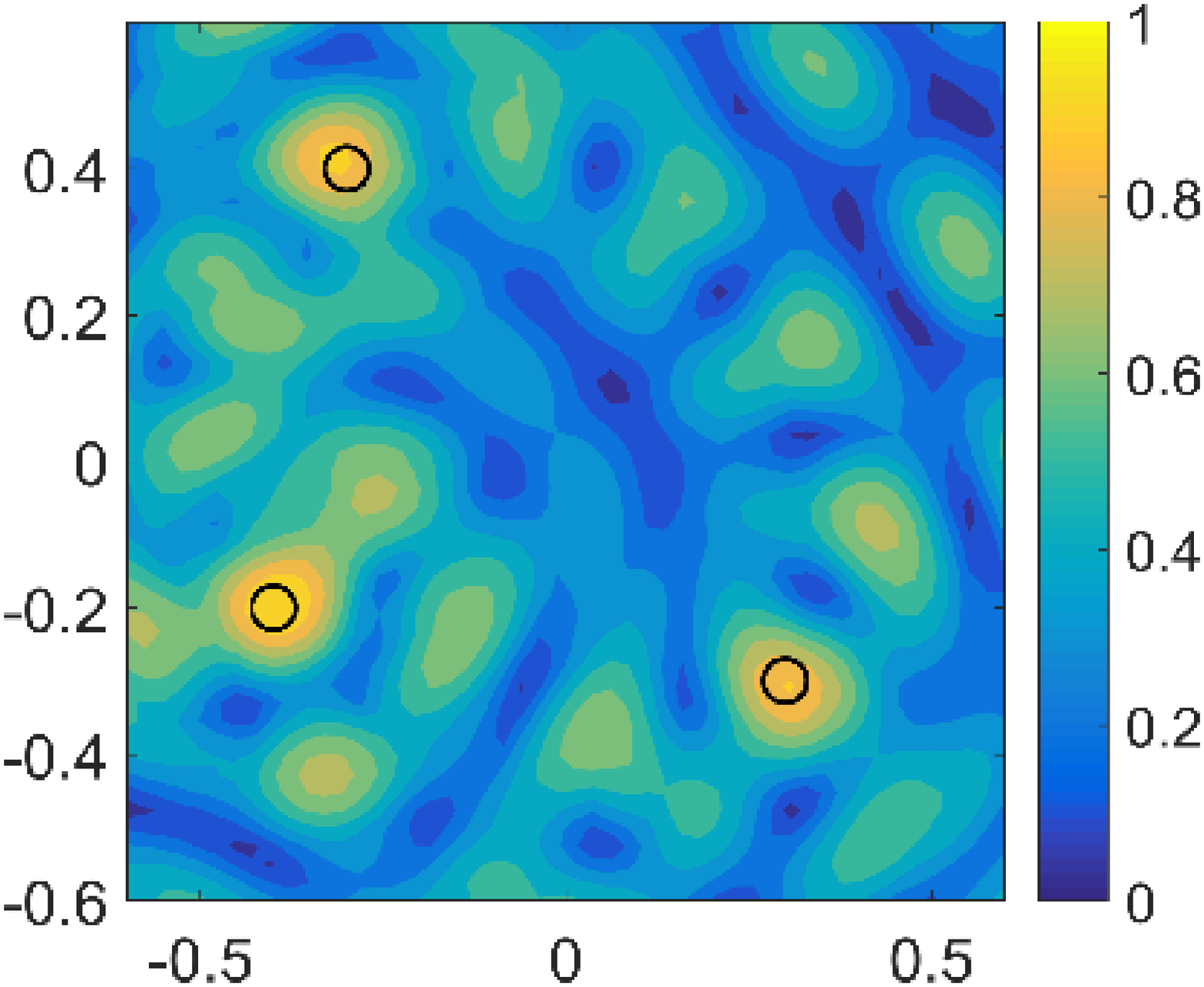}}
	\subfigure[Map of $\mathcal{I}_{\mathrm{DSMA}}(\fz)$ for  $L=2$]{\label{Result1-5}\centering\includegraphics[width=0.325\textwidth]{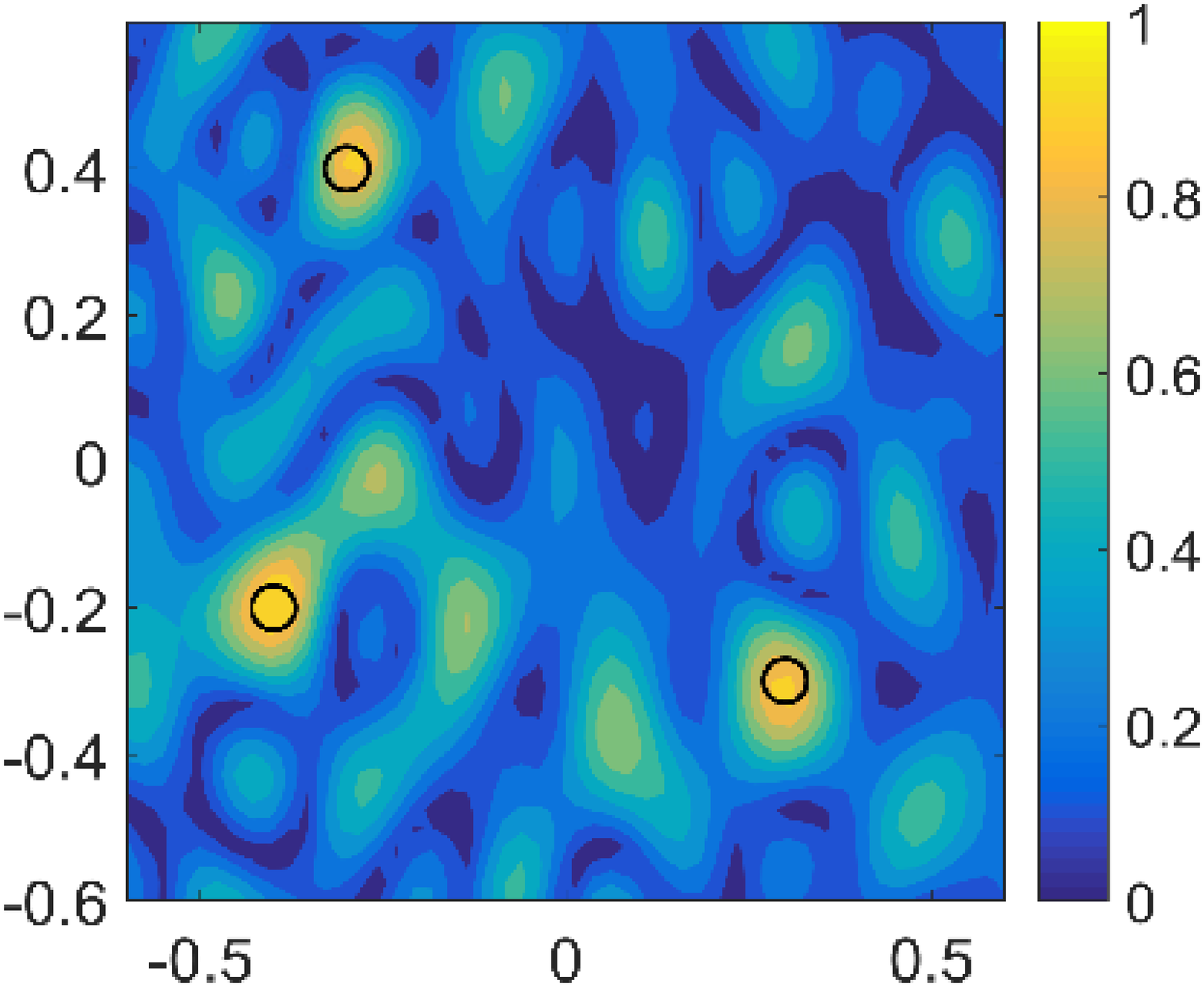}}
	\subfigure[Jaccard index for $L=2$]{\label{Result1-6}\centering\includegraphics[width=0.325\textwidth]{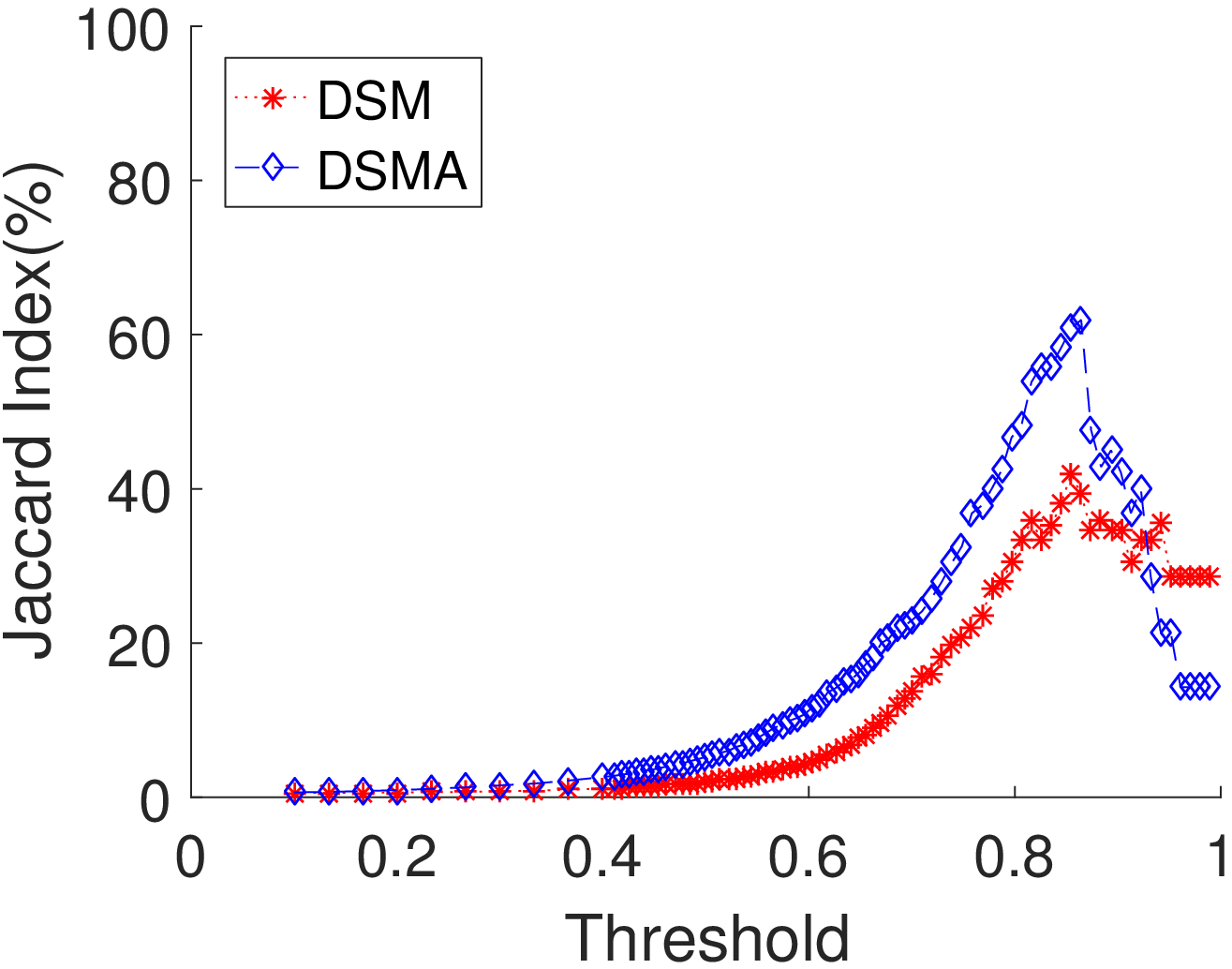}}
	
	\subfigure[Map of $\mathcal{I}_{\mathrm{DSM}}(\fz)$ for $L=12$]{\label{Result1-7}\centering\includegraphics[width=0.325\textwidth]{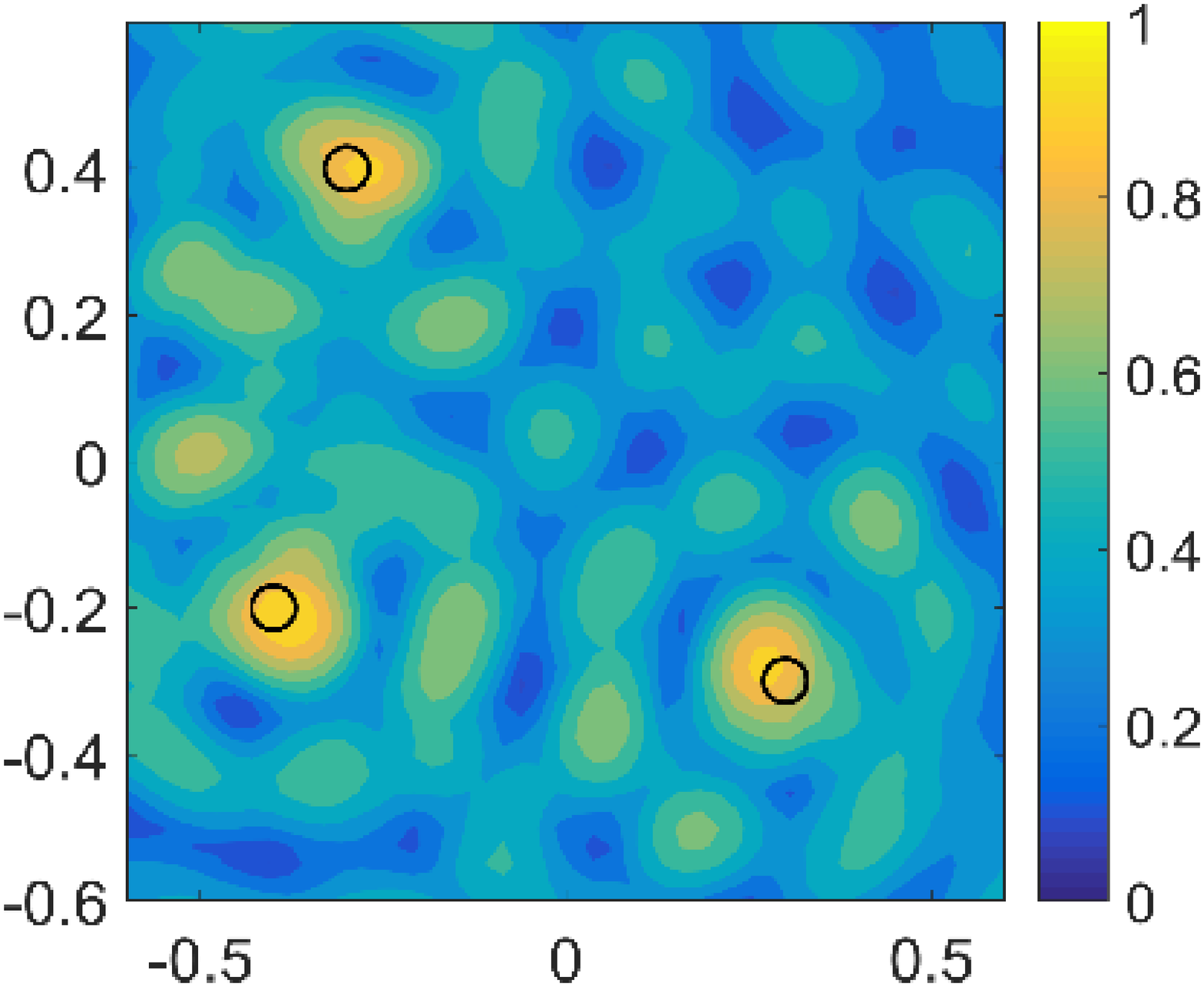}}
	\subfigure[Map of $\mathcal{I}_{\mathrm{DSMA}}(\fz)$ for $L=12$]{\label{Result1-8}\centering\includegraphics[width=0.325\textwidth]{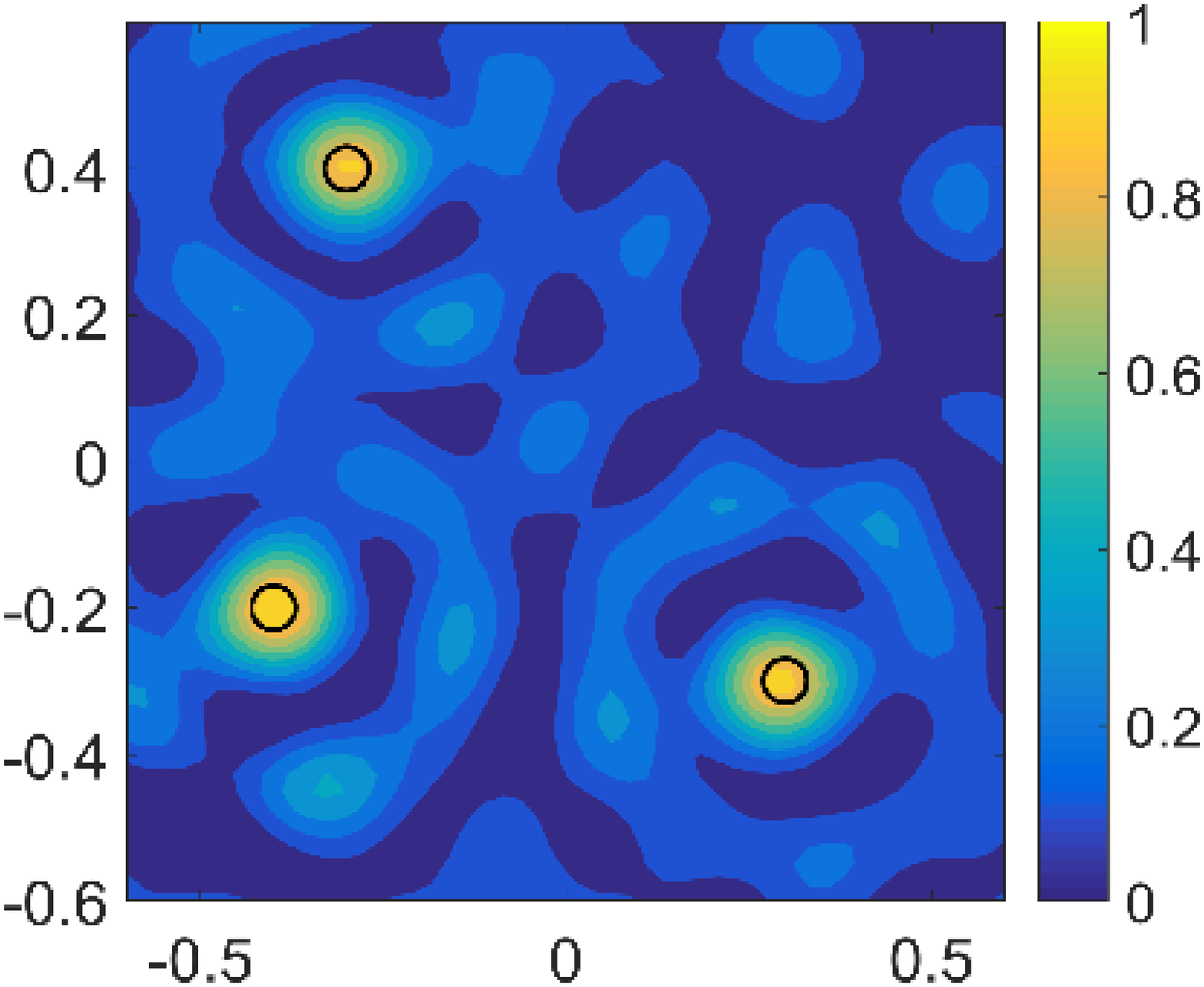}}
	\subfigure[Jaccard index for $L=12$]{\label{Result1-9}\centering\includegraphics[width=0.325\textwidth]{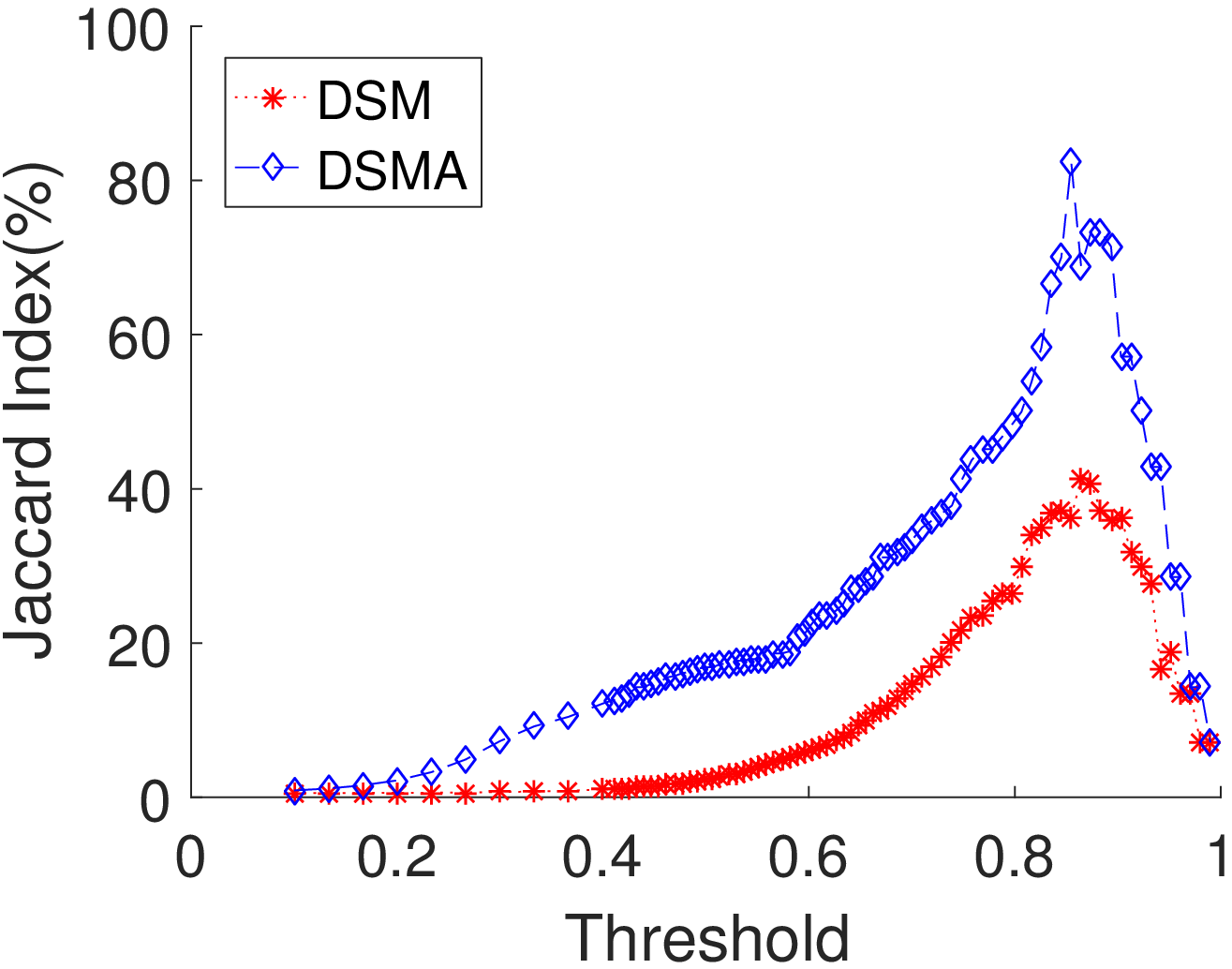}}
	
	\subfigure[Map of $\mathcal{I}_{\mathrm{DSM}}(\fz)$ for $L=36$]{\label{Result1-10}\centering\includegraphics[width=0.325\textwidth]{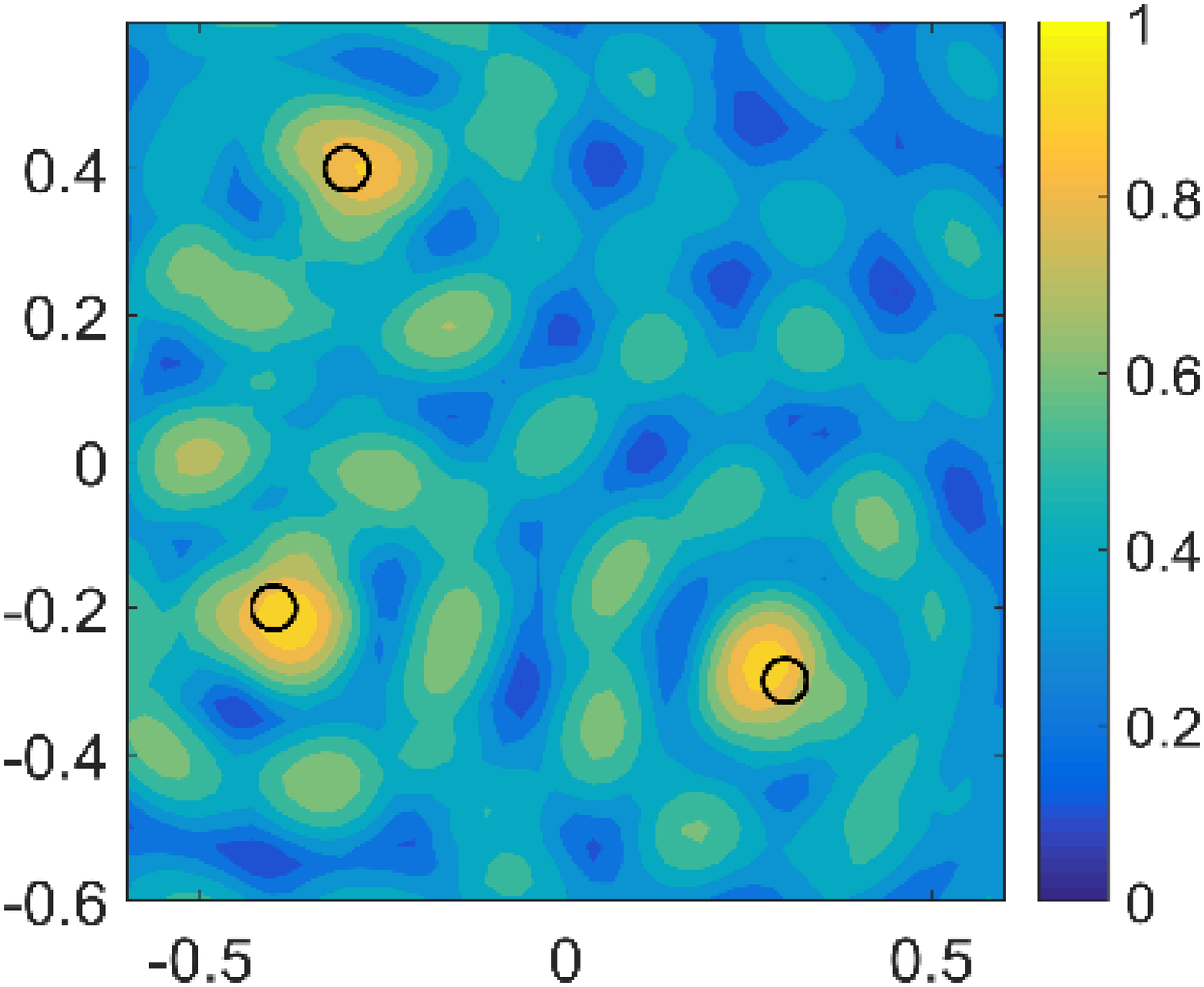}}
	\subfigure[Map of $\mathcal{I}_{\mathrm{DSMA}}(\fz)$ for $L=36$]{\label{Result1-11}\centering\includegraphics[width=0.325\textwidth]{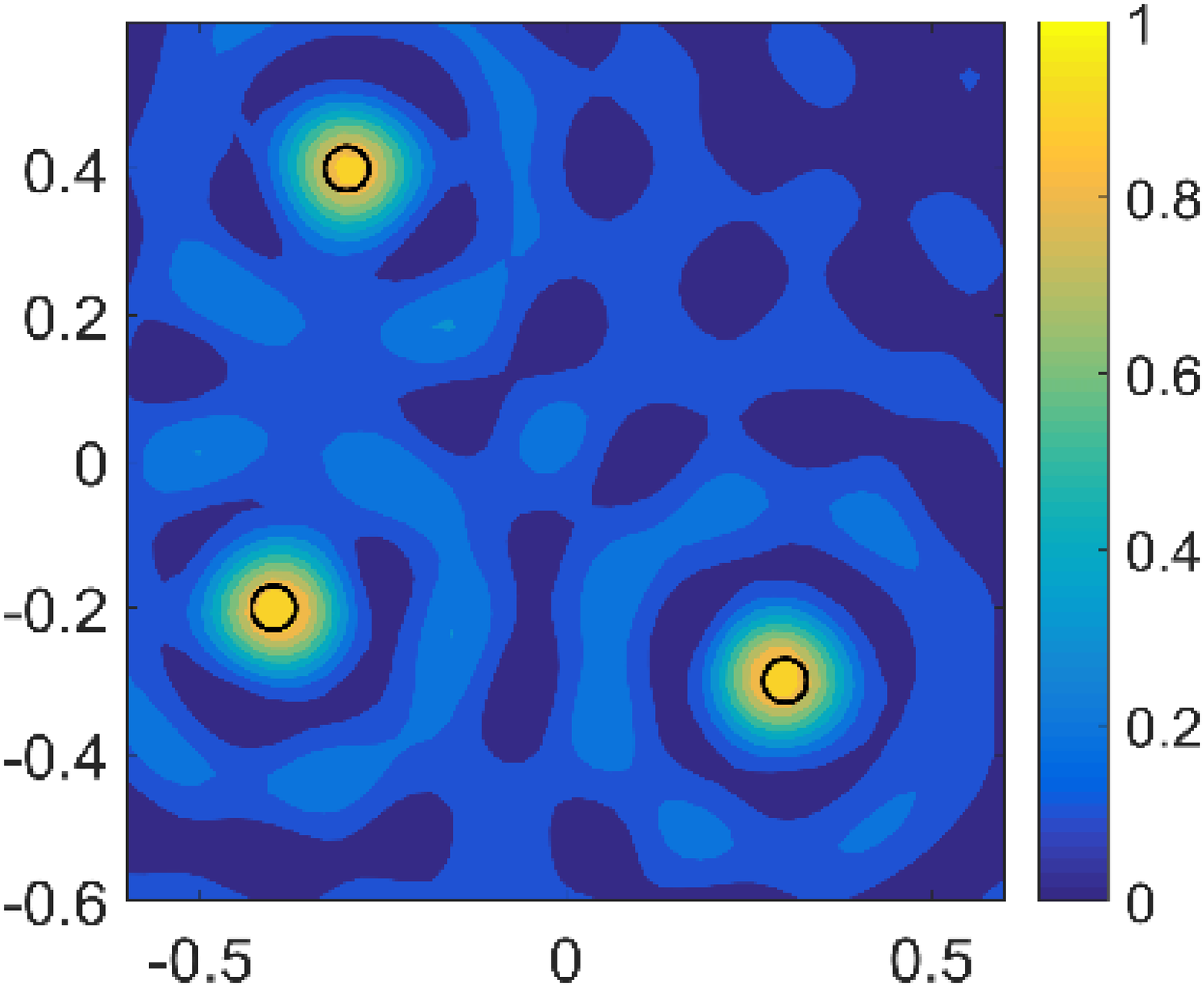}}
	\subfigure[Jaccard index for $L=36$]{\label{Result1-12}\centering\includegraphics[width=0.325\textwidth]{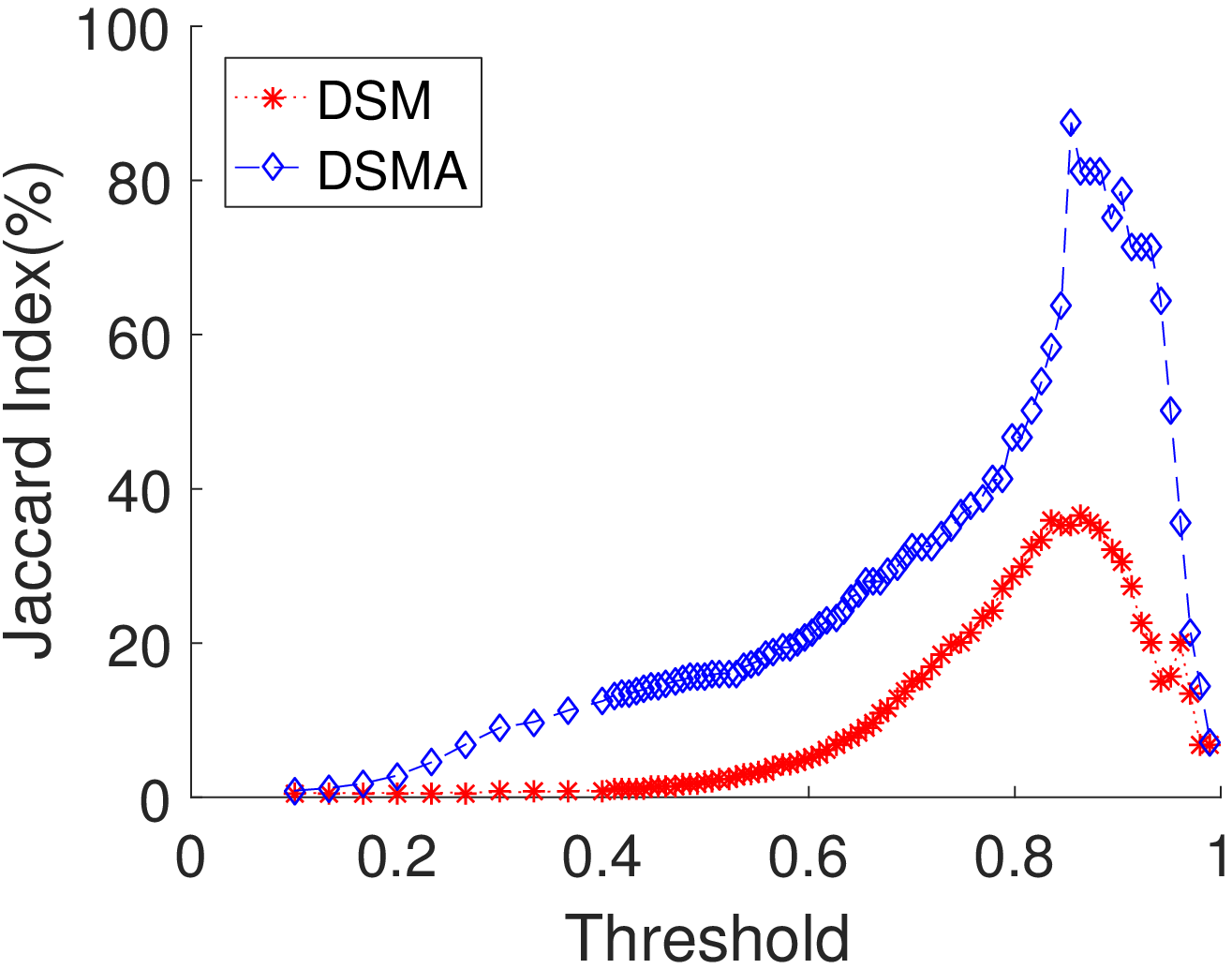}}
	\caption{\label{Result1}(Example~\ref{Example1}) Map of $\mathcal{I}_{\mathrm{DSM}}(\fz)$ (left column) $\mathcal{I}_{\mathrm{DSMA}}(\fz)$ (center column), and Jaccard index (right column).}
\end{figure}

Hereafter Remark~\ref{Remark4}  is verified by comparing $\mathcal{I}_{\mathrm{DSMA}}(\fz)$, and  $\mathcal{I}_{\mathrm{NKM}}(\fz)$ (Figure~\ref{Result2}),  only the maps and the corresponding Jaccard indexes for $L=1$ and $L=36$ incident fields being presented for brevety.  As expected, the maps of $\mathcal{I}_{\mathrm{DSMA}}(\fz)$ and $\mathcal{I}_{\mathrm{NKM}}(\fz)$ and their corresponding Jaccard index are identical whatever the number of incidences. From now on only the Jaccard index of $\mathcal{I}_{\mathrm{NKM}}(\fz)$ will be provided.
\begin{figure}[h]
	\centering
	\subfigure[Map of $\mathcal{I}_{\mathrm{DSMA}}(\fz)$ $L=1$]{\label{Result2-1}\centering\includegraphics[width=0.325\textwidth]{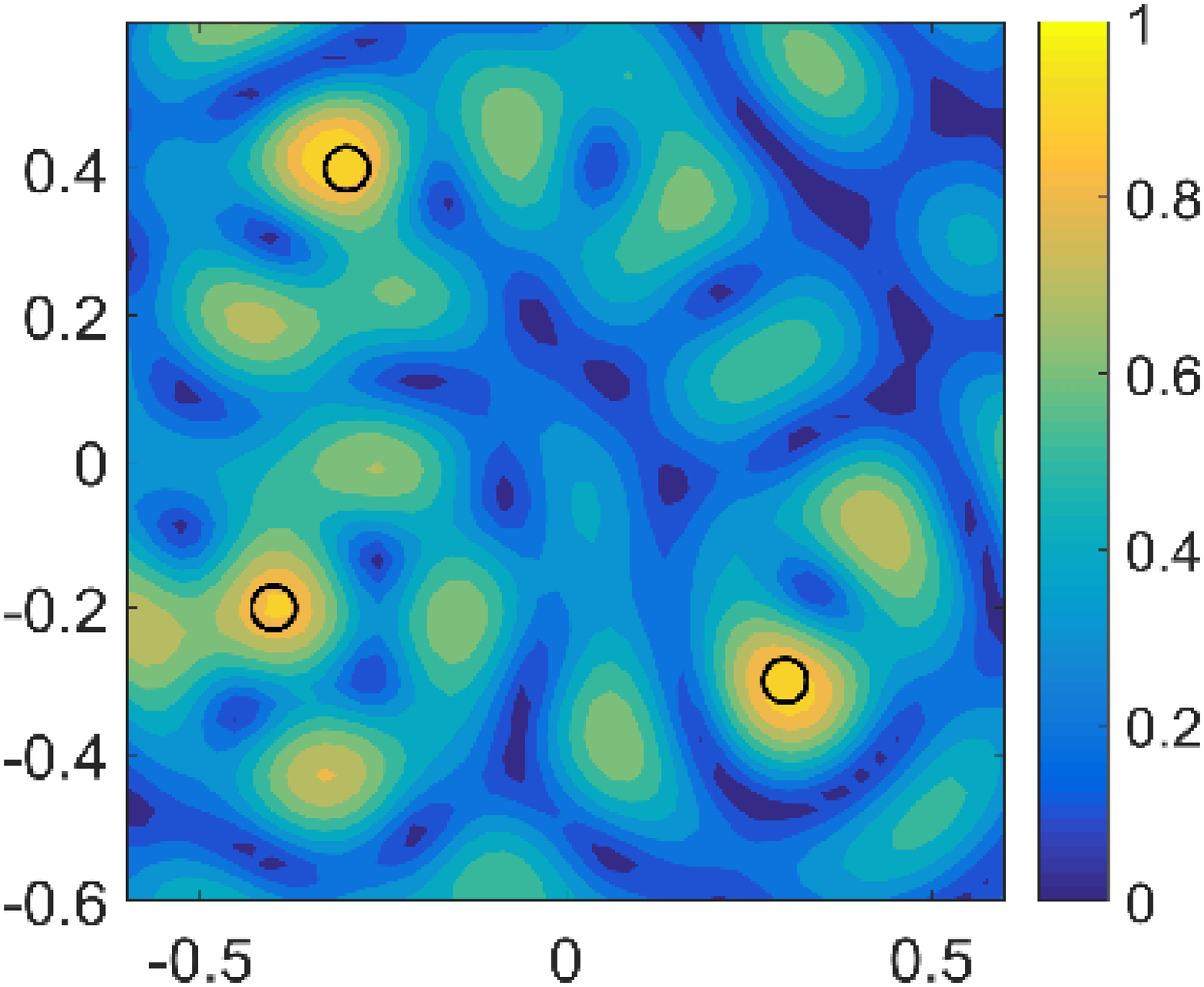}}
	\subfigure[Map of $\mathcal{I}_{\mathrm{NKM}}(\fz)$ for $L=1$]{\label{Result2-2}\centering\includegraphics[width=0.325\textwidth]{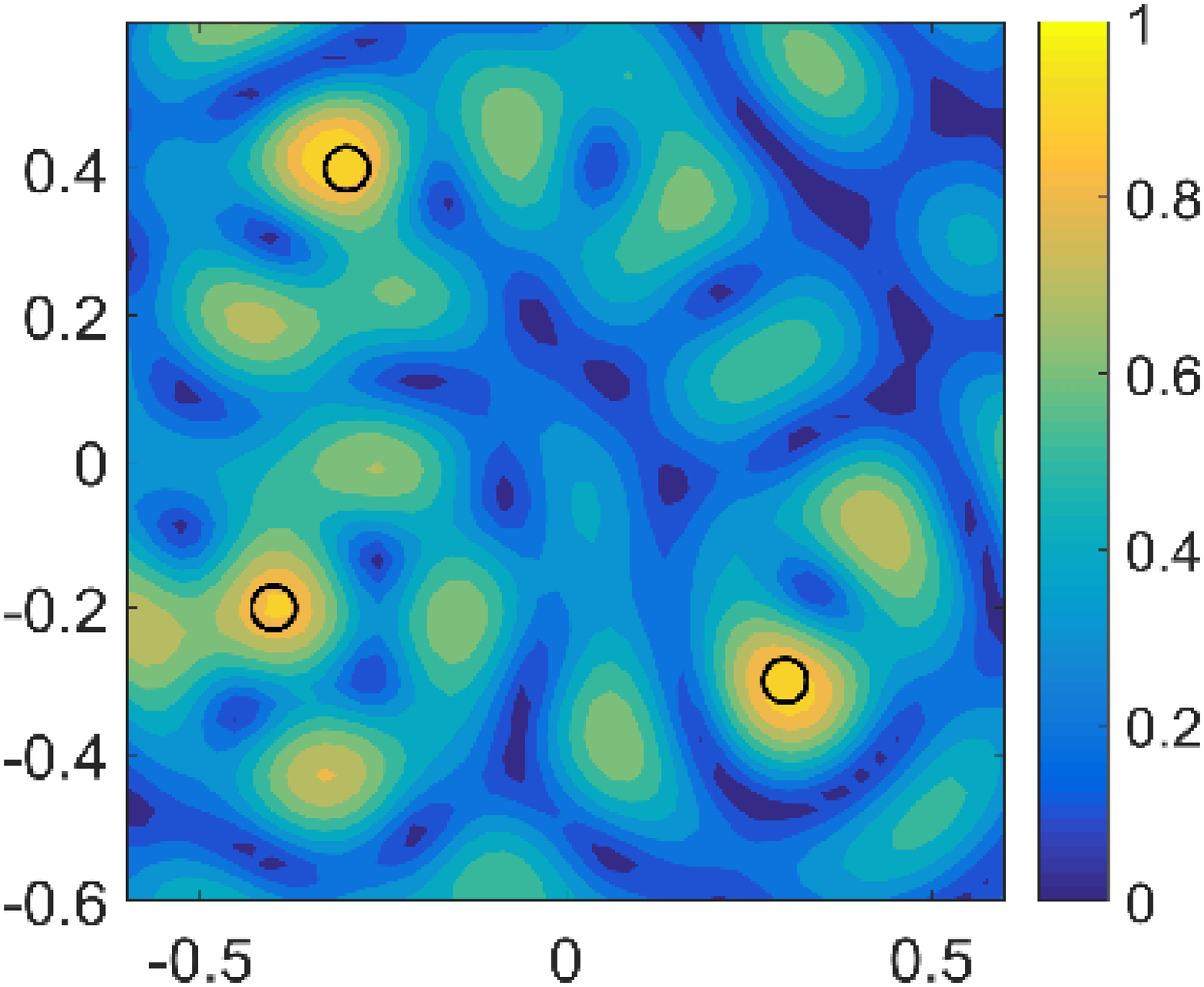}}
	\subfigure[Jaccard index for $L=1$]{\label{Result2-3}\centering\includegraphics[width=0.325\textwidth]{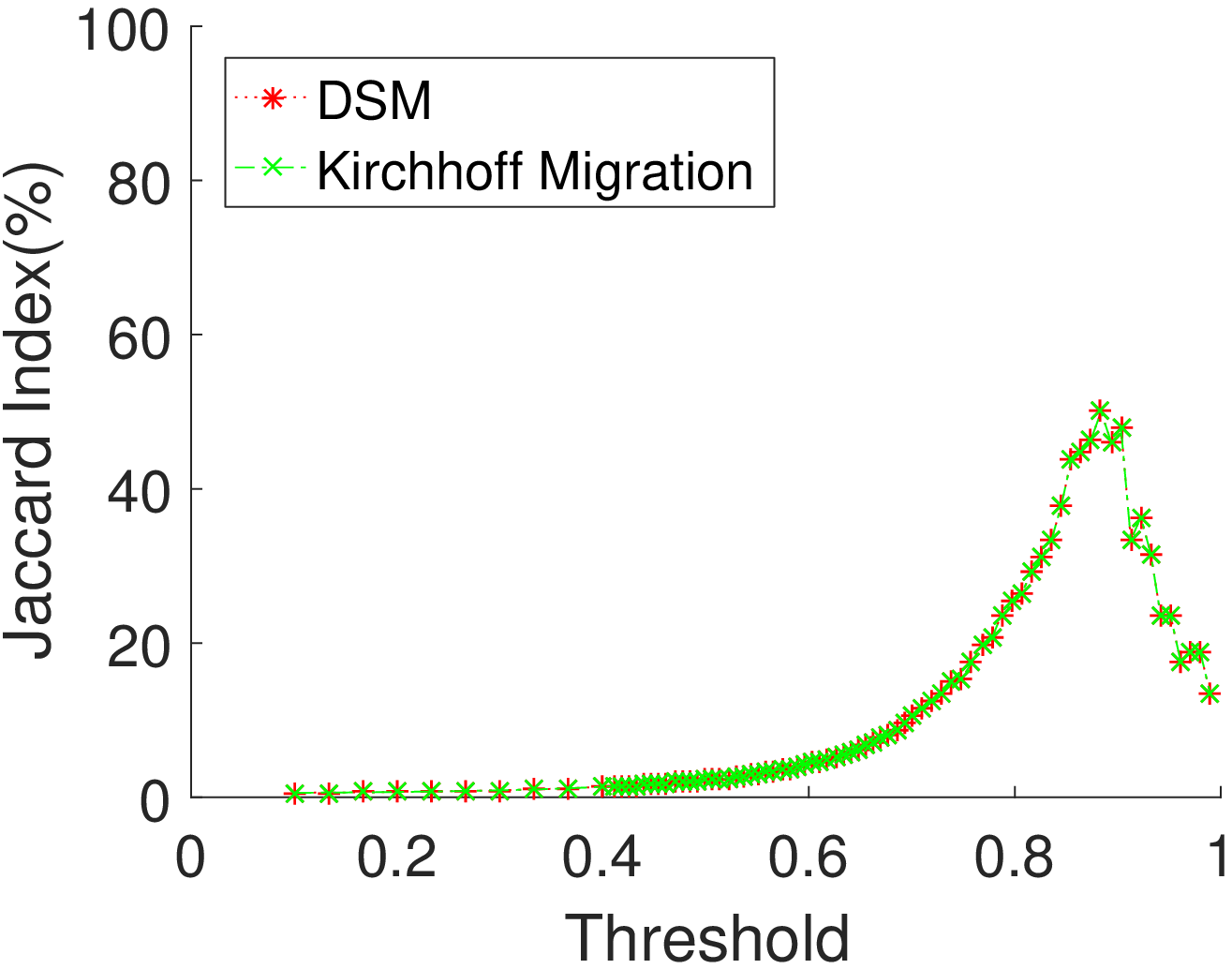}}
	
	\subfigure[Map of $\mathcal{I}_{\mathrm{DSMA}}(\fz)$ for $L=36$]{\label{Result2-4}\centering\includegraphics[width=0.325\textwidth]{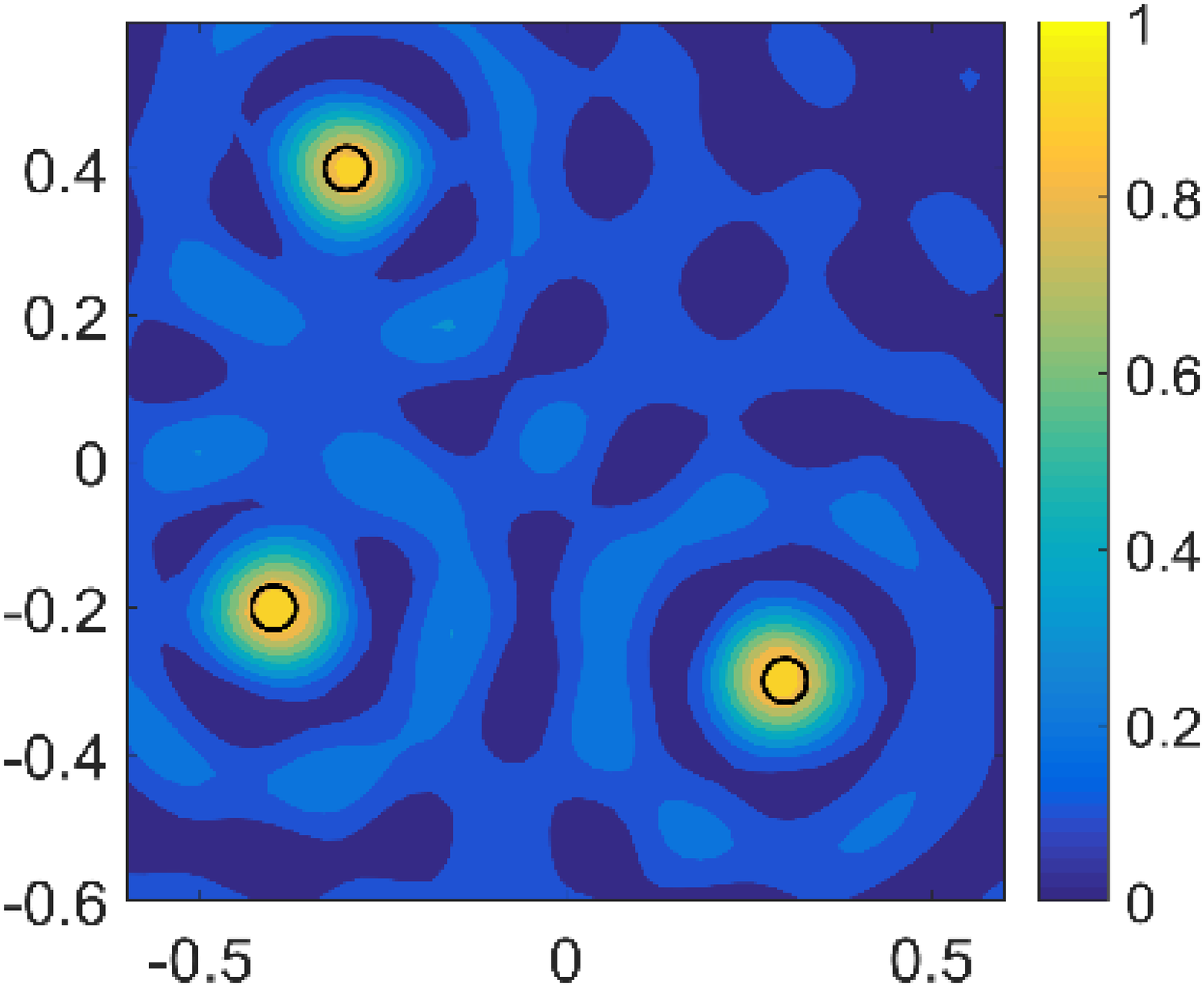}}
	\subfigure[Map of $\mathcal{I}_{\mathrm{NKM}}(\fz)$ for $L=36$]{\label{Result2-5}\centering\includegraphics[width=0.325\textwidth]{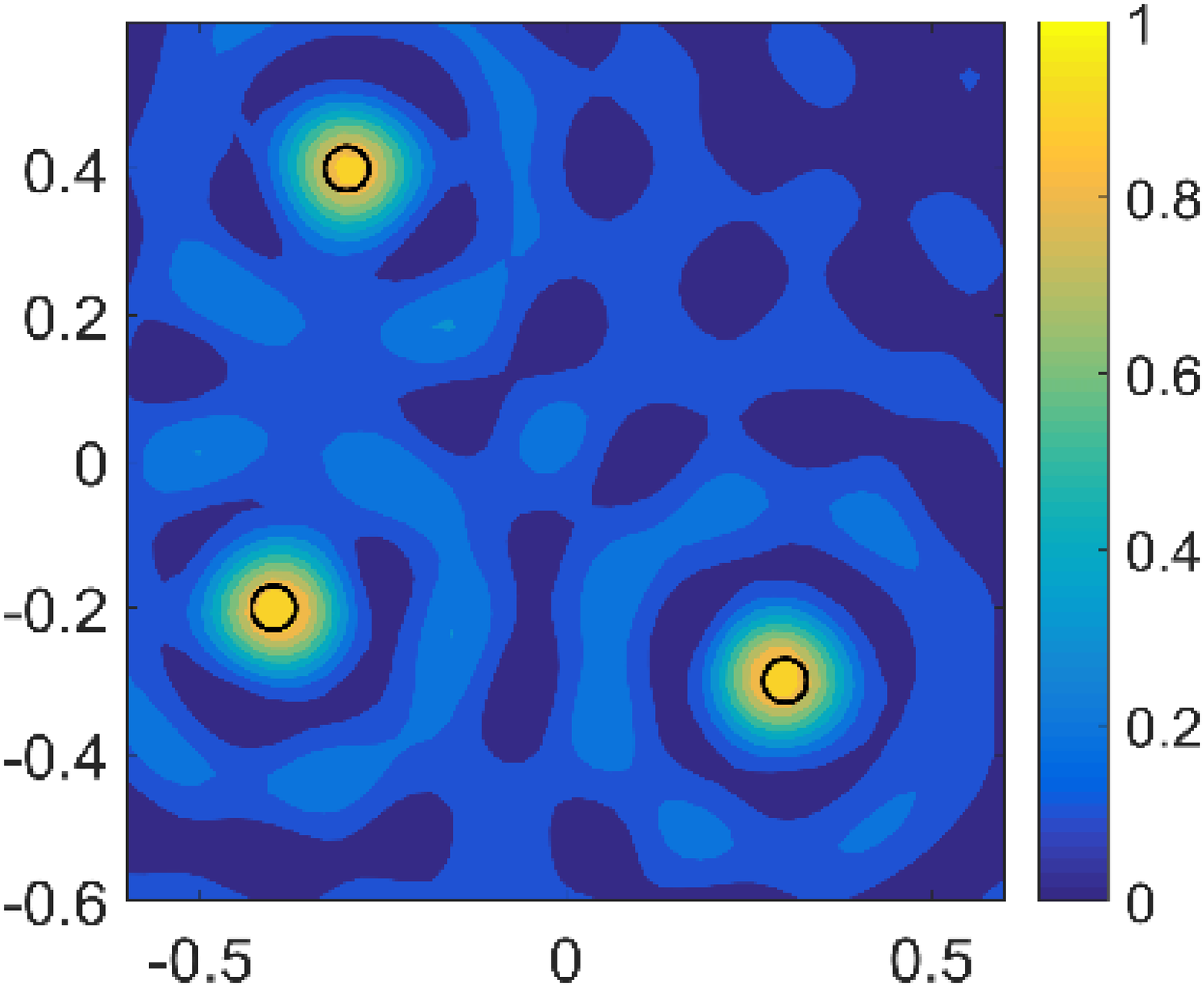}}
	\subfigure[Jaccard index for $L=36$]{\label{Result2-6}\centering\includegraphics[width=0.325\textwidth]{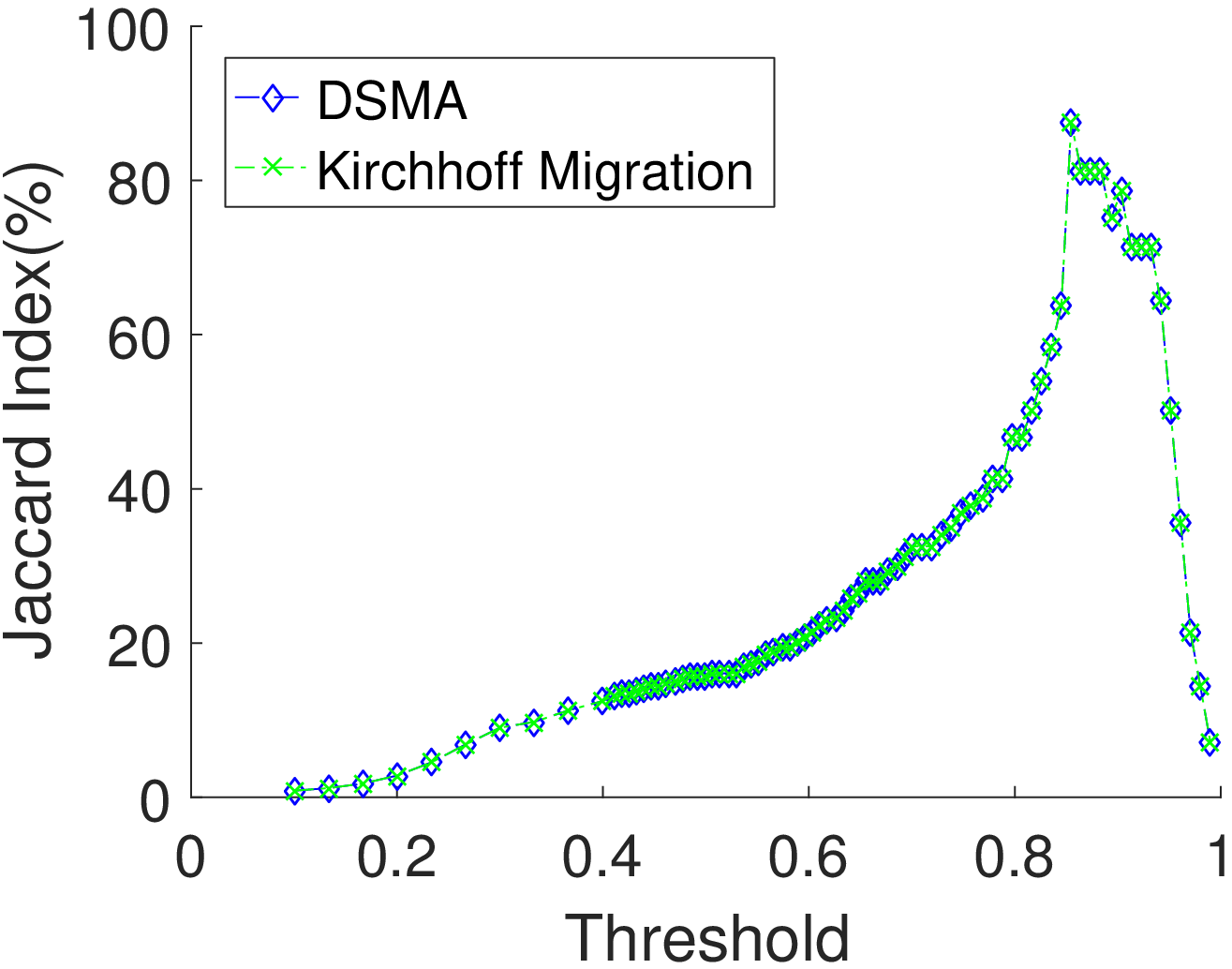}}
	
	\caption{\label{Result2}(Example~\ref{Example1}) Map of $\mathcal{I}_{\mathrm{DSMA}}(\fz)$ (first column),  $\mathcal{I}_{\mathrm{NKM}}(\fz)$ (second column), and Jaccard index (last column).}
\end{figure}

\begin{example}[Small disks with different radii but same permittivities]\label{Example2}	
Now, the imaging of $\tau_m$ with different radii but the same permittivity $\eps_m\equiv5\eps_0$ is dealt with. The values of $\alpha_m$ are $\alpha_1=0.0875\lambda=\SI{0.035}{\meter}$, $\alpha_2=0.075\lambda=\SI{0.03}{\meter}$, and $\alpha_3=0.0625\lambda=\SI{0.025}{\meter}$.
The locations $\fr_m$ of $\tau_m$ are
chosen as $\fr_1=(0.75\lambda,-0.75\lambda)=(\SI{0.3}{\meter},\SI{-0.3}{\meter})$, $\fr_2=(-\lambda,-0.5\lambda)=(\SI{-0.4}{\meter},\SI{-0.2}{\meter})$, and $\fr_3=(-0.75\lambda,\lambda)=(\SI{-0.3}{\meter},\SI{0.4}{\meter})$.
 
\end{example}
As illustrated in Figure \ref{Result3},  when using the original DSM (Figure \ref{Result3}, left column) the localization of the inhomogeneity with the largest radius $\left(\tau_1\right)$  is well identified whereas the others ($\tau_2$ and $\tau_3$) are not. Even when the number of sources is increased the localization of the inhomogeneity $\tau_{3}$ is still difficult to identify due to the presence of important artifacts in the image.  The use of DSMA (Figure \ref{Result3}, centered column) improves the quality of the image thanks to the smoothing of the artifacts. This illustrates the statement proposed in Theorem~\ref{DSM_case1} and the discussions in Remark~\ref{Remark1} and Remark \ref{Remark3}.   

\begin{figure}[h]
	\centering
	\subfigure[Map of $\mathcal{I}_{\mathrm{DSM}}(\fz)$ for $L=1$]{\label{Result3-1}\centering\includegraphics[width=0.325\textwidth]{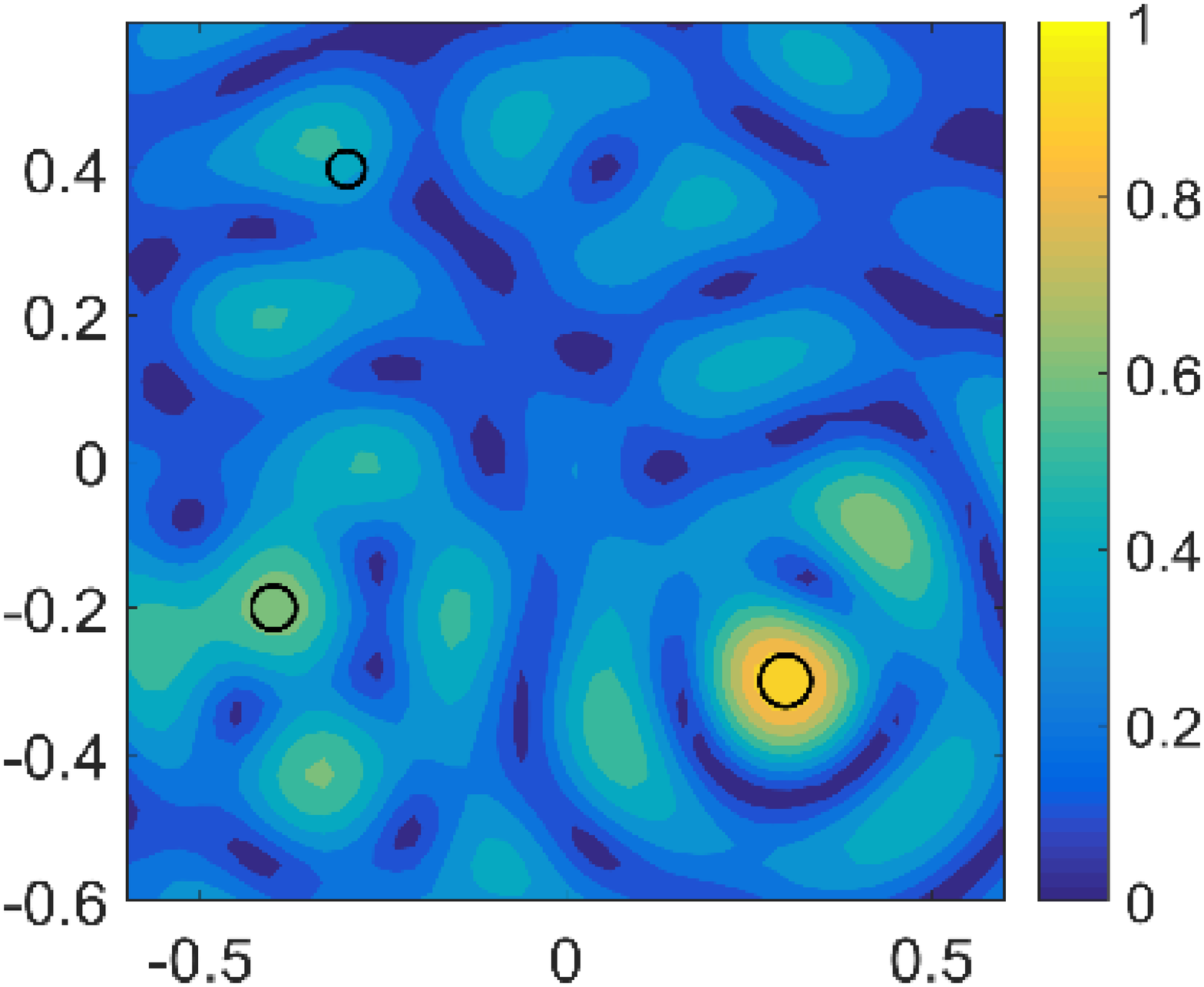}}
	\subfigure[Map of $\mathcal{I}_{\mathrm{DSMA}}(\fz)$ for $L=1$]{\label{Result3-2}\centering\includegraphics[width=0.325\textwidth]{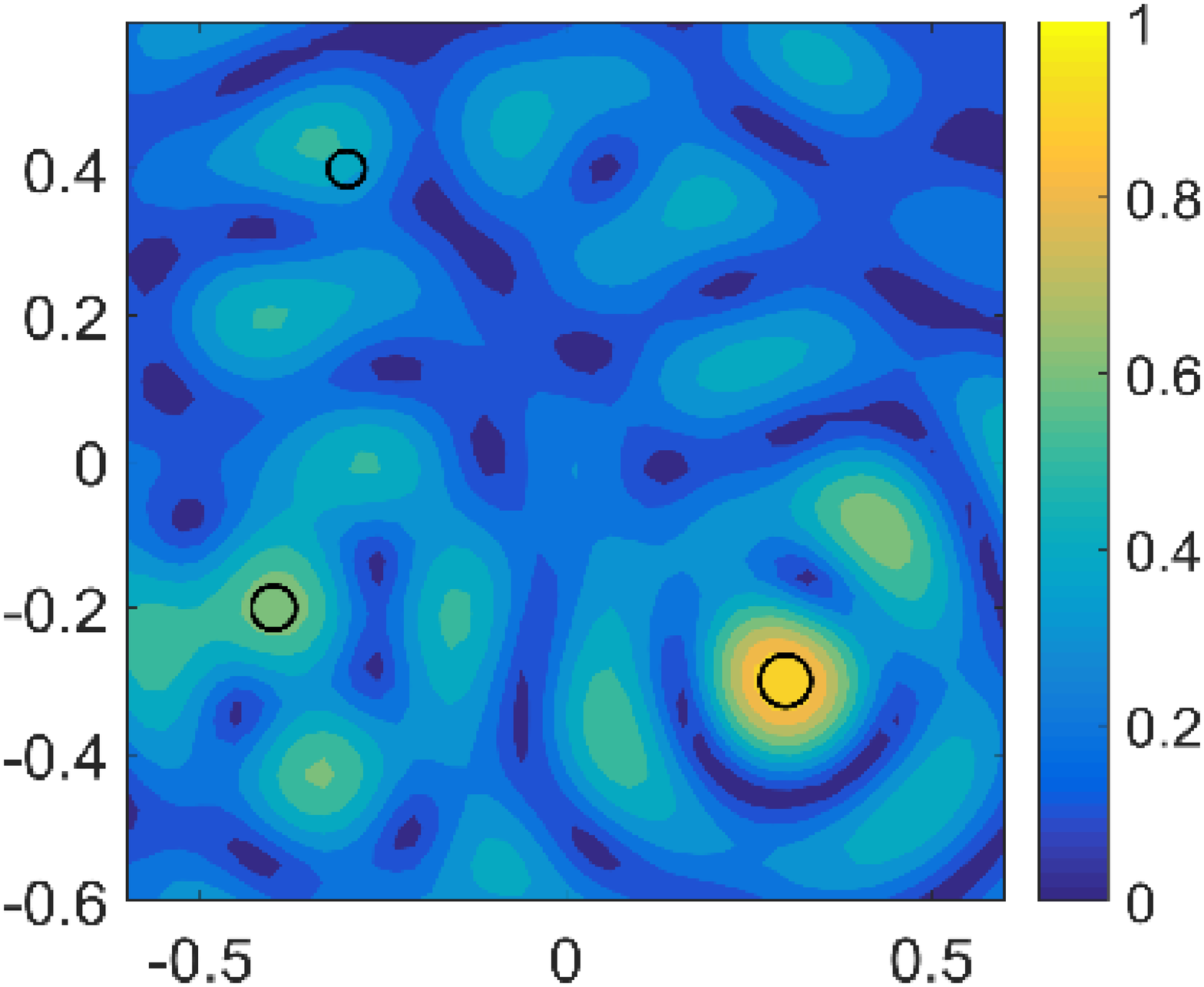}}
	\subfigure[Jaccard index for $L=1$]{\label{Result3-3}\centering\includegraphics[width=0.325\textwidth]{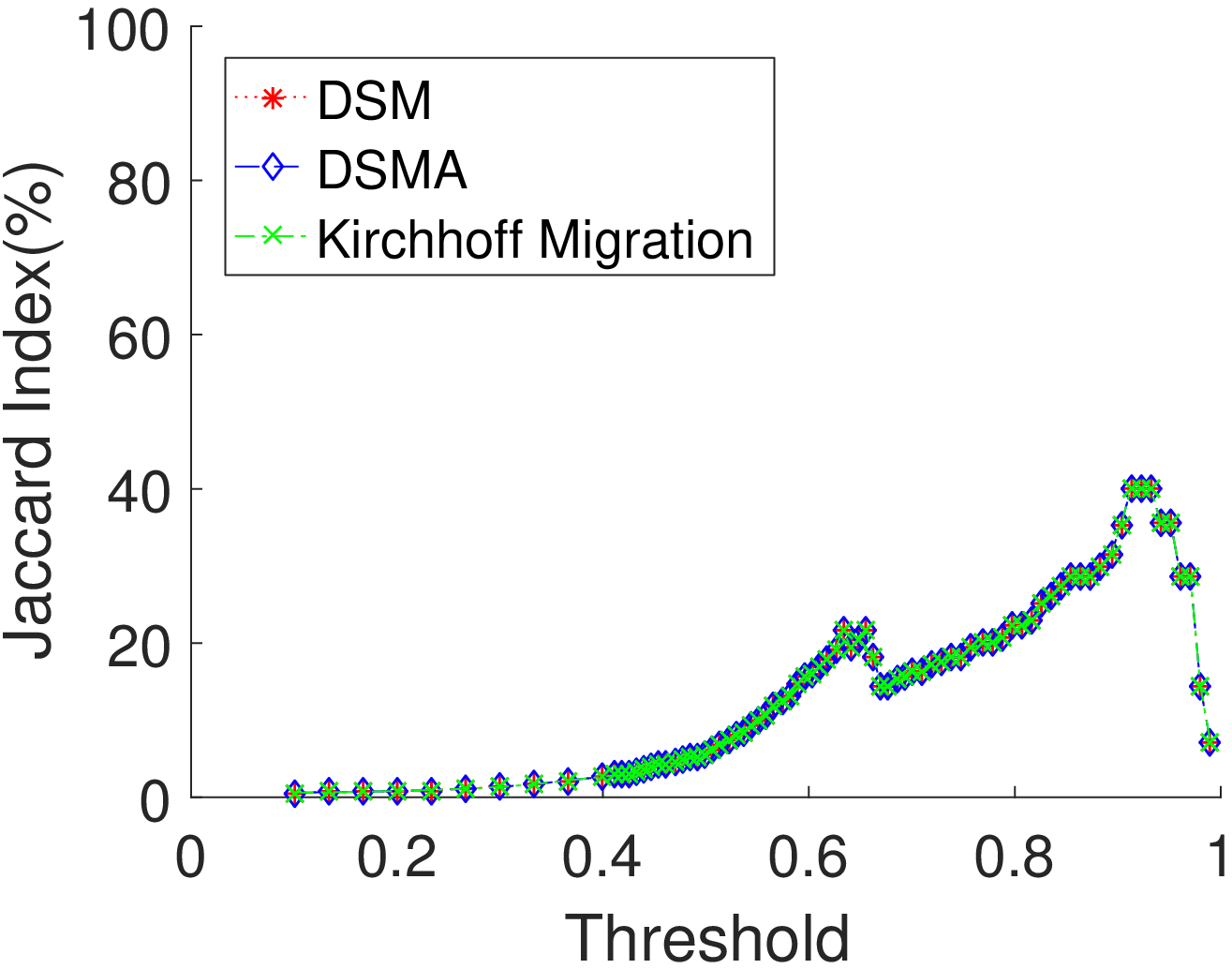}}
	
	\subfigure[Map of $\mathcal{I}_{\mathrm{DSM}}(\fz)$ for $L=2$]{\label{Result3-4}\centering\includegraphics[width=0.325\textwidth]{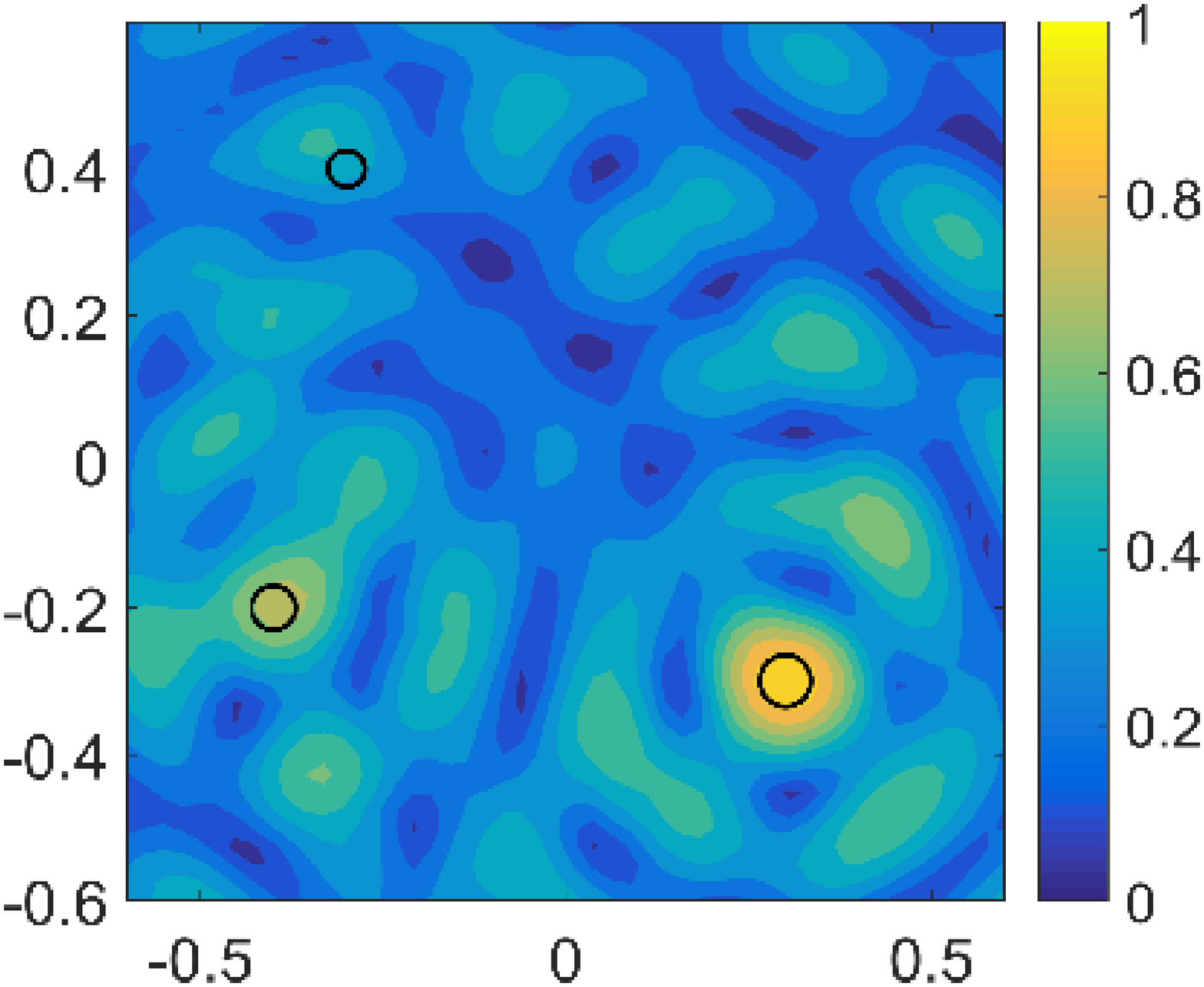}}
	\subfigure[Map of $\mathcal{I}_{\mathrm{DSMA}}(\fz)$ for  $L=2$]{\label{Result3-5}\centering\includegraphics[width=0.325\textwidth]{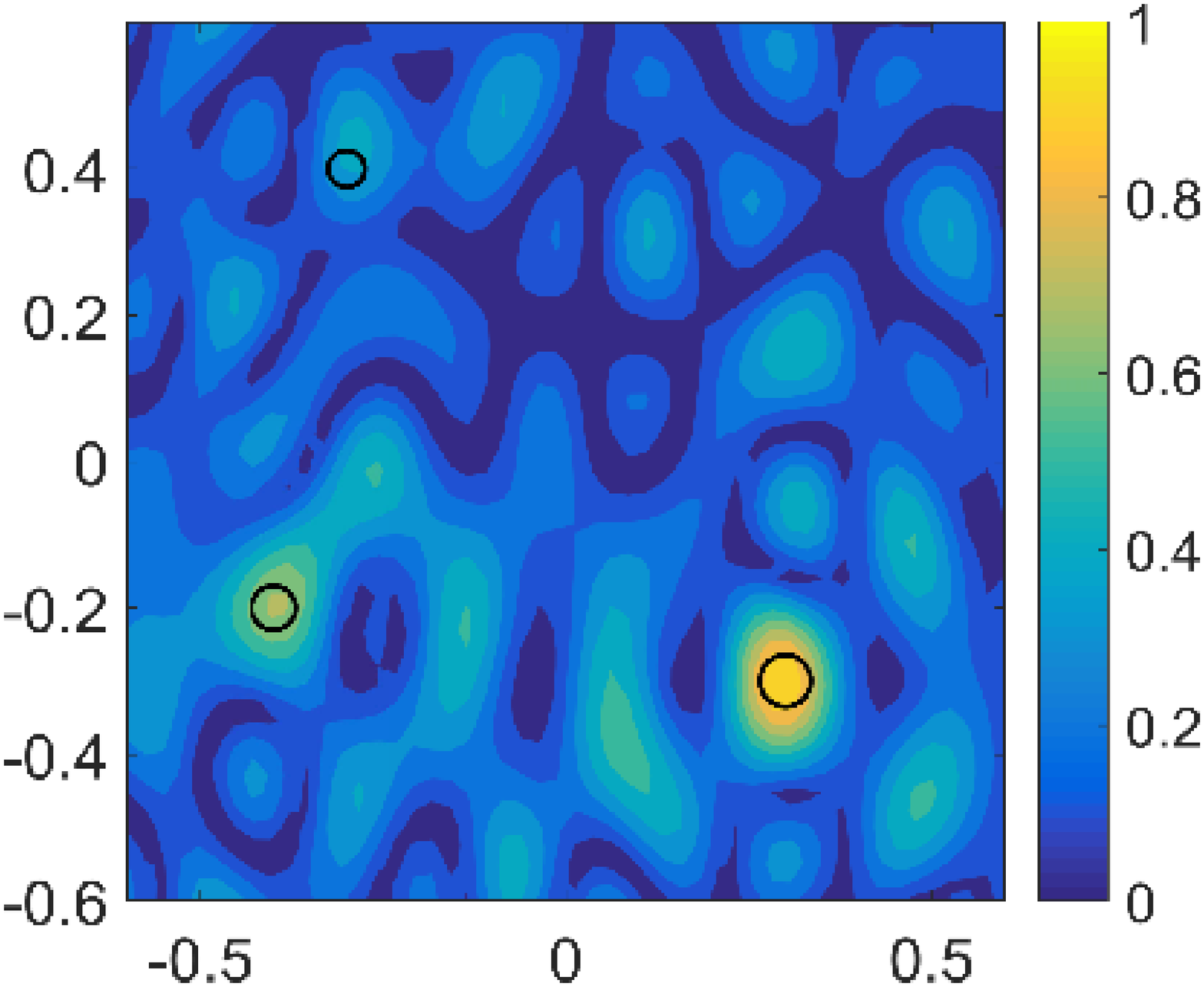}}
	\subfigure[Jaccard index for $L=2$]{\label{Result3-6}\centering\includegraphics[width=0.325\textwidth]{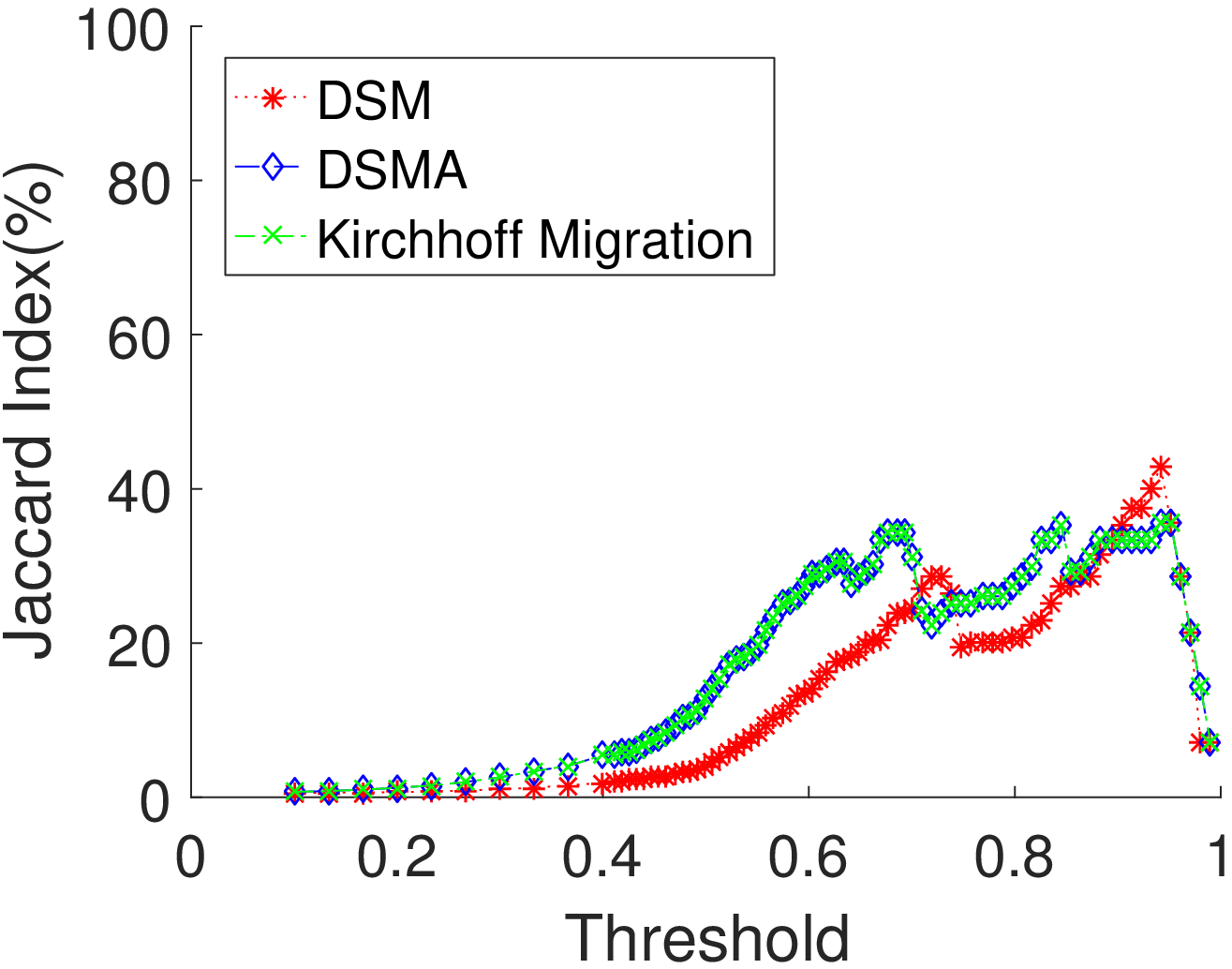}}
	
	\subfigure[Map of $\mathcal{I}_{\mathrm{DSM}}(\fz)$ for $L=12$]{\label{Result3-7}\centering\includegraphics[width=0.325\textwidth]{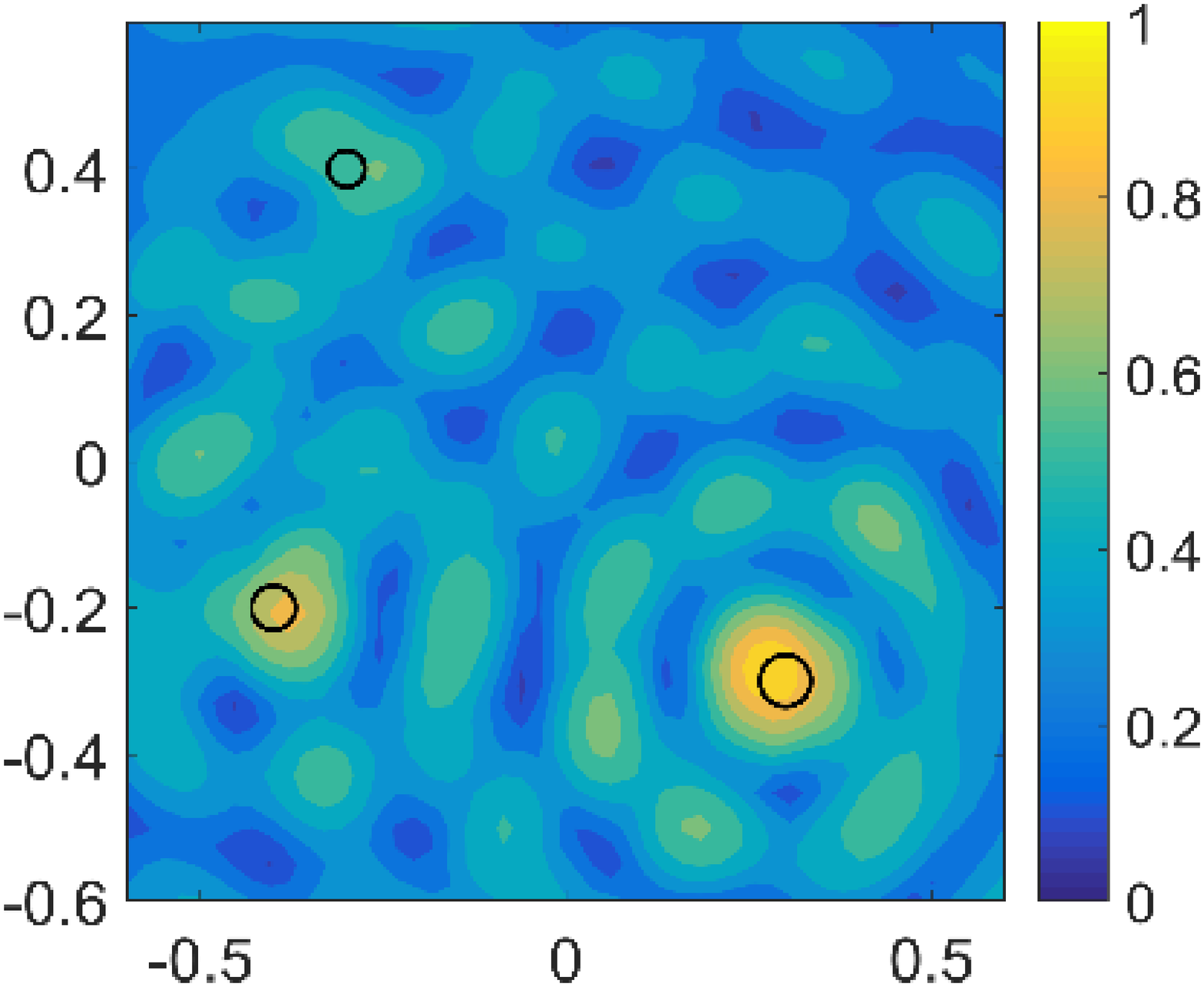}}
	\subfigure[Map of $\mathcal{I}_{\mathrm{DSMA}}(\fz)$ for $L=12$]{\label{Result3-8}\centering\includegraphics[width=0.325\textwidth]{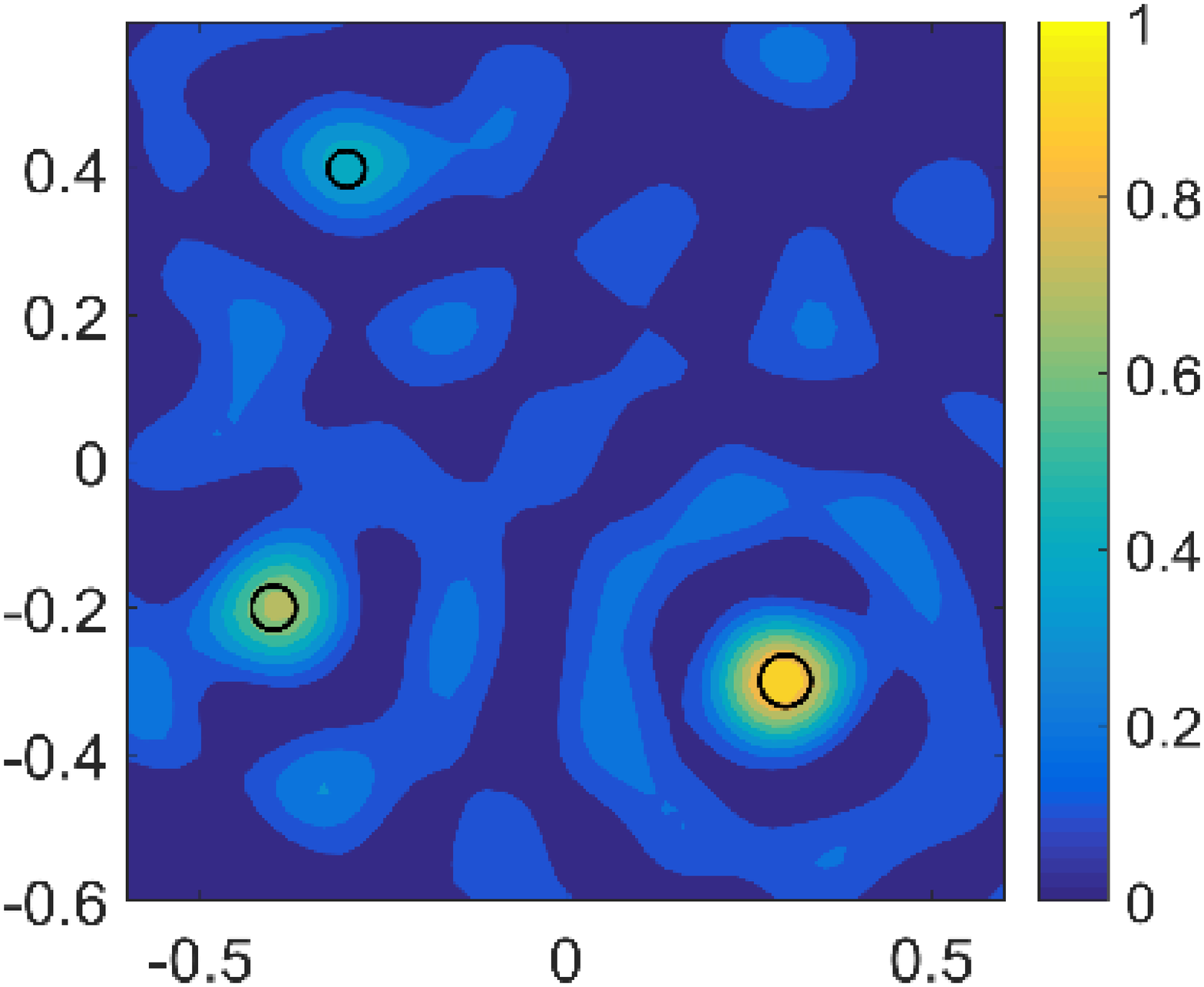}}
	\subfigure[Jaccard index for $L=12$]{\label{Result3-9}\centering\includegraphics[width=0.325\textwidth]{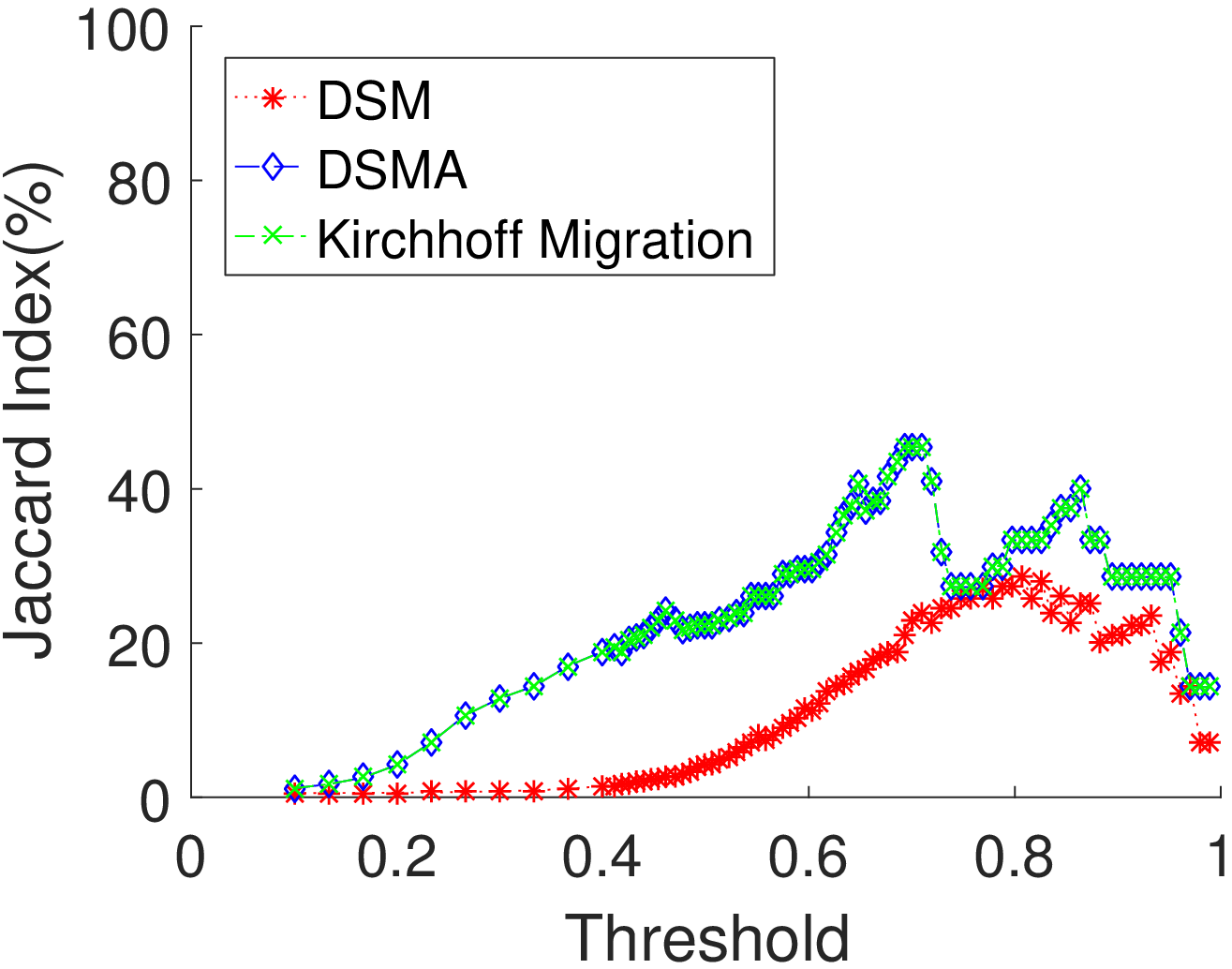}}
	
	\subfigure[Map of $\mathcal{I}_{\mathrm{DSM}}(\fz)$ for $L=36$]{\label{Result3-10}\centering\includegraphics[width=0.325\textwidth]{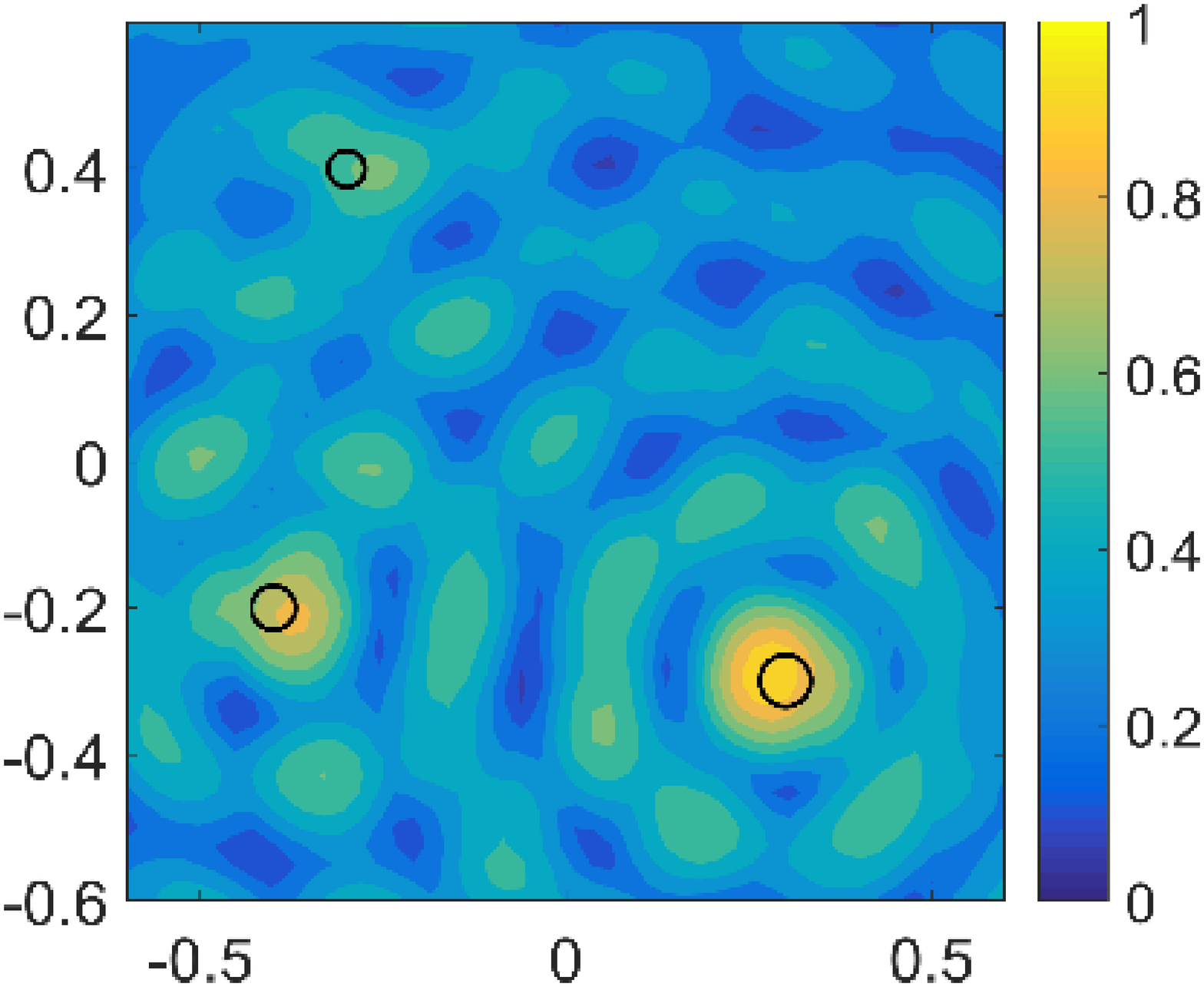}}
	\subfigure[Map of $\mathcal{I}_{\mathrm{DSMA}}(\fz)$ for $L=36$]{\label{Result3-11}\centering\includegraphics[width=0.325\textwidth]{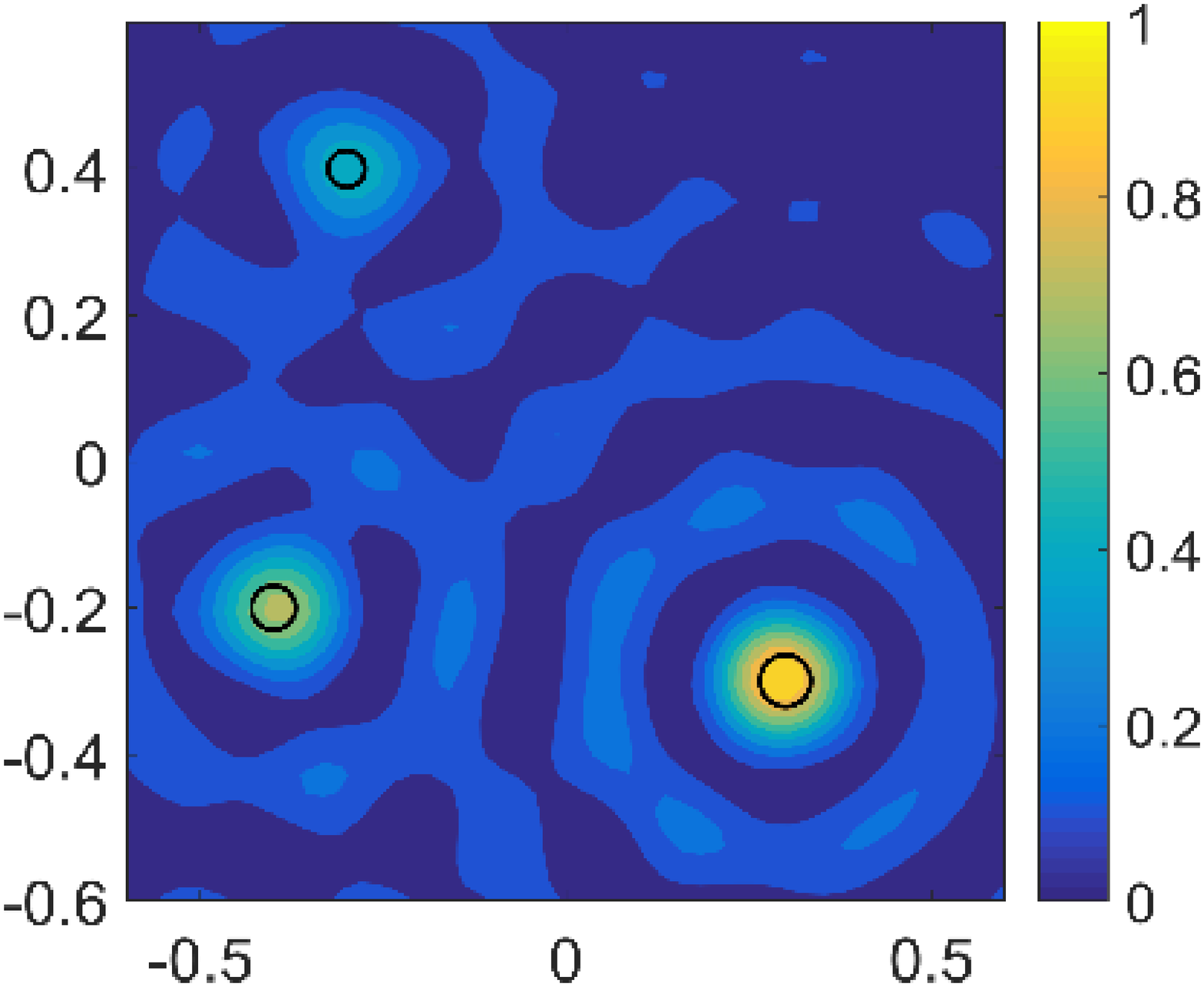}}
	\subfigure[Jaccard index for $L=36$]{\label{Result3-12}\centering\includegraphics[width=0.325\textwidth]{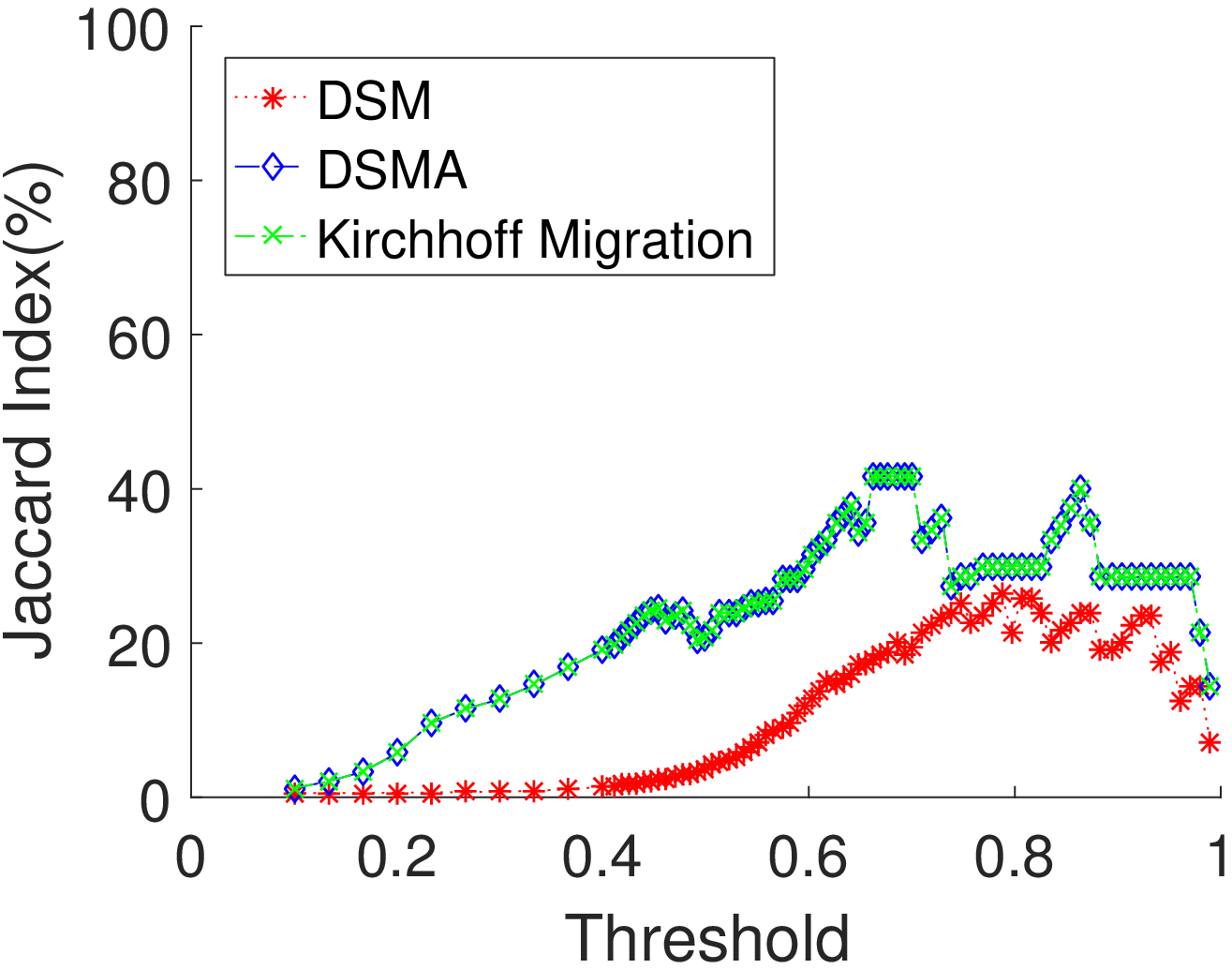}}
	
	\caption{\label{Result3}(Example~\ref{Example2}) Map of $\mathcal{I}_{\mathrm{DSM}}(\fz)$ (left column) $\mathcal{I}_{\mathrm{DSMA}}(\fz)$ (center column), and Jaccard index (right column).}
\end{figure}

Note that some numerical experiments (not presented in the following) have been done in the closely related case of small disks with the same radius but different permittivities and  the conclusions are the same, DSMA performs better than DSM when the number of incident fields increases. 

\begin{example}[Large disk]\label{Example4}
	In order to verify that our proposal still behaves properly when the small obstacle hypothesis is no longer verified, we are considering the case of a  large circular  single  inhomogeneity $\tau$ with radius $\alpha\equiv1\lambda=\SI{0.4}{\meter}$ and permittivity $\eps=5\eps_0$. A location is chosen as $\fr=(-0.75\lambda,-0.75\lambda)=(\SI{-0.3}{\meter},\SI{-0.3}{\meter})$. In this example, the search domain $\Omega_\Gamma$ is a square with side of $2.5\lambda(=\SI{1}{\meter})$, which is divided into small squares with side $h=0.102\lam=\SI{0.0408}{\meter}$.
\end{example}

According to Figure \ref{Result5},  the exact location and shape of $\tau$ with a few incident waves (one or two) is difficult to obtain both with  DSM and DSMA. But, as the number of incident waves increases, the image of $\tau$ is improving with  $\mathcal{I}_{\mathrm{DSM}}$ and with $\mathcal{I}_{\mathrm{DSMA}}$. From the Jaccard index it can be seen that DSMA has better performance than DSM even if we are no more within the small obstacle hypothesis.  

\begin{figure}[h]
	\centering
	\subfigure[Map of $\mathcal{I}_{\mathrm{DSM}}(\fz)$ for $L=1$]{\label{Result5-1}\centering\includegraphics[width=0.325\textwidth]{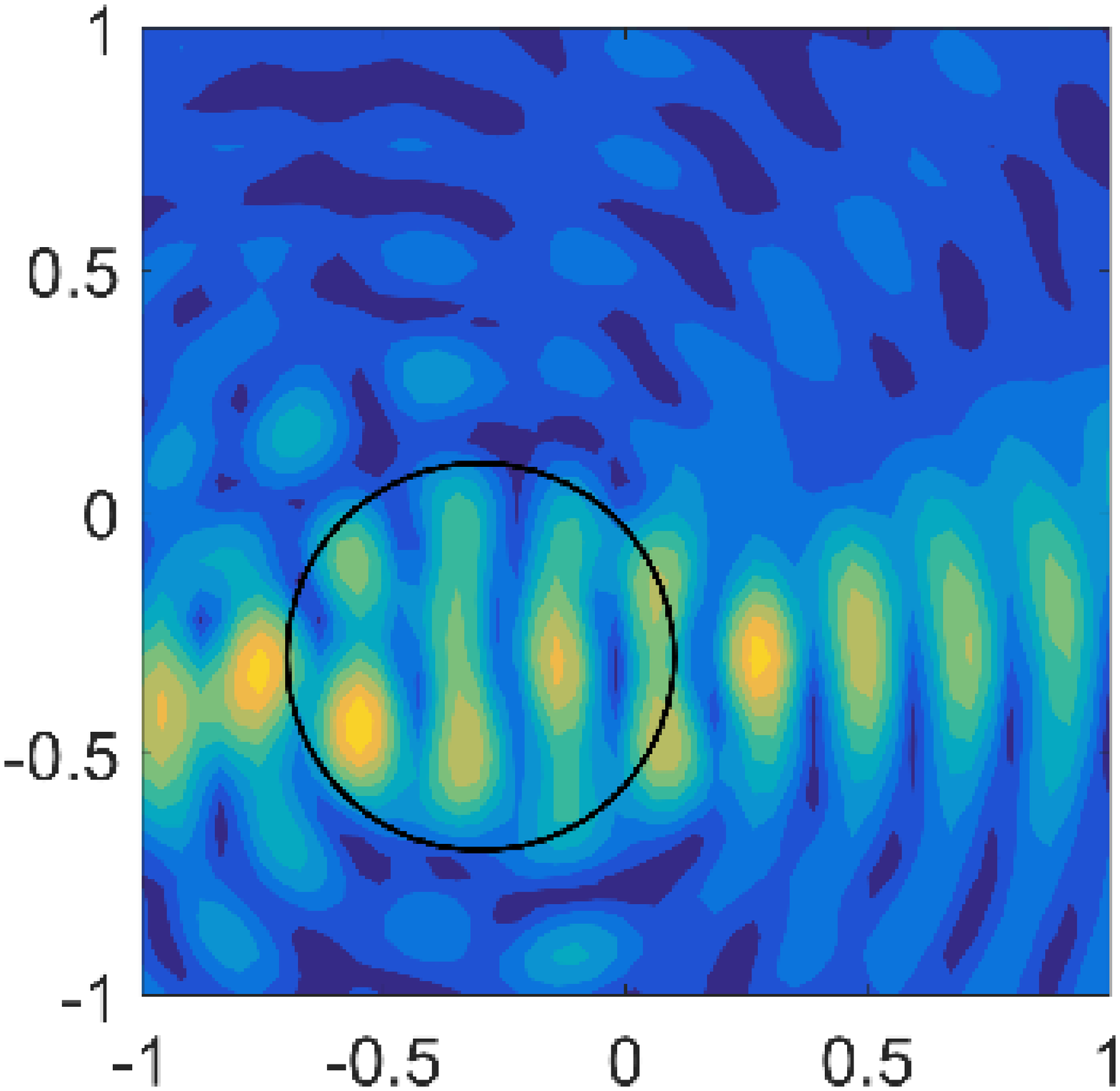}}
	\subfigure[Map of $\mathcal{I}_{\mathrm{DSMA}}(\fz)$ for $L=1$]{\label{Result5-2}\centering\includegraphics[width=0.325\textwidth]{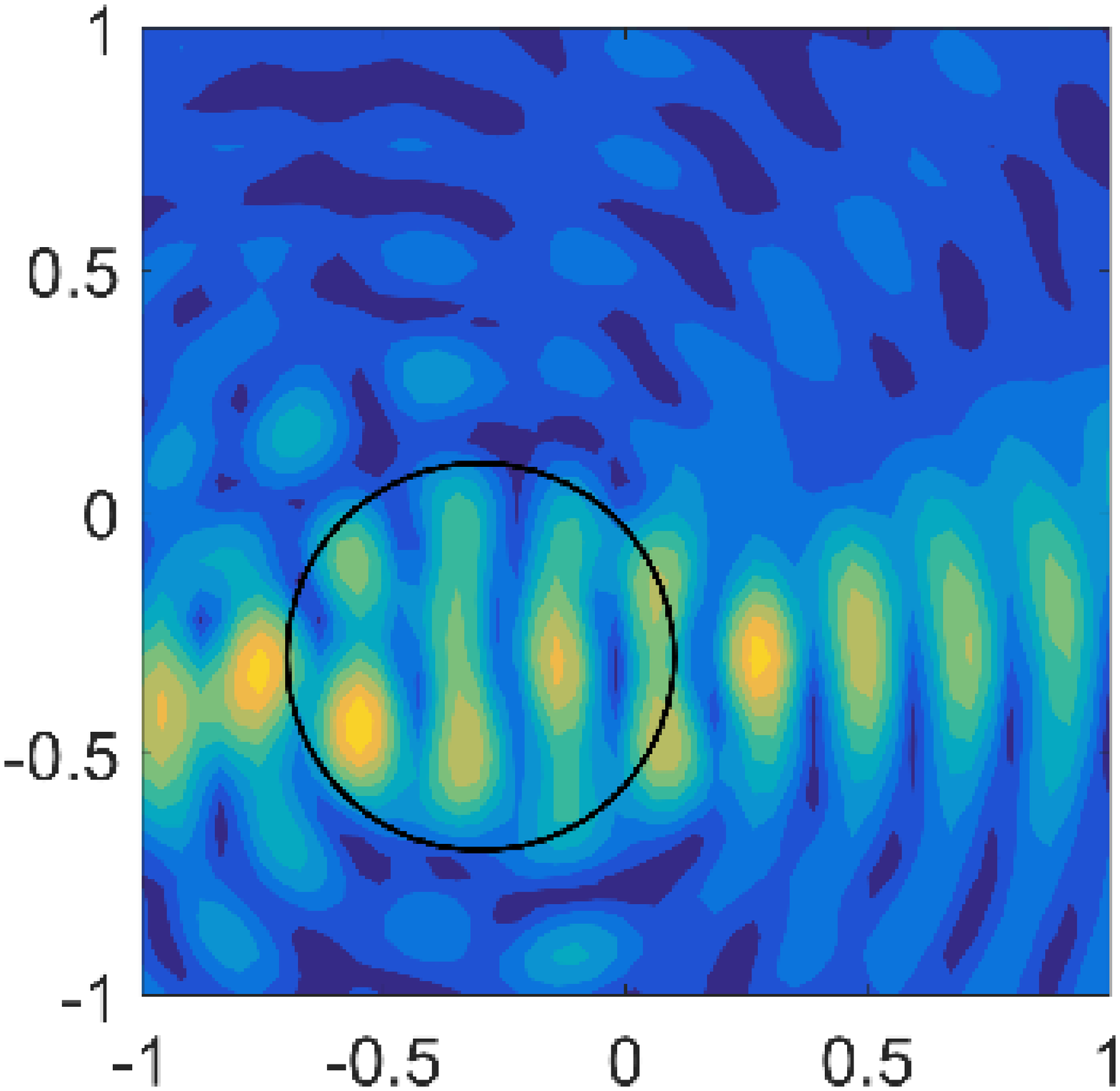}}
	\subfigure[Jaccard index for $L=1$]{\label{Result5-3}\centering\includegraphics[width=0.325\textwidth]{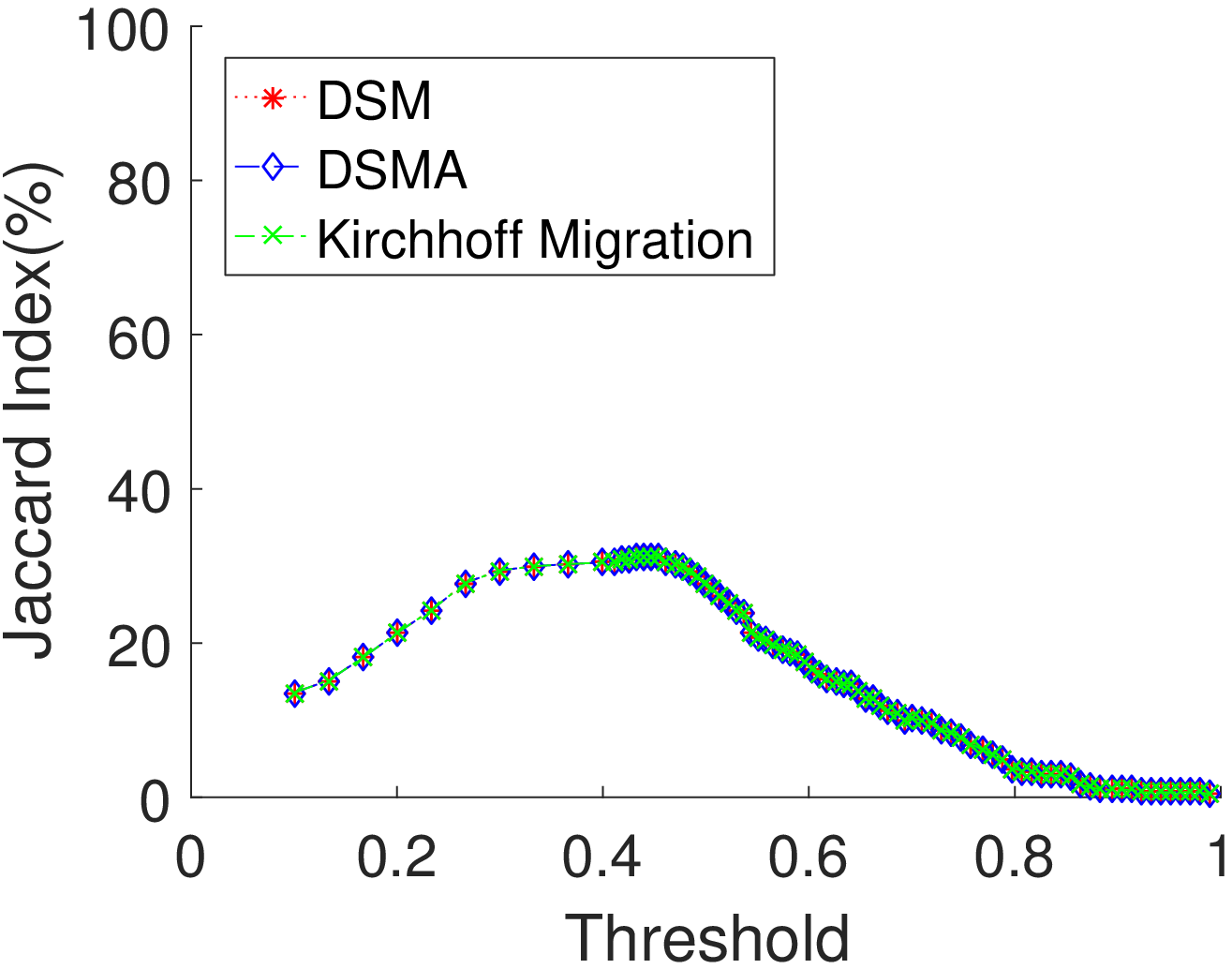}}
	
	\subfigure[Map of $\mathcal{I}_{\mathrm{DSM}}(\fz)$ for $L=2$]{\label{Result5-4}\centering\includegraphics[width=0.325\textwidth]{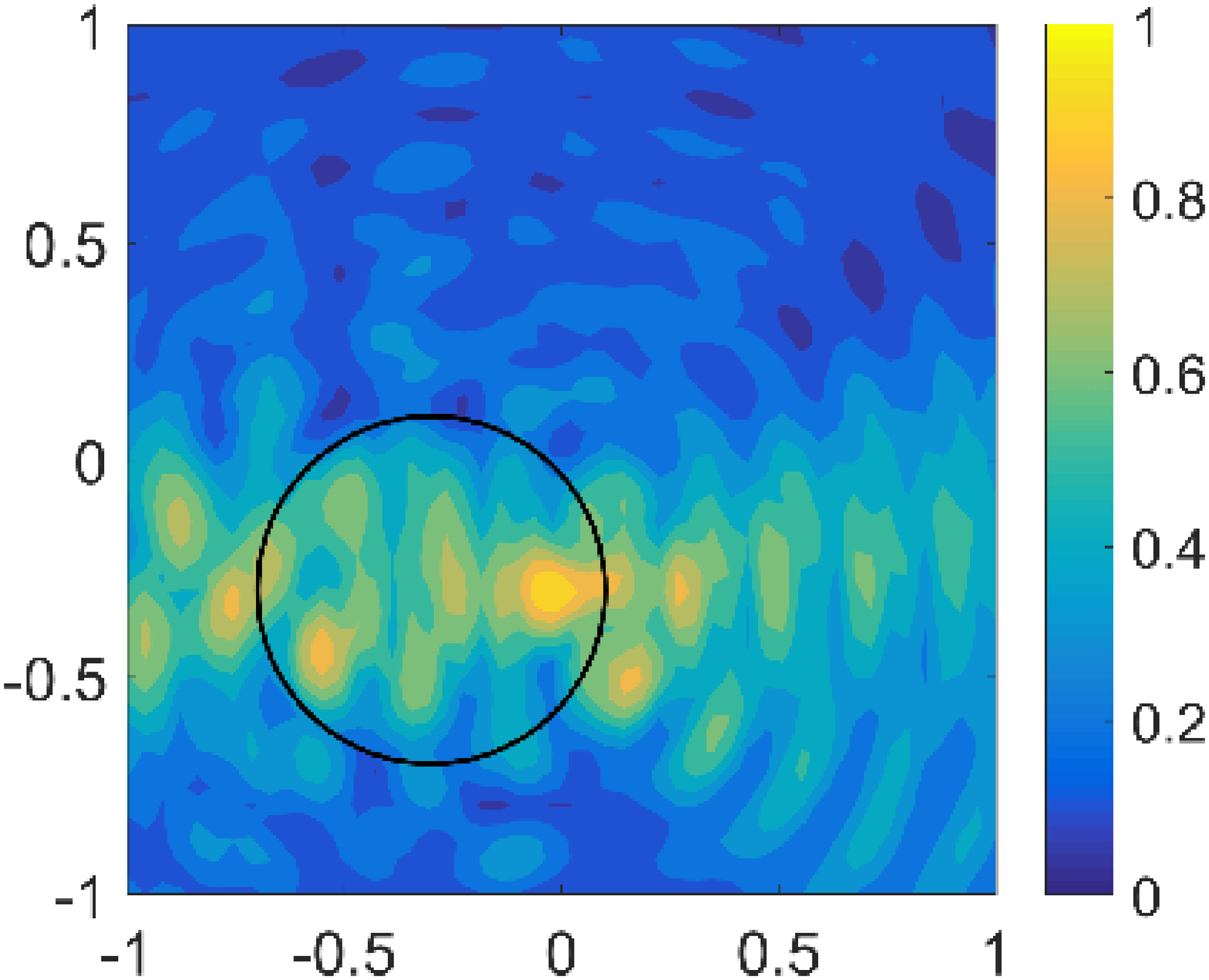}}
	\subfigure[Map of $\mathcal{I}_{\mathrm{DSMA}}(\fz)$ for  $L=2$]{\label{Result5-5}\centering\includegraphics[width=0.325\textwidth]{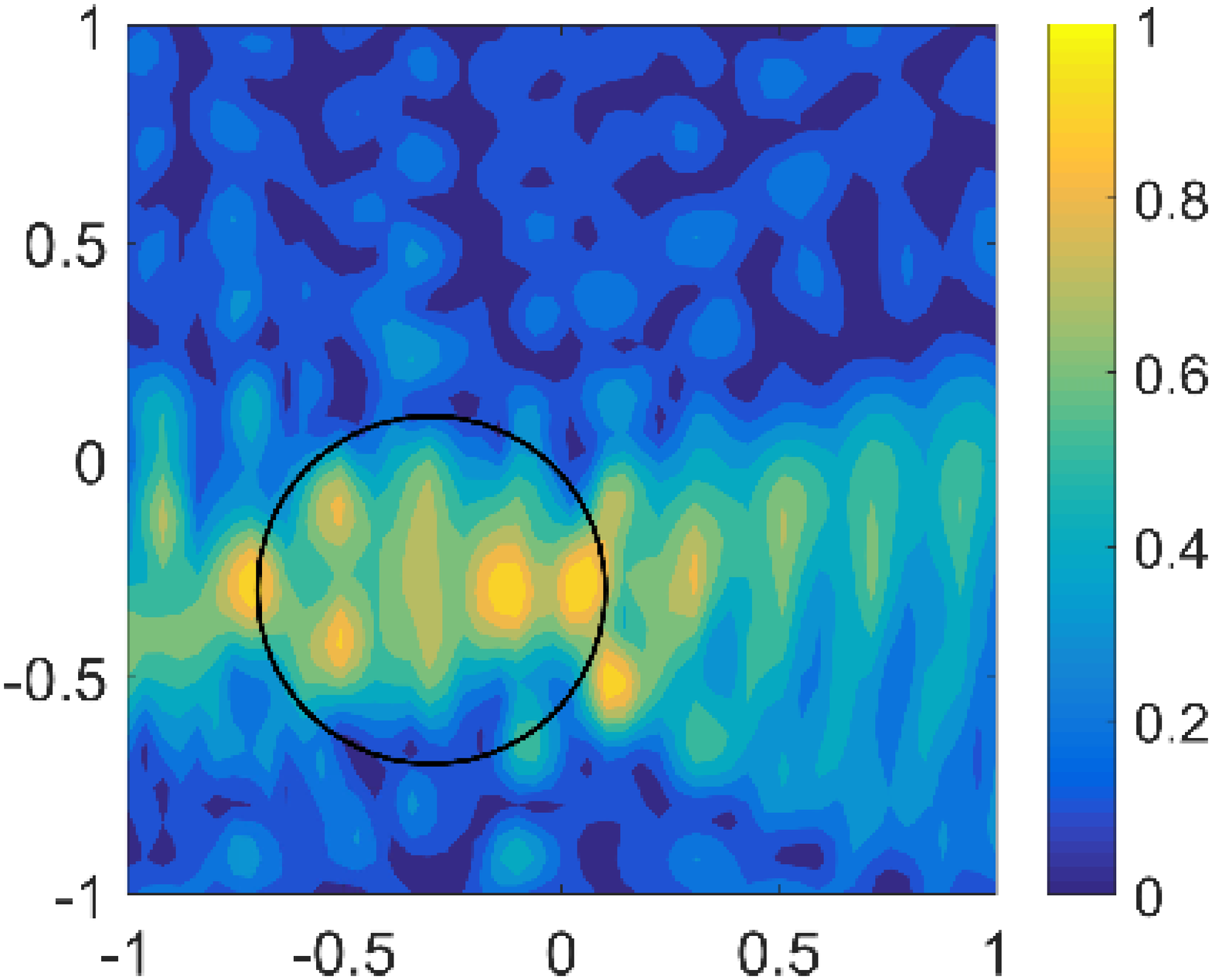}}
	\subfigure[Jaccard index for $L=2$]{\label{Result5-6}\centering\includegraphics[width=0.325\textwidth]{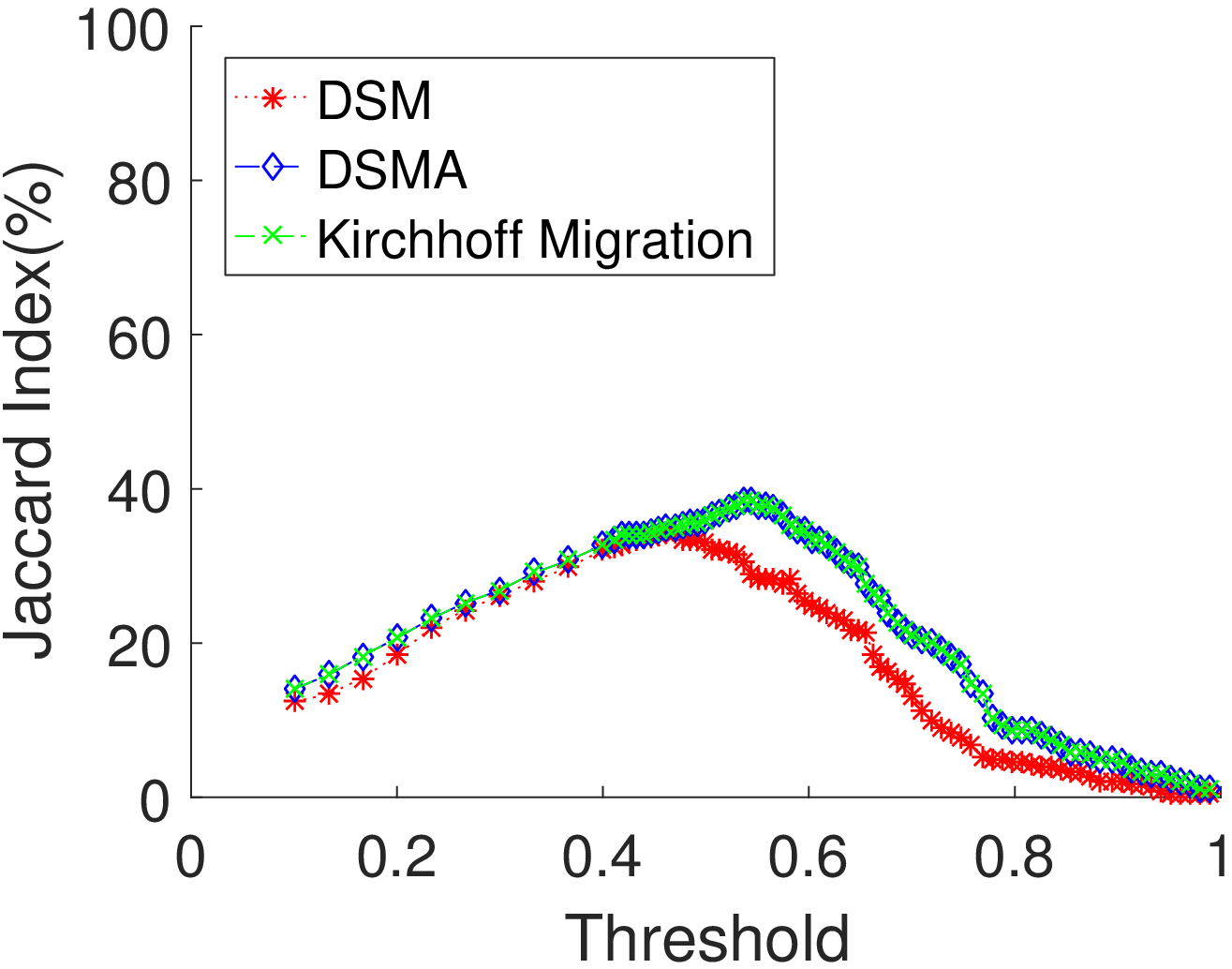}}
	
	\subfigure[Map of $\mathcal{I}_{\mathrm{DSM}}(\fz)$ for $L=12$]{\label{Result5-7}\centering\includegraphics[width=0.325\textwidth]{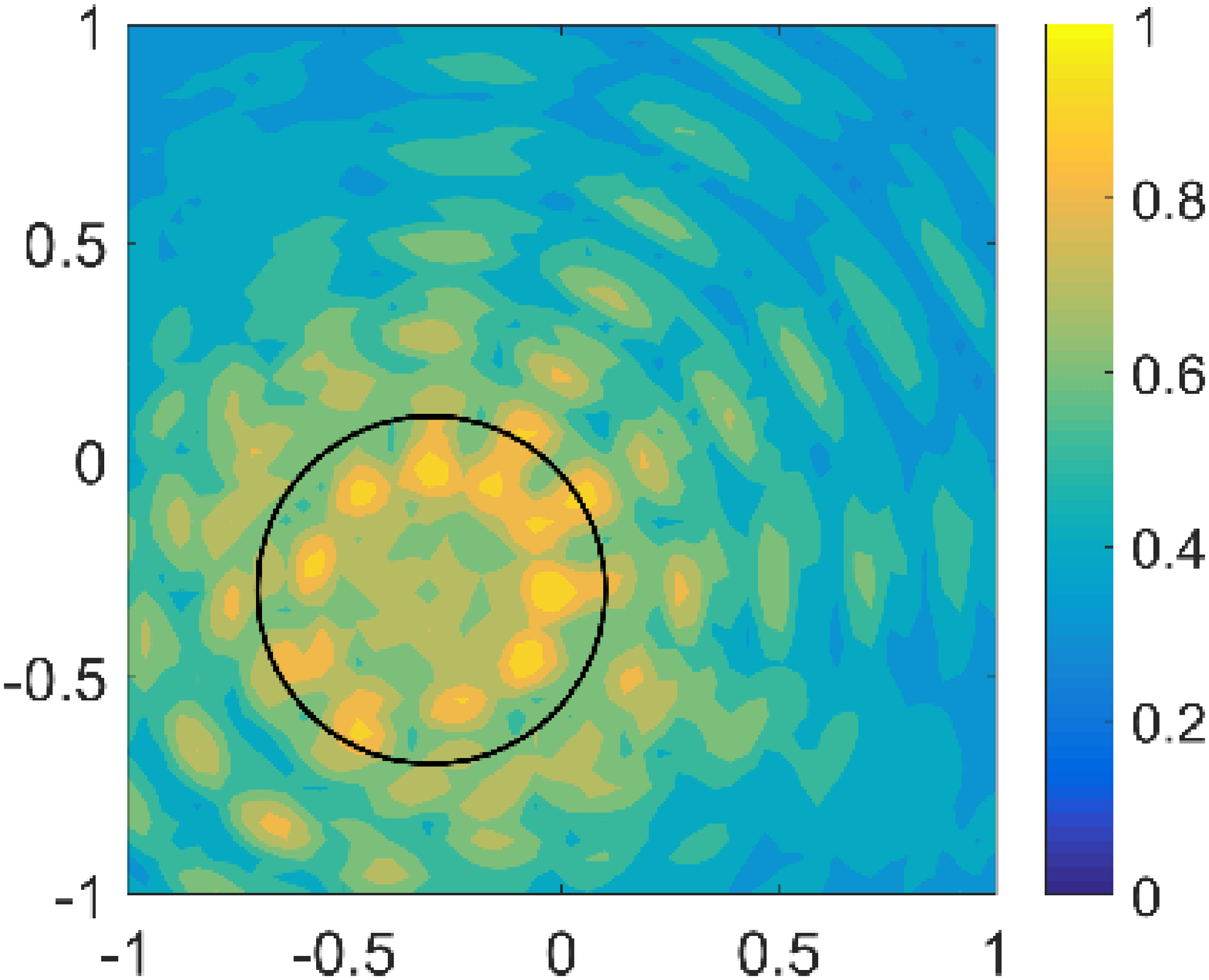}}
	\subfigure[Map of $\mathcal{I}_{\mathrm{DSMA}}(\fz)$ for $L=12$]{\label{Result5-8}\centering\includegraphics[width=0.325\textwidth]{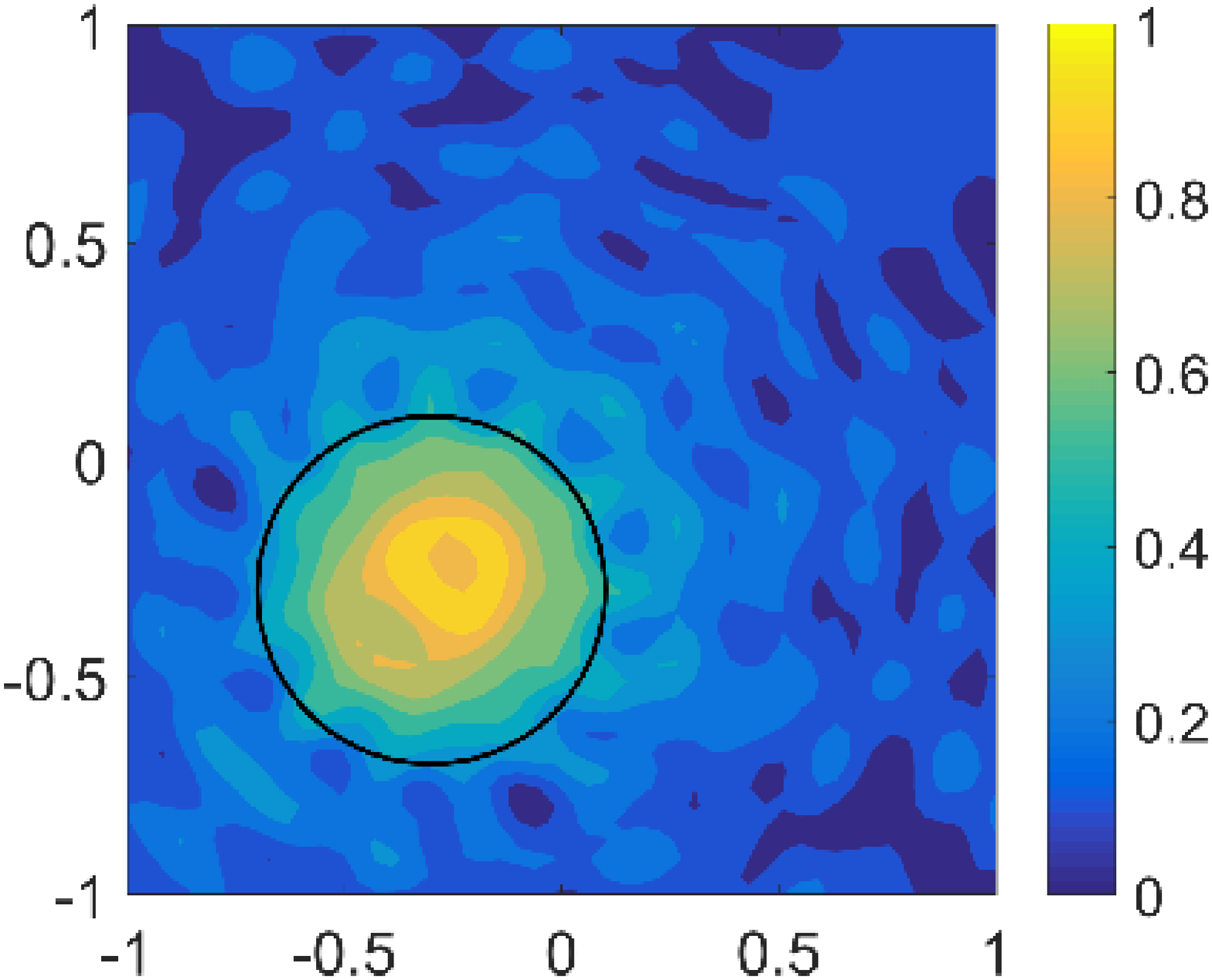}}
	\subfigure[Jaccard index for $L=12$]{\label{Result5-9}\centering\includegraphics[width=0.325\textwidth]{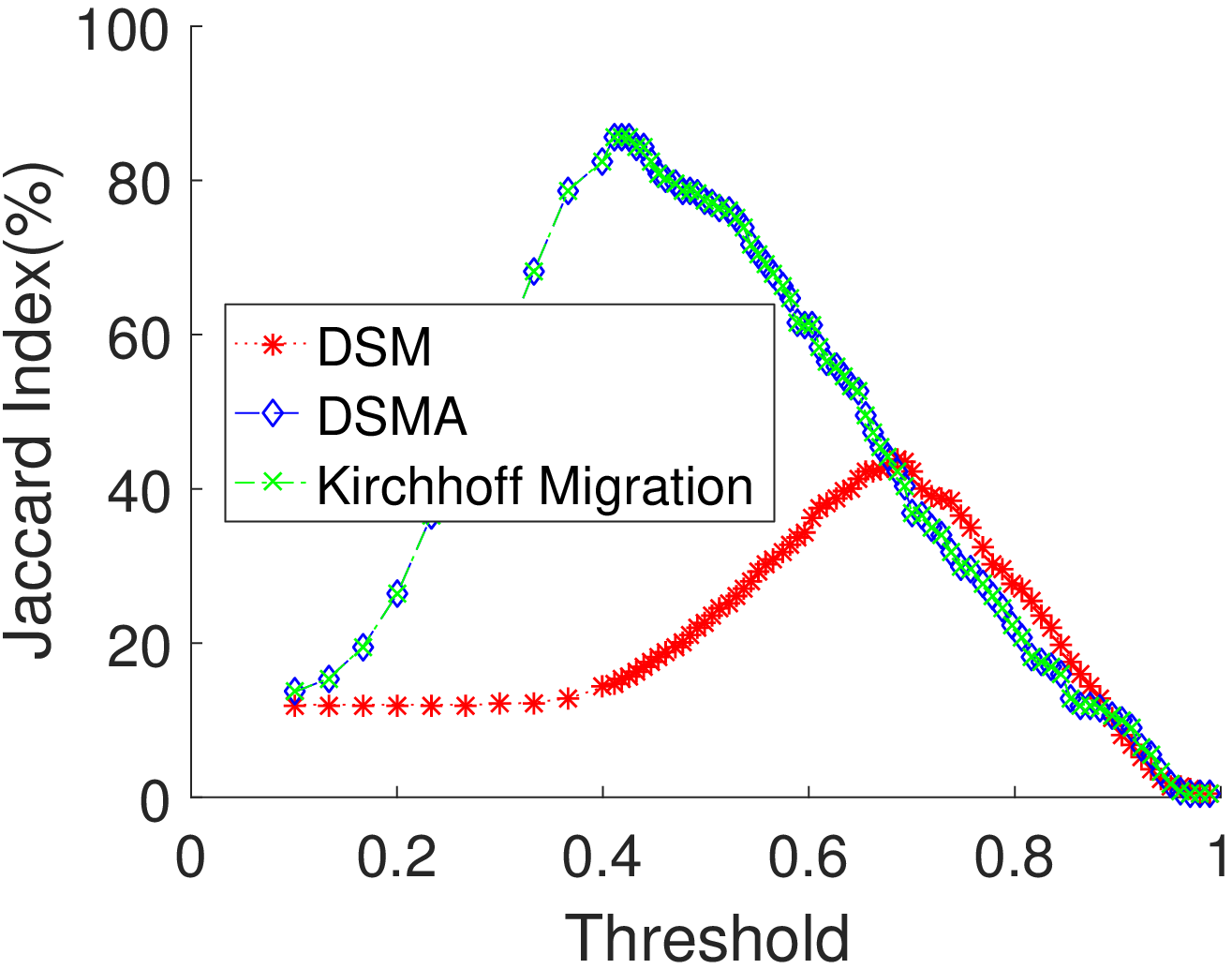}}
	
	\subfigure[Map of $\mathcal{I}_{\mathrm{DSM}}(\fz)$ for $L=36$]{\label{Result5-10}\centering\includegraphics[width=0.325\textwidth]{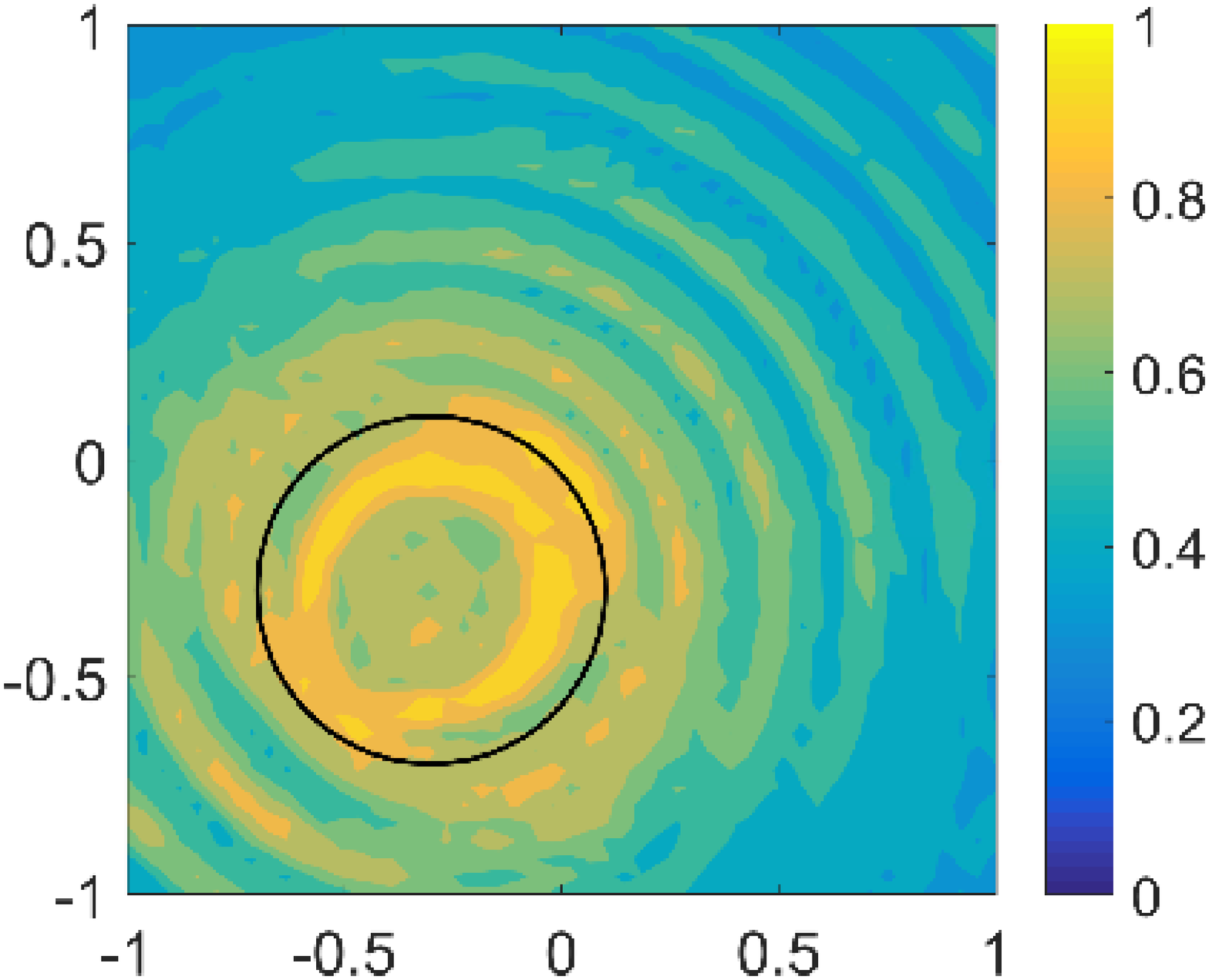}}
	\subfigure[Map of $\mathcal{I}_{\mathrm{DSMA}}(\fz)$ for $L=36$]{\label{Result5-11}\centering\includegraphics[width=0.325\textwidth]{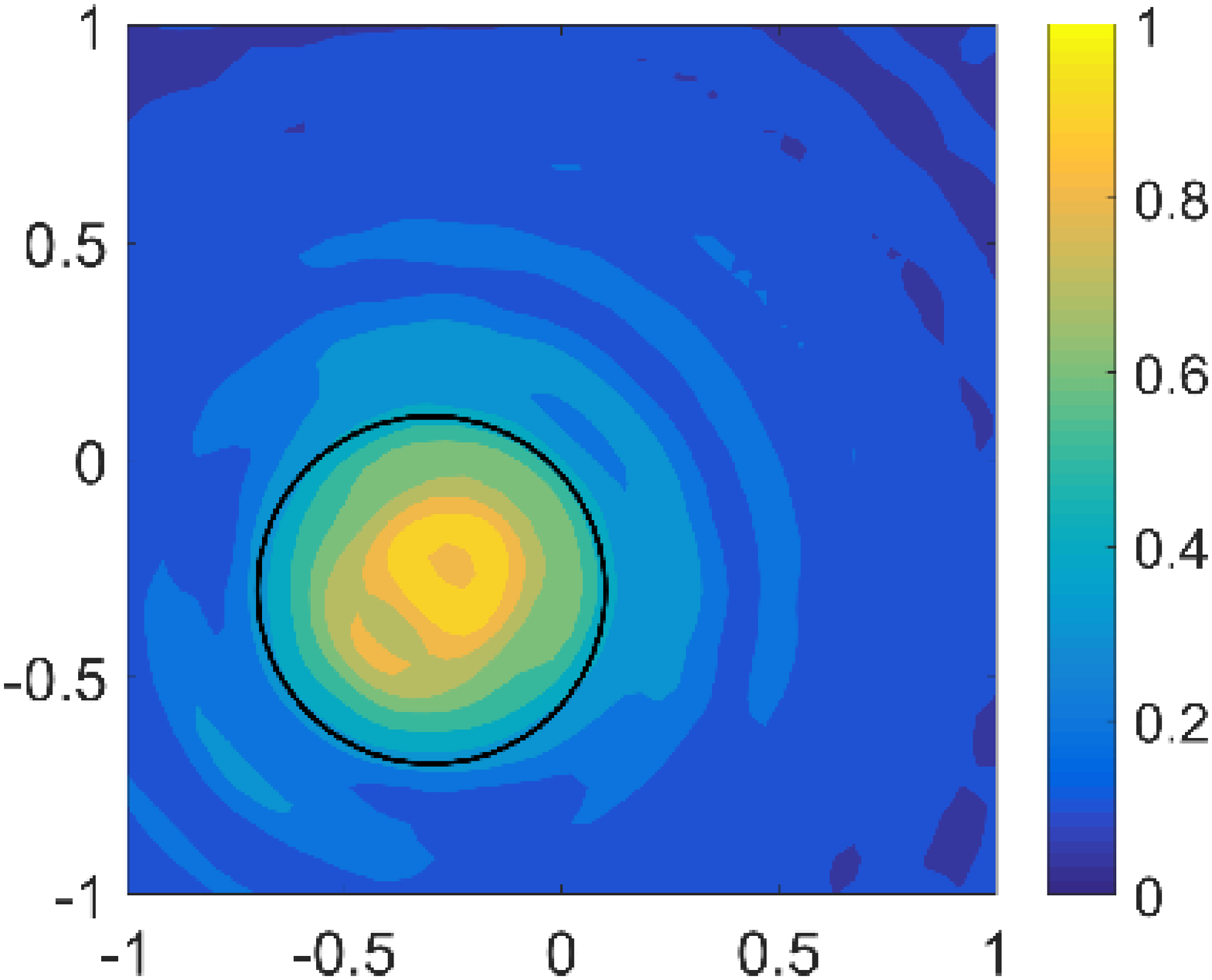}}
	\subfigure[Jaccard index for $L=36$]{\label{Result5-12}\centering\includegraphics[width=0.325\textwidth]{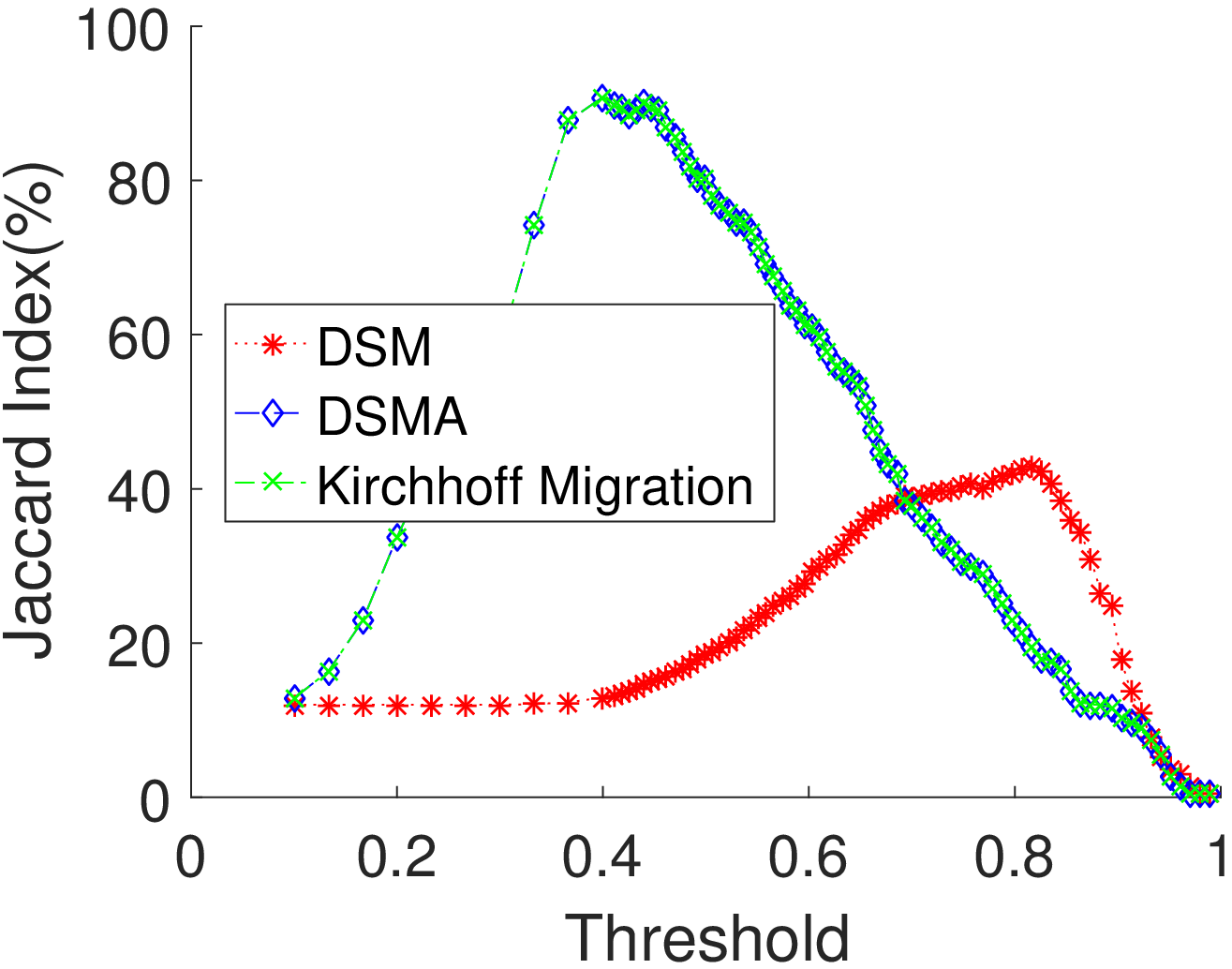}}
		
	\caption{\label{Result5}(Example~\ref{Example4}) Map of $\mathcal{I}_{\mathrm{DSM}}(\fz)$ (left column) $\mathcal{I}_{\mathrm{DSMA}}(\fz)$ (center column), and Jaccard index (right column).}
\end{figure}

\section{Illustration with some experimental data}\label{sec:ExperimentalData}
In the following both approaches are applied to  experimental data available at \url{http://www.fresnel.fr/3Ddatabase/database.php}  and described in \cite{0266-5611-17-6-301}. In order to be as close as possible of the framework to the theoretical results the frequency is chosen as $f=\SI{2}{\GHz}$, which corresponds to a wavelength $\lambda=\SI{0.1499}{\m}$. The chosen configuration is the one with two dielectric cylinders of radius $\alpha_m\approx\lambda/10, m=1, 2$ and of permittivity of $\varepsilon_m=\left(3.0\pm 0.3\right)\varepsilon_0$ with $r_1\approx \left(-2\lambda/30,-3\lambda/10~\si{\m}\right)$ and $r_2\approx \left(\SI{0}{\m},3\lambda/10~\si{\m}\right)$ (experimental file name is \textit{twodielTM 8f.exp}). The measurement configuration is as follows, $L=36$ sources are at a distance of $d_s\approx 4.80 \lambda$ evenly distributed from $\SI{10}{\degree}$ to $\SI{350}{\degree}$ and  $N=49$ receivers are placed at $d_l \approx 5.07\lambda$ and evenly distributed from $\SI{5}{\degree}$ to $\SI{355}{\degree}$. $\Omega_{\Gamma}$ is a square area of $3 \lambda \times 3 \lambda$ and has been discretized in $50\times 50$ pixels. It is worth to note that
\begin{itemize}
\item due to experimental set-up limitations the full Multi-Static Response is not available;
\item the incident field is no longer within the far field approximation and cannot be approximated by a plane wave.
\end{itemize}

The results are displayed in figure~\ref{FresnelTwoCylinder2GHz}. For the case of a single source (figure~\ref{FresnelTwoCylinder2GHz}, first line) neither DSM nor  DSMA provides a good localization  of the defect even though the two maps are almost identical. Some discrepancies can be seen between the two, thanks to the Jaccard index comparison. They are related to the fact that, as already mentioned, the incident field cannot be approximated by a plane wave. When the number of incident fields $L$ is increasing (figure~\ref{FresnelTwoCylinder2GHz}, second and third  line) the improvement provided by $\mathcal{I}_{\mathrm{DSMA}}$ compared to $\mathcal{I}_{\mathrm{DSM}}$ is still valid up to a threshold of $\kappa = \SI{80}{\percent}$.
\begin{figure}[h]
	\centering
		\subfigure[DSM with $L=1$]{\centering\includegraphics[width=0.32\textwidth]{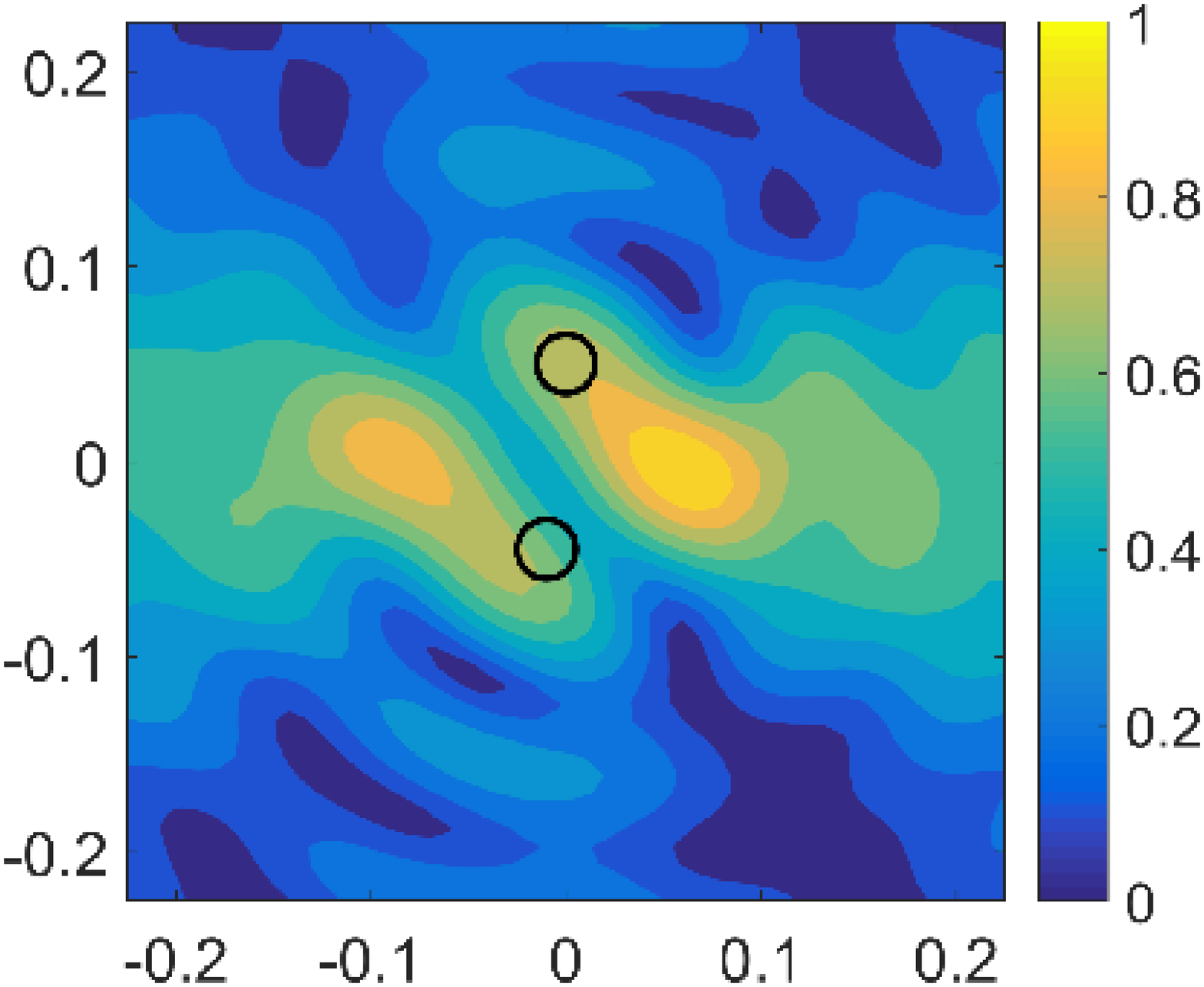}}
		\subfigure[DSMA with $L=1$]{\centering\includegraphics[width=0.32\textwidth]{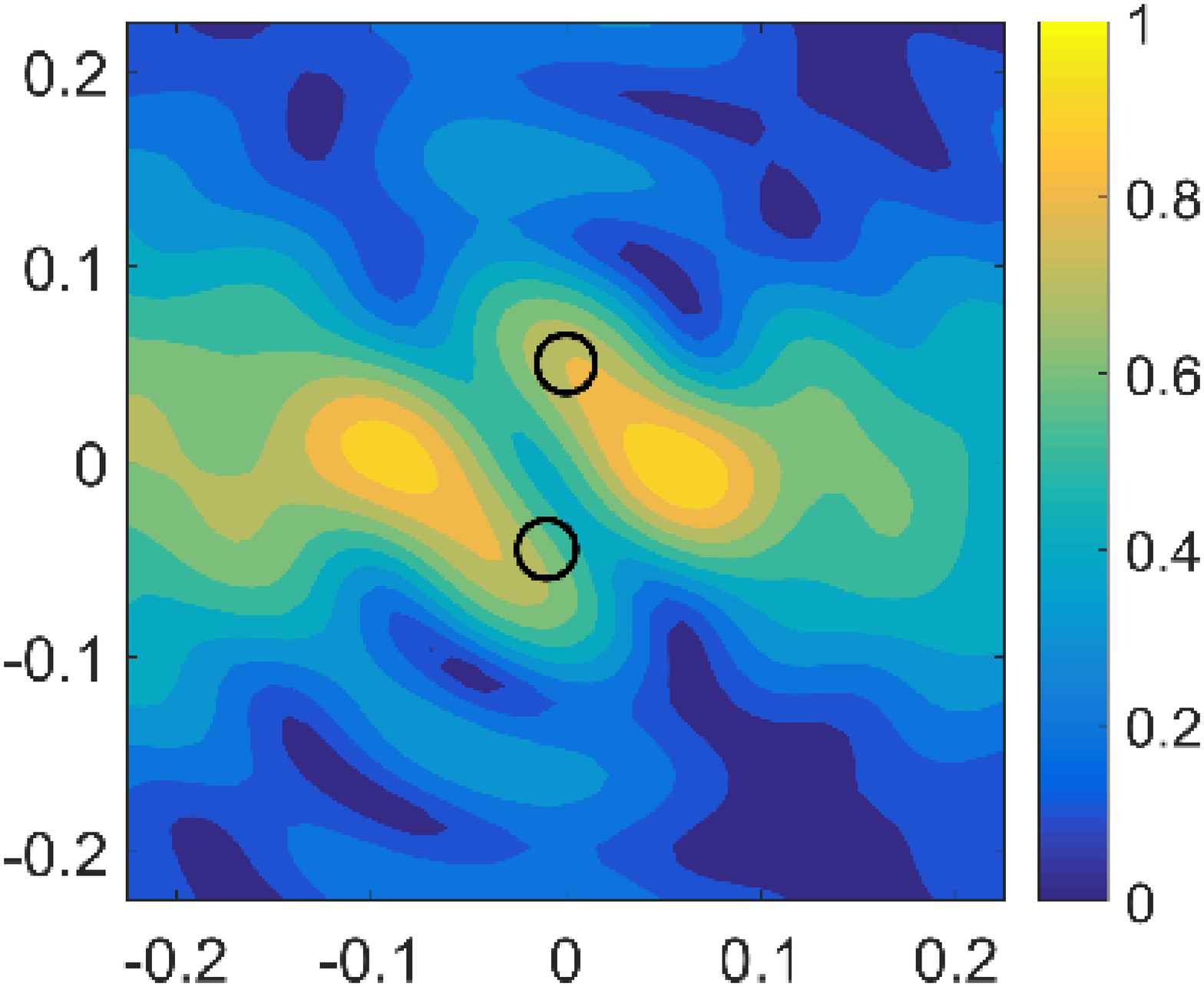}}
		\subfigure[Jaccard Index with $L=1$]{\centering\includegraphics[width=0.32\textwidth]{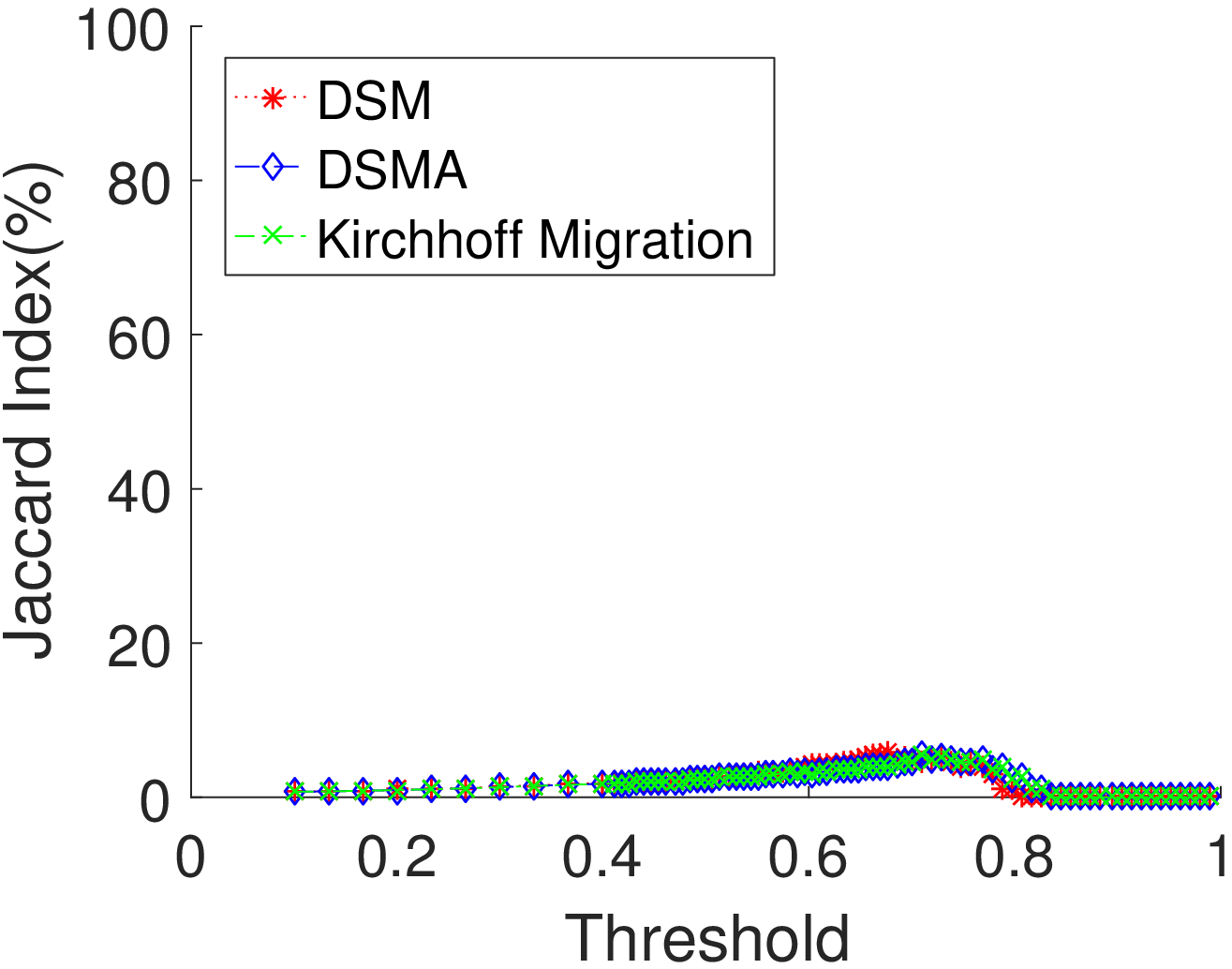}}
		
		\subfigure[DSM with $L=12$]{\centering\includegraphics[width=0.32\textwidth]{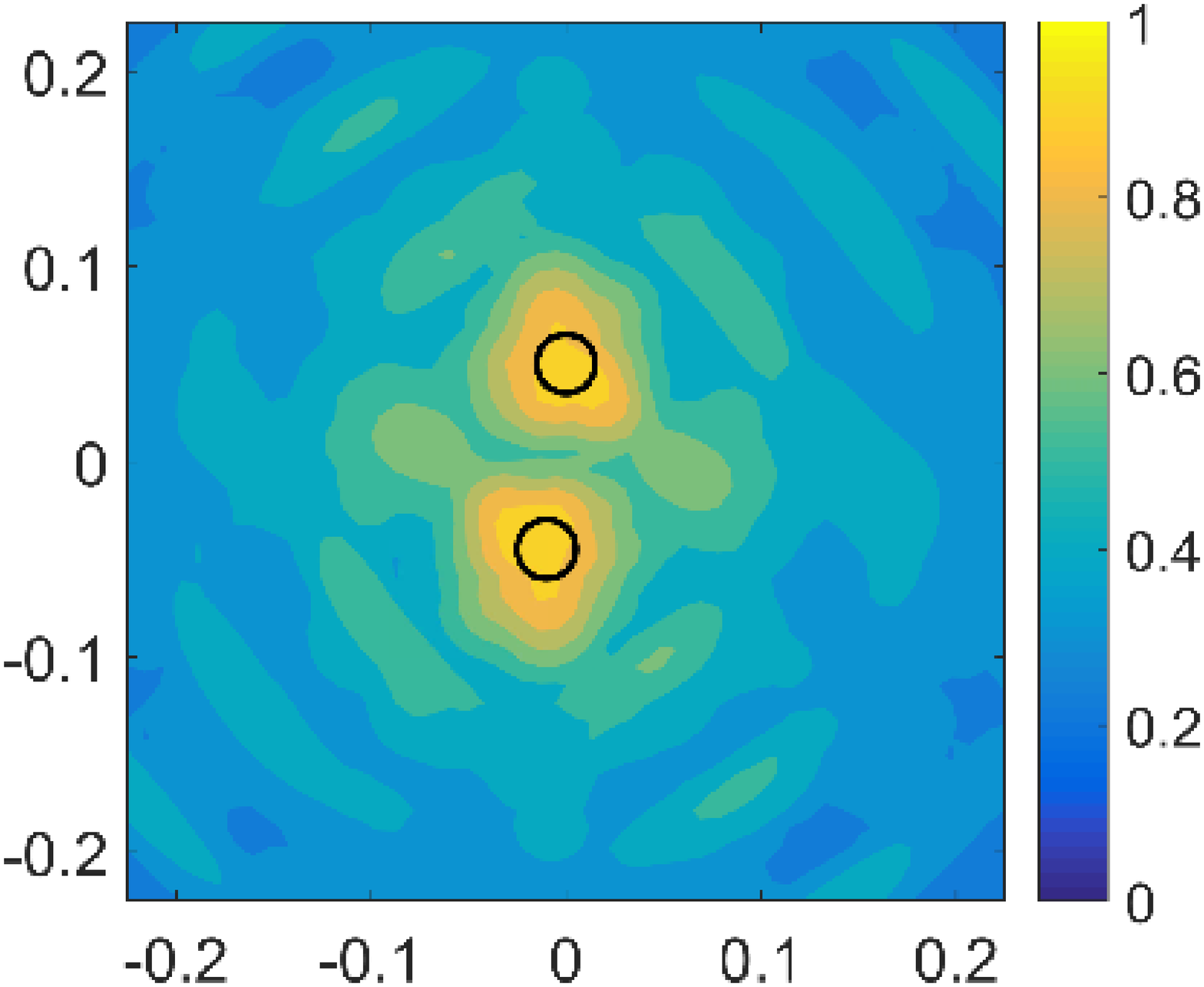}}
		\subfigure[DSMA with $L=12$]{\centering\includegraphics[width=0.32\textwidth]{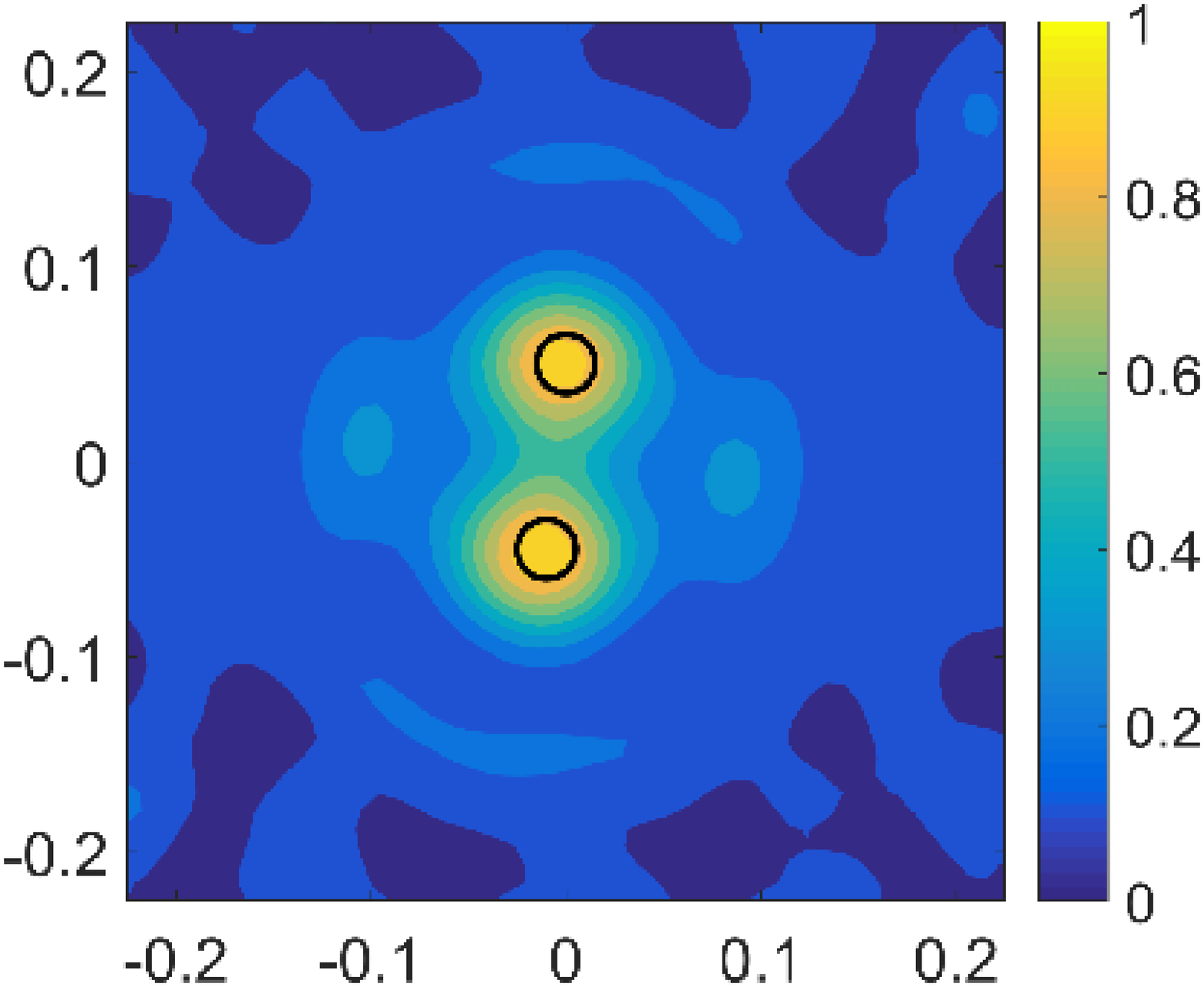}}
		\subfigure[Jaccard Index with $L=12$]{\centering\includegraphics[width=0.32\textwidth]{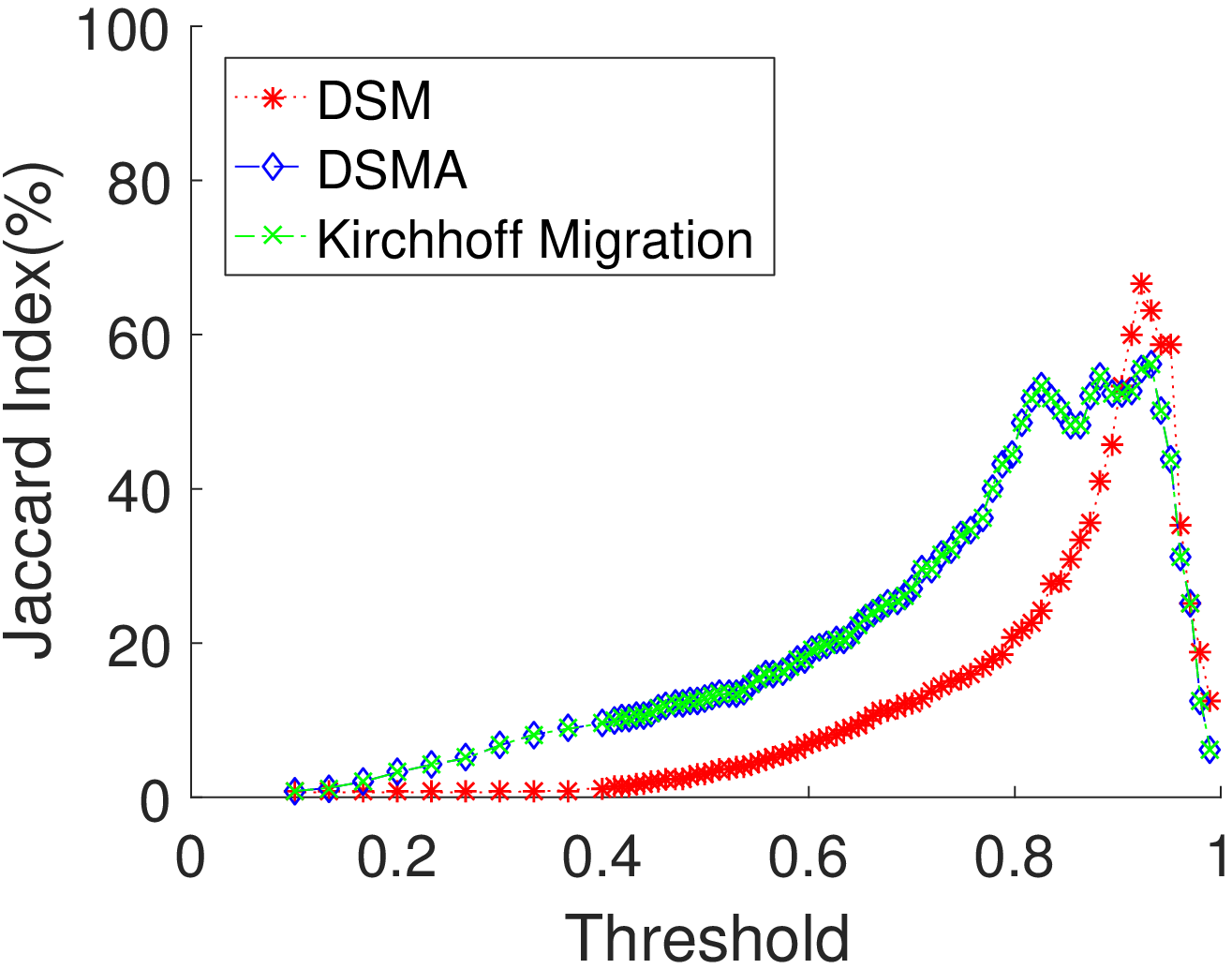}}
		
		\subfigure[DSM with $L=36$]{\centering\includegraphics[width=0.32\textwidth]{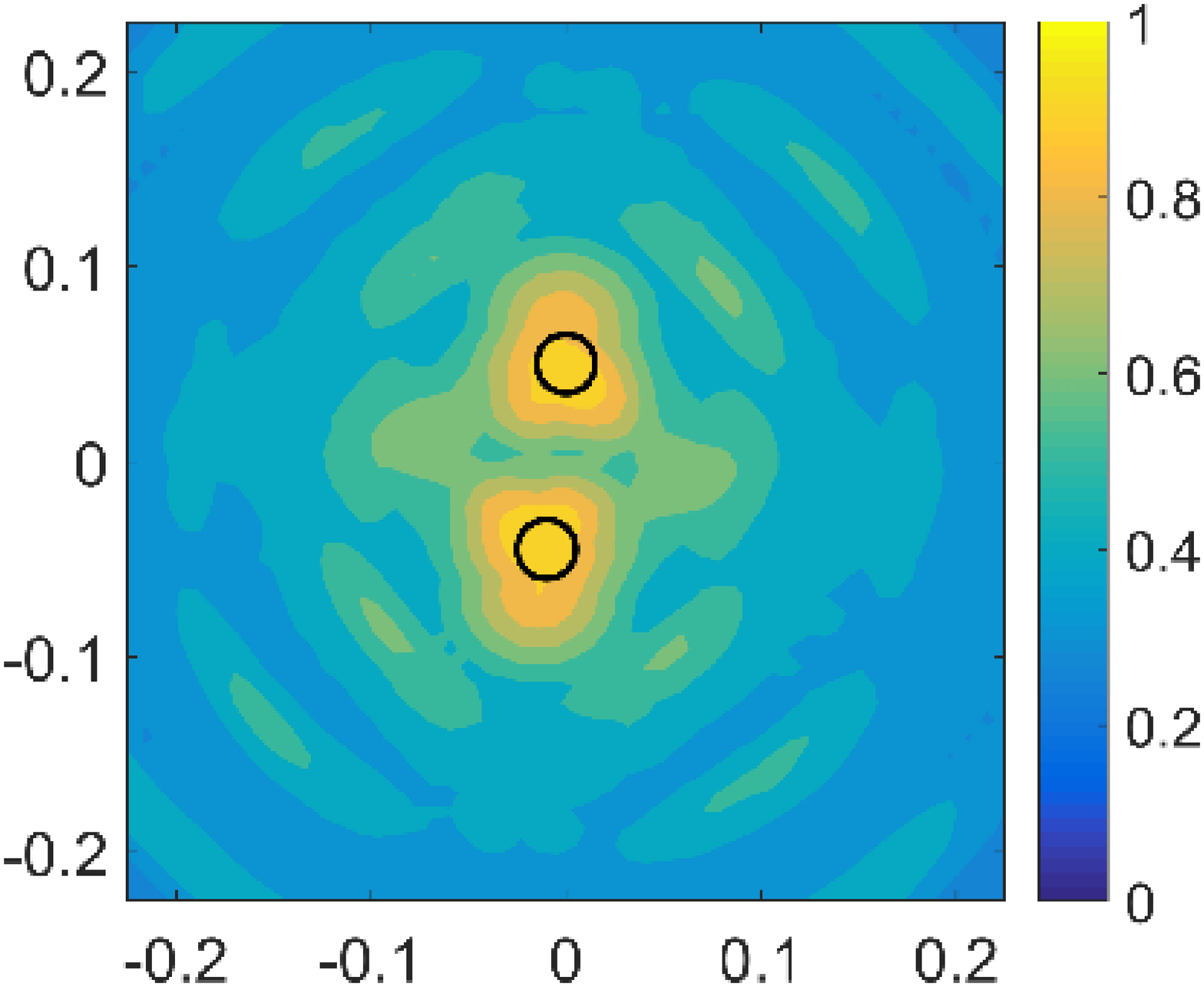}}
		\subfigure[DSMA with $L=36$]{\centering\includegraphics[width=0.32\textwidth]{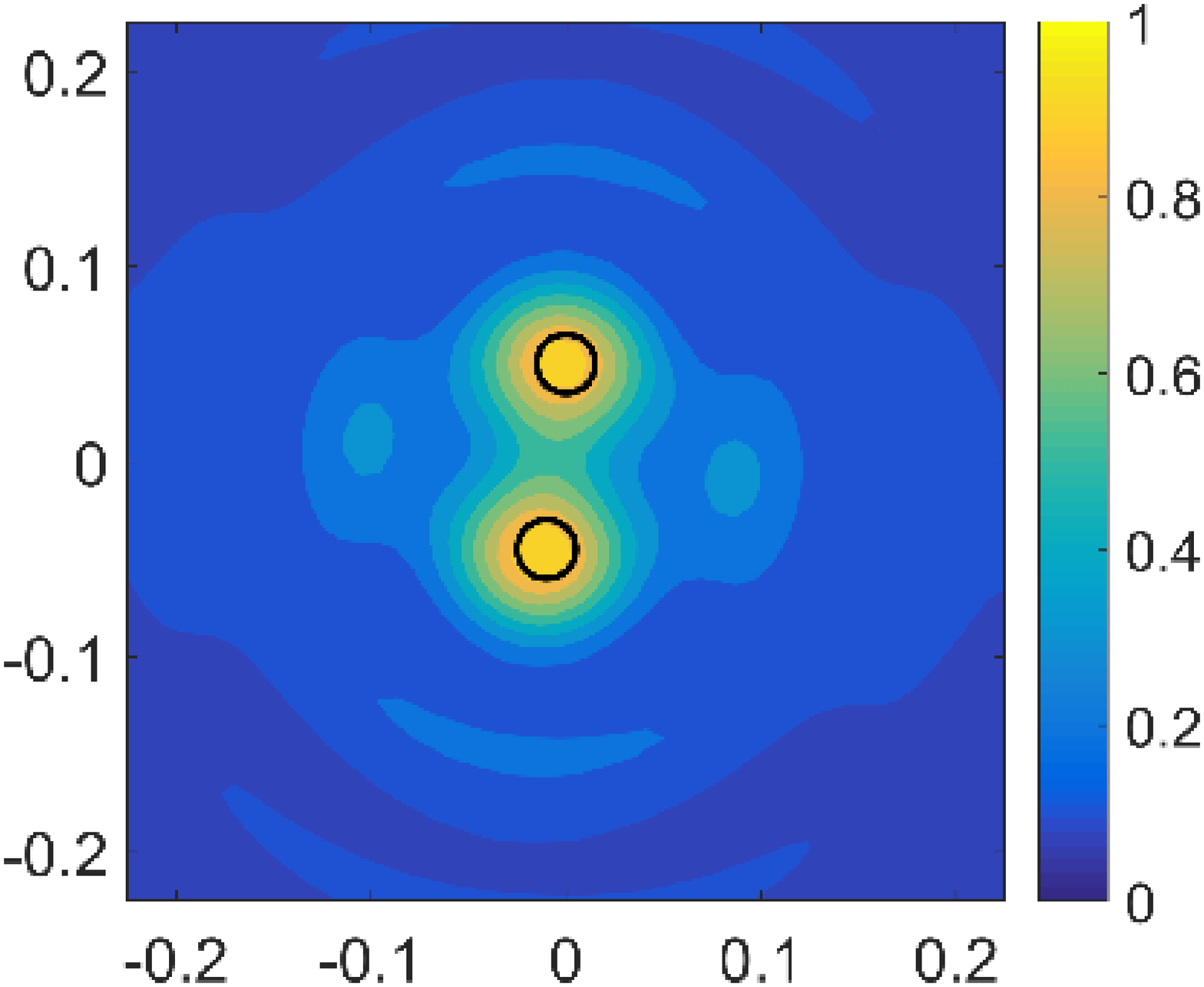}}
		\subfigure[Jaccard Index with $L=36$]{\centering\includegraphics[width=0.32\textwidth]{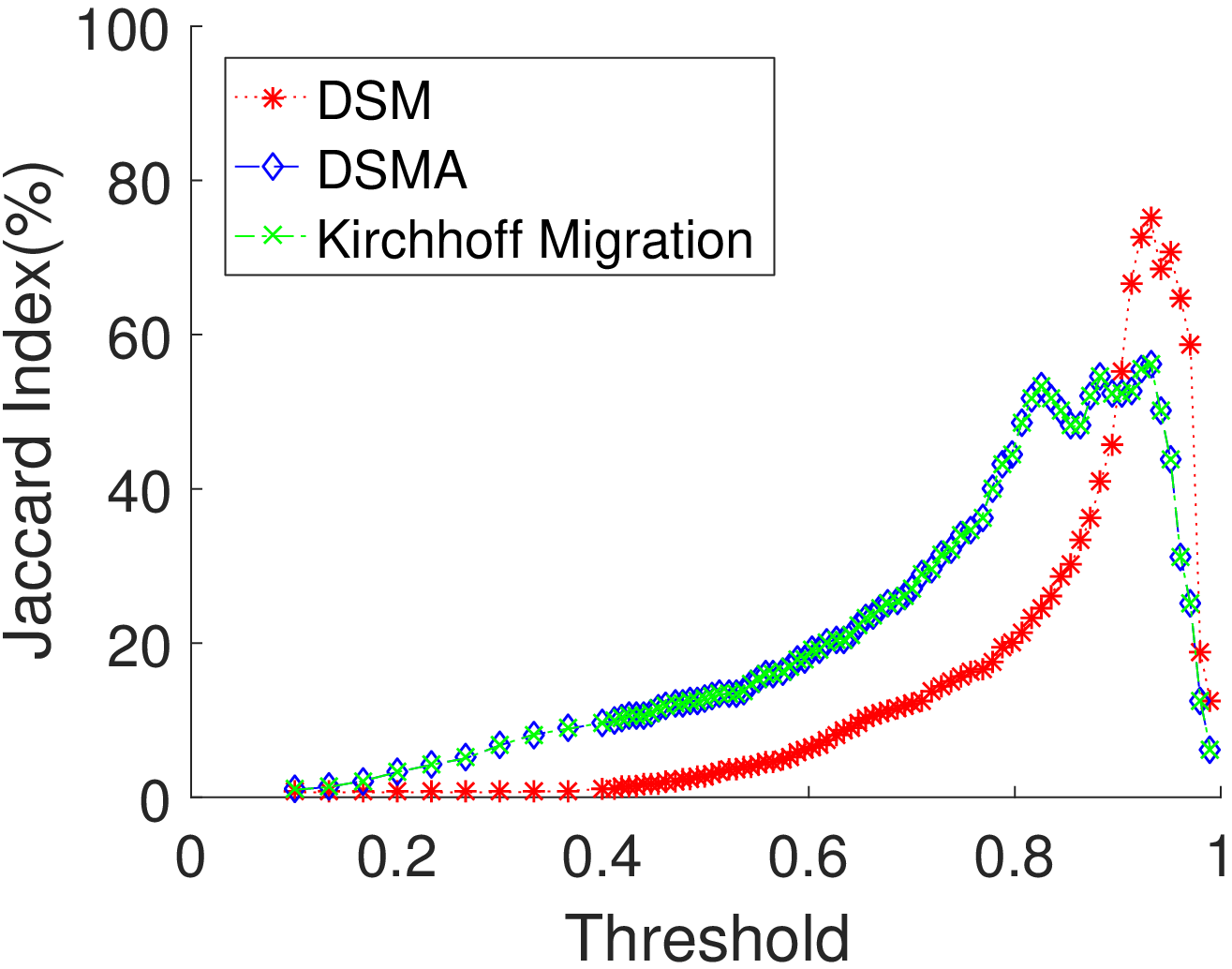}}
		\caption{\label{FresnelTwoCylinder2GHz} Map of $\mathcal{I}_{\mathrm{DSM}}(\fz)$ (left column) $\mathcal{I}_{\mathrm{DSMA}}(\fz)$ (center column), and Jaccard index (right column) using the experimental data at \SI{2}{\GHz}.}%
\end{figure}

\section{Conclusion}\label{sec:7}
In this contribution, the direct sampling method (DSM) is analyzed in the case of small obstacles thanks to the asymptotic formula of the scattered field. Some drawbacks of the classical DSM are exhibited and a alternative DSM which improved the performance of the former  in the case of multiple transmitters is proposed. Once the DSMA indicator function has been derived a strong connection between Kirchhoff migration and traditional and alternative DSM has been identified. Numerical simulations under various conditions are provided  to support our theoretical results either with synthetic or experimental data.

It would be interesting to investigate the mathematical structure and the various properties of the DSM indicator function in a limited-view configuration. Finally, we expect that the result in this contribution could be extended to three-dimensional inverse scattering problems.

\section*{Acknowledgement}
The authors are very grateful to Dominique Lesselier for his valuable advice. Part of this work was done while W.-K. Park was visiting G{\'e}nie {\'e}lectrique et {\'e}lectronique de Paris (GeePs), CentraleSup{\'e}lec, Universit{\'e} Paris-Sud. W.-K. Park was supported by the Basic Science Research Program through the National Research Foundation of Korea (NRF) funded by the Ministry of Education (grant no. NRF-2017R1D1A1A09000547).

\section*{References}
\bibliographystyle{unsrt}
\bibliography{Direct.sampling.method.for.imaging.small.dielectric.inhomogeneities.analysis.and.improvement-arxiv}
\end{document}